\documentclass{amsart}
\usepackage{amsmath}
\usepackage{paralist}
\usepackage{amsfonts}
\usepackage{amssymb}
\usepackage{amsthm,amsmath}
\usepackage{amscd}
\usepackage{float}
\usepackage{tikz}
\usepackage{graphicx}
\usepackage[colorlinks=true]{hyperref}
\hypersetup{urlcolor=blue, citecolor=red}
\usepackage{hyperref}
\usepackage{soul}

\newtheorem{theorem}{Theorem}[section]
\newtheorem{corollary}[theorem]{Corollary}

\newtheorem{lemma}[theorem]{Lemma}
\newtheorem{proposition}[theorem]{Proposition}

\theoremstyle{definition}

\newtheorem{remark}[theorem]{Remark}

\newcommand{\R}{\mathbb{R}}

\newcommand{\Z}{\mathbb{Z}}
\newcommand{\N}{\mathbb{N}}

\newcommand{\C}{\mathbb{C}}

\author[Akansha Sanwal]{Akansha Sanwal}

\address{Institut f\"ur Mathematik, Leopold--Franzens Universit\"at Innsbruck,
Technikerstrasse 13, 6020 Innsbruck, Austria}
\email{akansha.sanwal@uibk.ac.at}

\author{Robert Schippa}
\address{Korea Institute of Advanced Study, Hoegi-ro 85, Dongdaemun-gu, 02455 Seoul, Republic of Korea}
\email{rschippa@kias.re.kr}

\begin{document}
		\title[Fractional KP-I equations]{Low regularity well-posedness for KP-I equations: the dispersion-generalized case}
		\maketitle

		\begin{abstract}
			 We prove new well-posedness results for dispersion-generalized Kadomtsev--Petviashvili I equations in $\R^2$, which family links the classical KP-I equation with the fifth order KP-I equation. For strong enough dispersion, we show global well-posedness in $L^2(\R^2)$. To this end, we combine resonance and transversality considerations with Strichartz estimates and a nonlinear Loomis--Whitney inequality. Moreover, we prove that for small dispersion, the equations cannot be solved via Picard iteration. In this case, we use an additional frequency dependent time localization.
		\end{abstract}
	
\section{Introduction and main results}	
We consider the Cauchy problem for the fractional Kadomtsev--Petviashvili I (fKP-I) equation
		\begin{equation}
			\label{eq:fKPI}
			\left\{ \begin{array}{cl}
				\partial_t u - D^{\alpha}_x\partial_{x} u - \partial_{x}^{-1} \partial_y^2 u &= u \partial_x u, \quad (t,x,y) \in \R\times \R \times \R, \\ 
				u(0) &= u_0 \in H^{s_1,s_2}(\R^2),
			\end{array} \right.
		\end{equation}
	where $2<\alpha<4$, and the operator $D_x^{\alpha}$ is given by $(D_x^{\alpha}f)^{\wedge}(\xi) = |\xi|^{\alpha}\hat{f}(\xi)$. For $2 < \alpha \leqslant \frac{5}{2}$, we only consider real-valued solutions; for $\alpha > \frac{5}{2}$ we also treat complex-valued solutions. Note that the solution stays real-valued provided that the initial data is real-valued. In this paper, we consider initial data from anisotropic Sobolev spaces $H^{s_1,s_2}(\R^2)$, which are defined by
	\begin{equation*}
	\begin{split}
		H^{s_1,s_2}(\R^2)&:=\{\phi \in L^2(\R^2): \|\phi\|_{H^{s_1,s_2}(\R^2)} < \infty \}, \\
		 \|\phi\|_{H^{s_1,s_2}(\R^2)}  &= \| \hat{\phi}(\xi,\eta)(1+|\xi|^2)^{\frac{s_1}{2}} (1+|\eta|^2)^{\frac{s_2}{2}} \|_{L_{\xi,\eta}^2}.
	\end{split}
	\end{equation*}
The following quantities are conserved for real-valued solutions:
\begin{align}
\label{eq:MassConservation}
		M(u)(t) &= \int_{\R^2} u(x,y)^2 dxdy, \\
\label{eq:EnergyConservation}
	E_{\alpha}(u)(t) &=\int_{\R^2} \Big ( \frac{1}{2} |D_x^{\frac{\alpha}{2}} u|^2 + \frac{1}{2} |\partial_x^{-1}\partial_y u|^2 + \frac{1}{6}u^3\Big)dxdy.
\end{align}
Hence, the natural energy space is given by
\begin{equation*}
\mathbf{E}^{\alpha}(\R^2)= \{\phi \in L^2(\R^2): \|\phi\|_{E^{\alpha}(\R^2)}:= \|p(\xi,\eta)\hat{\phi}(\xi,\eta)\|_{L^2_{\xi,\eta}} <\infty \},
\end{equation*}
where 
\begin{equation*}
	p(\xi,\eta):= 1+ |\xi|^{\frac{\alpha}{2}} +\frac{|\eta|}{|\xi|}.
\end{equation*}	
We prefer to study the solutions in the scale of anisotropic Sobolev spaces. We believe that adapting the present analysis will yield global well-posedness in the energy space, which is a smaller space, as well. Here we focus on the much larger anisotropic Sobolev spaces. For further remarks on the connection between Sobolev spaces and the energy spaces, we refer to \cite{Tom1996}.

Moreover, if $u$ solves the problem \eqref{eq:fKPI} with initial data $\phi$, then $u_{\lambda}$ given by
\begin{equation*}
	u_{\lambda}(t,x,y) = \lambda^{-\alpha} u(\lambda^{-(\alpha+1)}t, \lambda^{-1} x, \lambda^{-\frac{\alpha+2}{2}} y )
\end{equation*}
also solves the same with scaled initial data
\begin{equation}\label{eq:ScalingfKPInitialData}
	\phi_{\lambda} = \lambda^{-\alpha}\phi(\lambda^{-1} x, \lambda^{-\frac{\alpha+2}{2}} y).
	\end{equation}
We have
\begin{equation}\label{eq:SubcriticalInitialDataFKP}
	\|\phi_{\lambda}\|_{\dot{H}^{s_1,s_2}(\R^2)} = \lambda^{-\frac{3\alpha}{4}+1-s_1-(\frac{\alpha}{2}+1)s_2} \|\phi\|_{\dot{H}^{s_1,s_2}(\R^2)}.
\end{equation}
This shows that for $\alpha=\frac{4}{3}$, \eqref{eq:fKPI} is $L^2$-critical. For $\alpha > 2$, which is considered presently, \eqref{eq:fKPI} is $L^2$-subcritical.

 By local well-posedness, we refer to existence, uniqueness, and continuity of the data-to-solution mapping locally in time.
	
The range of dispersion considered in this paper starts with the classical KP-I equation
\begin{equation}
\label{eq:KPI}
\left\{ \begin{array}{cl}
\partial_t u - D^{2}_x\partial_{x} u - \partial_{x}^{-1} \partial_y^2 u &= u \partial_x u, \quad (t,x,y) \in \R\times \R \times \R, \\ 
				u(0) &= u_0 \in H^{s_1,s_2}(\R^2),
\end{array} \right.
\end{equation}	
which has been extensively studied (see \cite{MolinetSautTzvetkov2002,IKT,Guoetal} and references therein). Ionescu--Kenig--Tataru \cite{IKT} proved global well-posedness in the energy space, and Z. Guo \emph{et al.} \cite{Guoetal} showed improved local well-posedness in the anisotropic Sobolev space $H^{1,0}(\R^2)$. The derivative loss in case of unfavorable resonance  makes the equation \emph{quasilinear}. This means it is not amenable to Picard iteration in standard Sobolev spaces as observed by Molinet--Saut--Tzvetkov \cite{MolinetSautTzvetkov2002}. In the works \cite{IKT,Guoetal}, short-time Fourier restriction was used to overcome the derivative loss in the nonlinearity. We refer to the PhD thesis of the second author for an overview of short-time Fourier restriction \cite{Schippa2019PhDThesis}. Since short-time Fourier restriction also involves energy estimates, the results in \cite{IKT,Guoetal} require real-valued solutions. Likewise, the results we prove for small dispersion require real-valued initial data:
\begin{theorem}
\label{thm:QuasilinearLWP}
	Let $2<\alpha\leqslant \frac{5}{2}$. Then, \eqref{eq:fKPI} is locally well-posed in $H^{s,0}(\R^2)$ for $s>5-2\alpha$ and real-valued initial data.
	\end{theorem}
	We give a technically more detailed version of the above theorem in Section \ref{section:Quasilinear}.\\
	
	However, the data-to-solution mapping constructed in the proof of Theorem \ref{thm:QuasilinearLWP} is not analytic. Indeed, we show that for $\alpha<\frac{7}{3}$, the data-to-solution mapping cannot be of class $C^2$. Previously, Molinet--Saut--Tzvetkov \cite{MolinetSautTzvetkov2002} showed that the data-to-solution mapping cannot be $C^2$ for the KP-I equation (see also \cite{KochTzvetkov2008}). This result was generalized by Linares--Pilod--Saut \cite{LPS} for $\alpha < 2$. It turns out that the argument extends to $\alpha < \frac{7}{3}$:
\begin{theorem}\label{thm:IllPosed}
	Let $\alpha <\frac{7}{3}$,  $(s_1,s_2) \in \R^2$. Then, there exists no $T>0$ such that there is a function space $X_T \hookrightarrow C([-T,T];H^{s_1,s_2}(\R^2))$, in which \eqref{eq:fKPI} admits a unique local solution such that the flow-map for \eqref{eq:fKPI} given by
	\begin{equation*}
		\Gamma_t: u_0 \mapsto u(t), ~~~ t\in [-T,T],
	\end{equation*}
	is $C^2$-differentiable at zero from $H^{s_1,s_2}(\R^2)$ to $H^{s_1,s_2}(\R^2)$. 
\end{theorem}
The problematic nonlinear interaction is a resonant $High \times Low$-interaction in which a free solution with high $x$ frequencies interacts with a solution at low $x$ frequencies. With the dispersion relation for the fractional KP-I equation given by
\begin{equation}
	\label{eq:DispersionRelation}
	\omega_\alpha(\xi,\eta) = |\xi|^{\alpha}\xi + \frac{\eta^2}{\xi}, \quad \xi \in \R \backslash \{0\}, \; \eta \in \R,
\end{equation}
we find the resonance function to be
	\begin{equation*}
		\Omega_{\alpha}(\xi_1,\eta_1,\xi_2,\eta_2) = |\xi_1+\xi_2|^{\alpha}(\xi_1+\xi_2)-|\xi_1|^{\alpha}\xi_1 -|\xi_2|^{\alpha}\xi_2 -\frac{(\eta_1\xi_2 - \eta_2\xi_1)^2}{\xi_1 \xi_2 (\xi_1+\xi_2) }.
	\end{equation*}
Due to opposite signs of the terms $|\xi_1+\xi_2|^{\alpha}(\xi_1+\xi_2)-|\xi_1|^{\alpha}\xi_1 -|\xi_2|^{\alpha}\xi_2$ and $\frac{(\eta_1\xi_2 - \eta_2\xi_1)^2}{\xi_1 \xi_2 (\xi_1+\xi_2) }$, the resonance function can become much smaller than the first term, which we refer to as resonant case.
 However, we shall see that in the resonant case, we can argue that the interaction between the two nonlinear waves and the dual factor with low modulation is strongly transverse, which we quantify via a nonlinear Loomis--Whitney inequality. This transversality was already observed in \cite{IKT}, while in the proof in \cite{IKT} this is not related to nonlinear Loomis--Whitney. We believe that pointing out the connection with nonlinear Loomis--Whitney inequalities makes the proof more systematic.
 
 Nonlinear Loomis--Whitney inequalities were first investigated by Bennett--Carbe\-ry--Wright \cite{BennettCarberyWright2005} and quantitative versions suitable for application to PDEs were proved by Bejenaru--Herr--Tataru \cite{BejenaruHerrHolmerTataru2009,BejenaruHerrTataru2010}. These were all local though. We use a global version to simplify the argument, which is a result of Kinoshita and the second author \cite{KS}. We also refer to references in \cite{KS} for further discussion of nonlinear Loomis--Whitney inequalities.
 
The crucial ingredient in the resonant case of low modulation is to use the nonlinear Loomis--Whitney inequality to show a genuinely trilinear estimate, which improves on the bilinear estimate. Let $f_i \in L^2(\R^3;\R_{+})$ denote functions dyadically localized in spatial frequency in the $x$ direction around $N_i \in 2^{\Z}$ and in modulation $L_i \in 2^{\mathbb{N}_0}$ with $N_1 \sim N_2 \geqslant N_3$. With notations explained below, we have $\text{supp}(f_i) \subseteq \tilde{D}_{N_i,L_i}$. Moreover, let $L_i \leqslant N_1^\alpha N_2$. Then, we show the estimate
\begin{equation*}
\int (f_1 * f_2) \cdot f_3 \lesssim N_1^{-\frac{3 \alpha}{4} + \frac{1}{2}} N_3^{-\frac{1}{2}} \prod_{i=1}^3 L_i^{\frac{1}{2}} \| f_i \|_{L^2} .
\end{equation*}
Clearly, for $\alpha \geqslant 2$ and $N_3 \gtrsim 1$, this ameliorates the derivative loss. The observation is that for $N_3 \gtrsim N_1^{-\kappa}$ for some $\kappa > 0$, this estimate still suffices to overcome the derivative loss, whereas for $N_3 \lesssim N_1^{-\kappa}$, the bilinear Strichartz estimate gains additional powers of $N_1^{-1}$. The bilinear Strichartz estimate is another consequence of transversality in case of resonance. It reads for free solutions in the resonant case with $N_1 \gg N_2$:
\begin{equation*}
\| P_{N_1} U_{\alpha}(t) u_0 P_{N_2}U_{\alpha}(t)v_0 \|_{L^2_{t,x}} \lesssim \frac{N_2^{\frac{1}{2}}} {N_1^{\frac{\alpha}{4}}} \|u_0\|_{L^2_x} \|v_0\|_{L^2_x}
\end{equation*}
with $(U_\alpha(t) f)^{\wedge}(\xi,\eta) = e^{i t \omega_\alpha(\xi,\eta) } \hat{f}(\xi,\eta)$.
 
By combining the nonlinear Loomis--Whitney inequality and the bilinear Stri\-chartz estimate, we note that the fractional KP-I equations are semilinearly well-posed for $\alpha>\frac{5}{2}$. In this range we solve the equations by applying the contraction mapping principle in suitable function spaces. This suggests the choice for frequency dependent time localization obtained by interpolating between $(\alpha, T(N)) = (2,N^{-1})$ and $(\alpha, T(N)) = \big(\frac{5}{2}+,1\big)$, which suggests $T(N) = N^{-(2\alpha-5)-\varepsilon}$. We shall choose $\varepsilon = \varepsilon(\alpha)$.
\begin{theorem}
\label{thm:SemilinearLWP}
	Let $\frac{5}{2}<\alpha <4$. Then, \eqref{eq:fKPI} is analytically locally well-posed in $H^{s,0}(\R^2)$ for $s>\frac{5}{4}-\frac{\alpha}{2}$. 
\end{theorem}
The analyticity of the data-to-solution mapping is a consequence of applying the contraction mapping principle and the analyticity of the nonlinearity.
By conservation of mass and persistence of regularity, we have the following:
\begin{corollary}
\label{thm:SemilinearGWP}
	Let $s \geqslant 0$, and $\frac{5}{2}<\alpha <4$. Then, \eqref{eq:fKPI} is globally well-posed in $H^{s,0}(\R^2)$ for real-valued initial data.
\end{corollary}

We remark that it was well-known that the fifth order KP-I equation
\begin{equation}
\label{eq:FifthKPI}
\left\{ \begin{array}{cl}
\partial_t u - D^{4}_x\partial_{x} u - \partial_{x}^{-1} \partial_y^2 u &= u \partial_x u, \quad (t,x,y) \in \R\times \R \times \R, \\ 
				u(0) &= u_0 \in H^{s_1,s_2}(\R^2)
\end{array} \right.
\end{equation}
can be solved via Picard iteration as pointed out by Saut--Tzvetkov \cite{SautTzvetkov1999,SautTzvetkov2000}. Their result was improved by B. Guo \emph{et al.} \cite{GuoHuoFang2017} using short-time Fourier restriction and Yan \emph{et al.} \cite{YanLiHuangDuan} (see also \cite{LX} for an earlier result) recovered the same local well-posedness result without using frequency dependent time localization.

In the limiting cases of $\alpha$ presently considered, we recover the currently best local well-posedness results in anisotropic Sobolev spaces. For $\alpha \downarrow 2$ we recover the result from \cite{Guoetal} and for $\alpha \uparrow 4$ we arrive at the result from \cite{YanLiHuangDuan}. We note that there is still a mismatch between the range of dispersion, for which we can show failure of Picard iteration and for which we actually use frequency-dependent time localization. It is unclear whether one has to improve the counterexample or the argument to show semilinear local well-posedness.

Moreover, in the companion paper \cite{HerrSanwalSchippa2022}, we consider the dispersion-generalized KP-I equation ($2<\alpha<4$) in three dimensions in non-periodic, periodic, and mixed settings.

\subsubsection{Organization}

In Section \ref{section:Notation}, we introduce the notation and function spaces. For the proof of Theorem \ref{thm:QuasilinearLWP}, we use short-time Fourier restriction spaces introduced by Ionescu--Kenig--Tataru \cite{IKT} and for the proof of Theorem \ref{thm:SemilinearLWP}, we use standard Fourier restriction spaces (cf. \cite{Bourgain1993}). We also recall linear Strichartz estimates. In Section \ref{section:Illposedness}, we show that the data-to-solution mapping fails to be $C^2$ for $\alpha< 7/3$ as stated in Theorem \ref{thm:IllPosed}. In Section \ref{section:Resonance}, we quantify the transversality in case of resonant interaction. This allows for the proof of bilinear Strichartz estimates and a trilinear estimate based on the nonlinear Loomis--Whitney inequality. In Section \ref{section:Quasilinear}, we prove Theorem \ref{thm:QuasilinearLWP} by showing short-time nonlinear estimates and energy estimates in short-time function spaces. In Section \ref{section:Semilinear}, we show Theorem \ref{thm:SemilinearLWP}.
In the Appendix, we provide details of the proof of the trilinear estimate as a consequence of the nonlinear Loomis--Whitney inequality.

\section{Notation and function spaces}
\label{section:Notation}
We use $a\pm$ to denote $a \pm\epsilon$ for $\epsilon >0$ sufficiently small. Also, we use notation $A \lesssim B$ for $A \leqslant C B$ with $C$ a harmless constant, which is allowed to change from line to line. Dyadic numbers are denoted by capital letters $N,L,\ldots \in 2^{\mathbb{Z}}$.
\subsection{Fourier transform}
Spatial variables are denoted by $(x,y) \in \R^2$, and the time variable by $t \in \R$. The corresponding Fourier variables are denoted by $(\xi,\eta) \in \R^2$ and $\tau$, respectively.
We use the following convention for the space-time Fourier transform:
\begin{equation*}
\hat{u}(\tau,\xi,\eta) = (\mathcal{F}_{t,x,y} u)(\tau,\xi,\eta) = \int_{\R^3} e^{-i(t \tau + x \xi + y \eta)} u(t,x,y) dt dx dy.
\end{equation*}
We shall also use notation $\hat{u} = \mathcal{F}_{x,y} u$ for the purely spatial Fourier transform, which should be clear from context.
The Fourier transform is inverted by
\begin{equation*}
u(t,x,y) = \frac{1}{(2 \pi)^3} \int_{\R^3} e^{i(t \tau + x \xi + y \eta)} \hat{u}(\tau,\xi,\eta) d\tau d\xi d\eta.
\end{equation*}

\subsection{Function spaces}
We introduce the short-time $X^{s,b}$ spaces now and state their properties. The proofs of the forthcoming results can be found in \cite{IKT}, and we refer to \cite[Section 2.5]{Schippa2019PhDThesis} for an overview of the properties.

Let $\phi_1 \in C^\infty_c(-2,2)$ be symmetric and decreasing on $[0,\infty)$ with $\phi_1(\xi) = 1$ for $\xi \in [-1,1]$.	For $N \in 2^{\N}$, let $\phi_N(\xi) = \phi_1(\xi/N) - \phi_1(2\xi/N)$. We have
\begin{equation*}
	\phi_1(\xi)+ \sum_{N \geqslant 2} \phi_N(\xi) \equiv 1.
\end{equation*}
Let $\N_0:=\N \cup \{0\}$. We define Littlewood-Paley projections: For $f \in \mathcal{S}'(\R^d)$ and $N \in 2^{\N_0}$, let
\begin{equation*}
	(P_{N} f)^{\wedge} (\xi,\eta) = \phi_N(\xi) \hat{f}(\xi,\eta).
\end{equation*}
For $N \in 2^{\mathbb{N}}$, let
\begin{equation*}
	\begin{split}
		A_N &= \Big\{ (\xi,\eta) \in \R^2 : \frac{N}{8} \leqslant |\xi| \leqslant 8N \Big\}, \\
	\end{split}
\end{equation*}
with the obvious modification for $A_1$. Moreover, for $N \in 2^{\Z}$, we let
\begin{equation*}
 \tilde{A}_N = \Big\{ (\xi,\eta) \in \R^2 : \frac{N}{8} \leqslant |\xi| \leqslant 8N \Big\}.
\end{equation*}
 Additionally, for $N \in 2^{\N_0}$, $L \in 2^{\N}$, we define
\begin{equation*}
\begin{split}
		D_{N,L} &= \Big\{ (\tau,\xi,\eta) \in \R \times \R\times \R : (\xi,\eta) \in A_N, \, \frac{L}{4} \leqslant |\tau - \omega_{\alpha}(\xi,\eta)| \leqslant 4L \Big\}, \\
		 D_{N,\leqslant L} &= \bigcup_{L'=1}^L D_{N,L'}, \\
		D_{N,1} &= \{ (\tau,\xi,\eta) \in \R \times \R\times \R : (\xi,\eta) \in A_N, \, |\tau - \omega_{\alpha}(\xi,\eta)| \leqslant 2 \}. 
\end{split}
\end{equation*}
For $N\in 2^{\Z}$, $L \in 2^{\mathbb{N}}$, we define
\begin{equation*}
\tilde{D}_{N,L} = \{ (\tau,\xi,\eta) \in \R \times \R\times \R : (\xi,\eta) \in \tilde{A}_N, \, \frac{L}{4} \leqslant |\tau - \omega_{\alpha}(\xi,\eta)| \leqslant 4L \}
\end{equation*}
with the obvious modification for $L=1$.

In the following we write for notational convenience, in order to distinguish modulation and spatial frequencies, $\eta_L(\tau) = \phi_L(\tau)$ for $L \in 2^{\N_0}$, and $\eta_{\leqslant L}(\tau) = \sum_{L'\in 2^{\N_0} \cap [1,L]} \eta_{L'} (\tau)$. We let
\begin{equation*}
	X_N = \{ f \in L^2(\R \times \R^2) : f \text{ is supported in } \R \times A_N, \; \| f \|_{X_N} < \infty \},
\end{equation*}
and
\begin{equation*}
	\| f \|_{X_N} = \sum_{L \in 2^{\N_0}} L^{\frac{1}{2}} \| \eta_L(\tau - \omega_{\alpha}(\xi,\eta))f \|_{L^2_\tau L^2_{\xi,\eta}}.
\end{equation*}
Note that
\begin{equation*}
	\Big\| \int_{\R} | f(\tau,\xi,\eta)| d\tau \Big\|_{L^2_{\xi,\eta}} \lesssim \| f \|_{X_{N}},
\end{equation*}
and we record the estimate
\begin{equation}\label{prop}
	\begin{split}
		&\sum_{L' \geqslant L} L'^{\frac{1}{2}} \Big\| \eta_{L'}(\tau-\omega_{\alpha}(\xi,\eta)) \int |f(\tau',\xi,\eta)| L'^{-1} (1+L'^{-1} | \tau - \tau'|)^{-4} d \tau' \Big\|_{L_{\tau,\xi,\eta}^2} \\
		&\quad + L^{\frac{1}{2}} \Big\| \eta_{\leqslant L}(\tau - \omega_{\alpha}(\xi,\eta)) \int |f(\tau',\xi,\eta)| L^{-1} (1+L^{-1} |\tau - \tau'|)^{-4} d\tau' \Big\|_{L_{\tau,\xi,\eta}^2} \lesssim \| f \|_{X_N}.
	\end{split}
\end{equation}
We find for Schwartz functions $\gamma \in \mathcal{S}(\R)$, $M,N \in 2^{\N_0}$, $t_0 \in \R$, $f \in X_N$, the estimate
\begin{equation*}
	\| \mathcal{F}_{t,x,y} [\gamma(M (t - t_0)) \mathcal{F}^{-1}_{t,x,y}(f)] \|_{X_N} \lesssim_\gamma \| f \|_{X_N}.
\end{equation*}
We define
\begin{equation*}
	E_N = \{ \phi : \R^2 \to \R \, : \, \hat{\phi} \text{ is supported in } A_N, \; \| \phi \|_{E_N} = \| \phi \|_{L^2} < \infty \}.
\end{equation*}
For $\alpha \in (2,5/2]$ and dyadic frequency $N \in 2^{\mathbb{N}_0}$, we choose the time localization as $N^{-(5-2\alpha)-\varepsilon(\alpha)}$. Next, define
\begin{equation*}
\begin{split}
	&\; F_N = \{ u_N \in C(\R;E_N) \, : \, \\
	&\qquad \quad \| u_N \|_{F_N} = \sup_{t_N \in \R} \| \mathcal{F}_{t,x,y}[u_N \cdot \eta_0(N^{(5-2\alpha)+\varepsilon}(t-t_N))] \|_{X_N} < \infty \}.
	\end{split}
\end{equation*}
The dependence on $\alpha$ is suppressed. We place the solution into these short-time function spaces after dyadic frequency localization. For the nonlinearity, we consider correspondingly
\begin{equation*}
\begin{split}
	\mathcal{N}_N &= \{ u_N \in C(\R;E_N) \, : \, \| u_N \|_{\mathcal{N}_N} = \sup_{t_N \in \R} \| (\tau- \omega_{\alpha}(\xi,\eta) + i N^{(5-2 \alpha)+\varepsilon})^{-1} \\
	&\quad \times \mathcal{F}_{t,x,y} [u_N \cdot \eta_0(N^{(5-2\alpha)+\varepsilon}(t-t_N))] \|_{X_N} < \infty\}.
\end{split}
\end{equation*}
We localize the spaces in time by the usual means: For $T \in (0,1]$, let
\begin{equation*}
	\begin{split}
		F_N(T) &= \{ u_N \in C([-T,T];E_N) \, : \, \| u_N \|_{F_N(T)} = \inf_{\substack{\tilde{u}_N = u_N \text{ in } \\
				 [-T,T] \times \R^2}} \| \tilde{u}_N \|_{F_N} < \infty \}, \\
		\mathcal{N}_N(T) &= \{ u_N \in C([-T,T];E_N) \, : \, \| u_N \|_{\mathcal{N}_N(T)} = \inf_{\substack{\tilde{u}_N = u_N \text{ in } \\  [-T,T] \times \R^2 }} \| \tilde{u}_N \|_{\mathcal{N}_N} < \infty \}.
	\end{split}
\end{equation*}
Let $H^{\infty,0}(\R^2) = \bigcap_{s\geqslant 0} H^{s,0}(\R^2)$. We assemble the spaces $F^{s,0}(T)$, $\mathcal{N}^{s,0}(T)$, and $E^{s,0}(T)$ via Littlewood-Paley decomposition:
\begin{equation*}
	\begin{split}
		F^{s,0}(T) &= \{ u \in C([-T,T];H^{\infty,0}(\R^2)) \, : \, \| u \|^2_{F^{s,0}(T)} = \sum_{N \in 2^{\N_0}} N^{2s} \| P_N  u \|^2_{F_N(T)} < \infty \}, \\
		\mathcal{N}^{s,0}(T) &= \{ u \in C([-T,T];H^{\infty,0}(\R^2)) \, : \, \| u \|^2_{\mathcal{N}^{s,0}(T)} = \sum_{N \in 2^{\N_0}} N^{2s} \| P_N u \|^2_{\mathcal{N}_N(T)} < \infty \}, \\
		E^{s,0}(T) &= \{ u \in C([-T,T];H^{\infty,0}(\R^2)) \, :\\
		&\hspace{2cm} \, \| u \|^2_{E^{s,0}(T)} = \sum_{N \in 2^{\N_0}} N^{2s} \sup_{t\in [-T,T]} \| P_N  u(t) \|^2_{E_N(T)} < \infty \}. \\
	\end{split}
\end{equation*}

We state the multiplier properties of admissible time-multiplication. For $N \in 2^{\N_0}$, we define the set $S_N$ of $N$-acceptable time multiplication factors:
\begin{equation*}
	S_N = \{m_N:\R\rightarrow \R : \|m_N\|_{S_N} = \sum_{j=0}^{10} N^{-((5-2 \alpha)+\varepsilon)j} \| \partial_j m_N\|_{L^{\infty}} < \infty\}.
\end{equation*}

We have, for any $s \geqslant 0$ and $T \in (0,1]$
\begin{equation}\label{eq:TimeMult}
	\left\{ \begin{array}{cl}
		\| \sum_{N \in 2^{\N_0}} m_N(t) P_N(u) \|_{F^{s,0}(T)} &\lesssim \big( \sup_{N \in 2^{\N_0}} \| m_N \|_{S_N} \big) \| u \|_{F^{s,0}(T)}, \\
		\| \sum_{N \in 2^{\N_0}} m_N(t) P_N(u) \|_{\mathcal{N}^{s,0}(T)} &\lesssim \big( \sup_{N \in 2^{\N_0}} \| m_N \|_{S_N} \big) \| u \|_{\mathcal{N}^{s,0}(T)}, \\
		\| \sum_{N \in 2^{\N_0}} m_N(t) P_N(u) \|_{E^{s,0}(T)} &\lesssim \big( \sup_{N \in 2^{\N_0}} \| m_N \|_{S_N} \big) \| u \|_{E^{s,0}(T)}.
	\end{array} \right.
\end{equation}

Next, recall the embedding $F^{s,0}(T) \hookrightarrow C([-T,T];H^{s,0})$ and the linear energy estimate for short-time $X^{s,b}$ spaces. The following statements were proved for the KP-I equation in \cite{IKT} with the proofs carrying over to the present setting.
\begin{lemma}[{cf.~\cite[Lemma~3.1]{IKT}}]\label{lemma:Embed}
	Let $T \in (0,1]$. If $u\in F^{s,0}(T)$, then
	\begin{equation}
		\sup_{t\in [-T,T]}\|u(t)\|_{H^{s,0} }\lesssim \|u\|_{F^{s,0}(T)}.
	\end{equation}
\end{lemma}

\begin{lemma}[{cf.~\cite[Proposition~3.2]{IKT}}]\label{lemma:LinShortTime}
	Let $T \in (0,1]$, $u \in C([-T,T];H^{\infty,0})$ and 
	\begin{equation*}
		\partial_t u - D_x^{\alpha}\partial_{x} u - \partial_{x}^{-1} \partial_y^2 u = f, \quad (x,y) \in \R^2, ~~t \in (-T,T).  
	\end{equation*}
	Then, the following estimate holds:
	\begin{equation}
		\|u\|_{F^{s,0}(T)} \lesssim \|u\|_{E^{s,0}(T)} + \| f \|_{\mathcal{N}^{s,0}(T)}.
	\end{equation}
	
\end{lemma}

\subsection{Linear Strichartz estimates} We define the linear propagator $U_{\alpha}(t)$ as a Fourier multiplier acting on functions $\phi \in \mathcal{S}(\R^2)$ whose Fourier transform is supported away from the origin
\begin{equation}\label{eq:LinearProp}
    (U_{\alpha}(t)\phi)^{\wedge}(\xi,\eta) = e^{it \omega_\alpha(\xi,\eta)} \hat{\phi}(\xi,\eta) = e^{it(|\xi|^{\alpha}\xi +\frac{\eta^2}{\xi})}\hat{\phi}(\xi,\eta).
\end{equation}
Since $U_{\alpha}(t)$ is a linear isometric mapping on $H^{s_1,s_2}$, the above extends by density. We state the linear Strichartz estimates. These enable us to handle the non-resonant interactions. Furthermore, we observe the smoothing effect pertaining to the higher dispersion for $\alpha>2$. The following Strichartz estimates are due to Hadac \cite{Hadac} for dispersion-generalized KP-II equations, but it is easy to see that the argument transfers to KP-I equations, as pointed out for $\alpha =2 $ by Saut \cite{Saut1993}.
\begin{theorem}[Linear~Strichartz~estimate,~{cf.~\cite[Theorem 3.1]{Hadac}}]
	Let $\alpha \geqslant 2$, $2<q\leqslant \infty$, and
	\begin{equation*}
		\frac{1}{q}+\frac{1}{r} = \frac{1}{2},\quad \gamma := \Big (1-\frac{2}{r}\Big) \Big(\frac{1}{2} - \frac{\alpha}{4}\Big).
	\end{equation*}
	Then, we have
	\begin{equation}\label{eq:LinStriFKP}
		\|D_x^{-\gamma} U_{\alpha}(t)u_0\|_{L_t^qL_{x,y}^r} \lesssim \|u_0\|_{L^2_{x,y}}
	\end{equation}
	with $(D_x^{-\gamma} f)^{\wedge} (\xi,\eta) = |\xi|^{-\gamma} \hat{f}(\xi,\eta)$ for $\gamma \in \R$.
\end{theorem}
We record a second linear Strichartz estimate for low $x$ frequencies whose proof is simpler:
\begin{lemma}[Strichartz estimates for low frequencies]
Let $N,K \in 2^{\Z}$, $I \subseteq \R$ be an interval of length $|I| \sim K$, and $|\xi| \sim N$ for any $\xi \in I$. Suppose that $\hat{u}_0(\xi,\eta) = 0$, if $\xi \notin I$. Then, the following estimate holds:
\begin{equation}
    \label{eq:L4StrichartzDyadic}
    \| U_\alpha(t) u_0 \|_{L_t^4([0,1]; L^4_{x,y}(\R^2))} \lesssim K^{\frac{1}{4}} N^{\frac{1}{8}} \| u_0 \|_{L^2}.
\end{equation}
\end{lemma}
\begin{proof}
    We use Bernstein's inequality in $x$, Plancherel's theorem, and Minkowski's inequality to find
    \begin{equation*}
        \begin{split}
        &\quad \Big\| \int e^{i(x\xi + y\eta + t(\xi |\xi|^\alpha + \frac{\eta^2}{\xi}))} \hat{u}_0(\xi,\eta) d\xi d\eta \Big\|_{L_t^4([0,1];L^4_{x,y}(\R^2))} \\
         &\lesssim K^{\frac{1}{4}} \Big\| \int e^{i(y\eta + t \frac{\eta^2}{\xi})} \hat{u}_0(\xi,\eta) d\eta \Big\|_{L_t^4([0,1];L_y^4 L_\xi^2)} \\
        &\lesssim K^{\frac{1}{4}} \Big( \int_I d\xi \Big\| \int e^{i(y\eta + t \frac{\eta^2}{\xi} )} \hat{u}_0(\xi,\eta) d\eta \big\|^2_{L_t^4([0,1];L^4_y(\R)} \Big)^{\frac{1}{2}}.
        \end{split}
    \end{equation*}
     Hence, it suffices to prove
    \begin{equation}
        \Big\| \int e^{i(y\eta + t \frac{\eta^2}{\xi})} \hat{u}_0(\xi,\eta) d\eta \Big\|_{L_t^4([0,1]; L^4_y(\R))} \lesssim N^{\frac{1}{8}} \| \hat{u_0}(\xi,\cdot) \|_{L^2_\eta}.
    \end{equation}
    By a change of variables supposing $\xi >0$ without loss of generality and H\"older in time, we find
    \begin{equation*}
    \begin{split}
        &\qquad \Big\| \int e^{i(y\eta + t \frac{\eta^2}{\xi})} \hat{u}_0(\xi,\eta) d\eta \big\|_{L_t^4([0,1]; L^4_y(\R))} \\
         &\lesssim N^{\frac{1}{4}} \Big\| \int e^{i(y\eta + t \eta^2)} \hat{u}_0(\xi,\eta) d\eta \Big\|_{L_t^4([0,\xi^{-1}]; L^4_y(\R))} \\
        &\lesssim N^{\frac{1}{8}} \Big\| \int e^{i(y\eta + t \eta^2)} \hat{u}_0(\xi,\eta) d\eta \Big\|_{L_t^8([0,\xi^{-1}]; L^4_y(\R))} \\
        &\lesssim N^{\frac{1}{8}} \| \hat{u}_0(\xi,\cdot) \|_{L^2_\eta}.
    \end{split}    
    \end{equation*}
    The ultimate estimate is an application of the $L_t^8 L_y^4$-Strichartz estimate for the one-dimensional Schr\"odinger equation (cf. \cite[Section~2.3]{Tao2006}).
\end{proof}

As a consequence of the transfer principle (cf. \cite[Lemma~2.9]{Tao2006}), we have the following:
\begin{corollary}
	Let $u \in X_N$ be supported in $\tilde{D}_{N,L}$. Then 
	\begin{equation}\label{eq:StriEmbed}
\| \mathcal{F}^{-1}(u)\|_{L^4_{t,x,y}} \lesssim [\max(1,N)]^{\frac{1}{4}-\frac{\alpha}{8}} L^{\frac{1}{2}}\|u\|_{L^2}.
	\end{equation}
\end{corollary}

\section{$C^2$ ill-posedness}
\label{section:Illposedness}
In this section, we prove that \eqref{eq:fKPI} cannot be solved via Picard iteration for $\alpha$ close to $2$ as stated in Theorem \ref{thm:IllPosed}. This is a consequence of the derivative nonlinearity in case of resonance.

Recall that the resonance function is given by
	\begin{equation*}
		\Omega_{\alpha}(\xi_1,\xi_2,\eta_1,\eta_2) = |\xi_1+\xi_2|^{\alpha}(\xi_1+\xi_2)-|\xi_1|^{\alpha}\xi_1 -|\xi_2|^{\alpha}\xi_2 -\frac{(\eta_1\xi_2 - \eta_2\xi_1)^2}{\xi_1 \xi_2 (\xi_1+\xi_2) }.
	\end{equation*}
This will quantify the time oscillation in the Duhamel integral.	
To estimate the size of the resonance function, we separate $\Omega_{\alpha}$ as $\Omega_{\alpha} = \Omega_{\alpha}^1 - \Omega_{\alpha}^2$, where
\begin{equation*}
	\begin{split}
	\Omega_{\alpha}^1(\xi_1,\xi_2,\eta_1,\eta_2)& = |\xi_1+\xi_2|^{\alpha}(\xi_1+\xi_2)-|\xi_1|^{\alpha}\xi_1 -|\xi_2|^{\alpha}\xi_2,\\
	\Omega_{\alpha}^2(\xi_1,\xi_2,\eta_1,\eta_2)& = \frac{(\eta_1\xi_2 - \eta_2\xi_1)^2}{\xi_1 \xi_2 (\xi_1+\xi_2) }.
	\end{split}
	\end{equation*}
In the following we denote $H^{\bar{s}} = H^{s_1,s_2}$.

\begin{proof}[Proof of Theorem \ref{thm:IllPosed}]
    
We define the functions $\phi_1$ and $\phi_2$ via their Fourier transform
\begin{equation}
	\begin{split}
	\hat{\phi}_1(\xi_1,\eta_1) &= \gamma^{-\frac{3}{2}} \mathbf{1}_{D_1}(\xi_1,\eta_1), \\
	\hat{\phi}_2(\xi_2,\eta_2)&= \gamma^{-\frac{3}{2}} N^{-s_1-(1+\frac{\alpha}{2})s_2}\mathbf{1}_{D_2}(\xi_2,\eta_2),
	\end{split}
\end{equation}
where $D_i = \tilde{D_i} \cup (-\tilde{D_i})$ and $\tilde{D_i}$ are defined as follows:
\begin{equation}\label{eq:DefinitionSets}
	\begin{split}
		\tilde{D}_1&:= [\gamma/2,\gamma]\times [-\sqrt{1+\alpha}~\gamma^2,\sqrt{1+\alpha}~\gamma^2],\\
		\tilde{D}_2&:= [N,N+\gamma] \times [\sqrt{1+\alpha}~N^{\frac{\alpha+2}{2}}, \sqrt{1+\alpha}~N^{\frac{\alpha+2}{2}}+\gamma^2].
	\end{split}
\end{equation}
Here $N,\gamma>0$ are real numbers such that $N\gg 1, \gamma\ll 1$ and will be chosen later.
A simple computation gives $\|\phi_i\|_{H^{\bar{s}}(\R^2)}\sim 1, i=1,2$. We consider the initial data $u_0 = \phi_1 + \phi_2$ and to disprove that $\Gamma_t$ is $C^2$ at the origin, it suffices to show that
\begin{equation*}
\Big\| \int_0^t U_\alpha(t-s) \partial_x (U_\alpha(s) u_0 U_\alpha(s) u_0) ds \Big\|_{H^{s_1,s_2}} \to \infty
\end{equation*}
as $N \to \infty$. We show the above for the contribution, which comes from the interaction of a high with a low frequency. This is denoted by $u_2$ below. Here we are using that the Fourier support is disjoint from the Fourier support of other possible interactions like low-low- or high-high-interaction. For more details, we refer to the proof of \cite[Theorem~3.2]{LPS} for fractional KP-I equations with weaker dispersion (see also \cite{MolinetSautTzvetkov2002} for KP-I).
 We can write the Fourier transform of
 \begin{equation*}
     u_2(t) = \int_0^t U_\alpha(t-s) \partial_x (U_\alpha(s) \phi_1 U_{\alpha}(s) \phi_2) ds
 \end{equation*}
 as
\begin{equation*}
	\hat{u}_2(t,\xi,\eta) = \frac{c\xi e^{it(|\xi|^{\alpha}\xi +\frac{\eta^2}{\xi})}}{|D_1|^{\frac{1}{2}} |D_2|^{\frac{1}{2}}N^{s_1+(1+\frac{\alpha}{2})s_2}}  \int_{\substack{{(\xi_1,\eta_1)\in D_1,}\\{(\xi_2,\eta_2)\in D_2}}} \frac{e^{-it\Omega_{\alpha}(\xi_1,\xi_2,\eta_1,\eta_2)} - 1}{\Omega_{\alpha}(\xi_1,\xi_2,\eta_1,\eta_2)}d\xi_1d\eta_1.
\end{equation*}

We estimate the size of the resonance function as follows.

\begin{lemma}[Size of the resonance function]\label{lemma:ResoFuncSize}
Let $(\xi_i,\eta_i)\in D_i$, $i=1,2$, then
\begin{equation*}
	|\Omega^1_{\alpha}(\xi_1,\xi_2,\eta_1,\eta_2)| \sim N^{\alpha}\gamma.
\end{equation*}
\begin{proof}
	We carry out a case-by-case analysis:\\
	\textbf{(i)} $\xi_1>0, \xi_2>0$: Using the mean value theorem, 
	\begin{equation*}
		(\xi_1+\xi_2)^{\alpha +1}  - \xi_2^{\alpha+1} = (\alpha+1)\xi_1 \xi_{\ast}^{\alpha}, \quad \xi_{\ast}\in (\xi_2,\xi_2+\xi_1).
	\end{equation*}
This gives
\begin{equation*}
	|\Omega^1_{\alpha}| \sim |\xi_1((\alpha+1)(\xi_{\ast}^{\alpha} - \xi_1^{\alpha}))| \sim N^{\alpha}\gamma.
\end{equation*}
\textbf{(ii)} $|\xi_1| <|\xi_2|, \xi_1>0, \xi_2<0$: Define new variables by $\xi_2^{'} = -\xi_2, \; \xi_2^{'} >0$.
We have $ \xi_2^{'} = \xi_1 - (\xi_1-\xi_2^{'})$. Hence,
\begin{equation*}
	\begin{split}
		\Omega_{\alpha}^1 (\xi_1, \xi_2,\eta_1,\eta_2) &= |\xi_1+\xi_2|^{\alpha}(\xi_1+\xi_2) - |\xi_1|^{\alpha}\xi_1 - |\xi_2|^{\alpha}\xi_2\\
		&=-(\xi_2^{'}-\xi_1)^{\alpha+1} -\xi_1^{\alpha+1} + (\xi_2^{'})^{\alpha+1}\\
		&=(\xi_2^{'})^{\alpha+1} - (\xi_2^{'} -\xi_1)^{\alpha+1} - \xi_1^{\alpha+1}.
	\end{split}
\end{equation*}
This is the same form as obtained in case \textbf{(i)}. Hence, we can conclude the same for this case.\\
\end{proof}
\end{lemma}
\begin{remark}
The above argument can be used to determine the size of the resonance function in other cases.
\end{remark}

\emph{Proof of Theorem \ref{thm:IllPosed} (ctd)}. Using Taylor's theorem, we have
\begin{equation*}
	\begin{split}
	&|\Omega_{\alpha}^1(\xi_1,\xi_2,\eta_1,\eta_2)| = N^{\alpha}\gamma + O(N^{\alpha-1}\gamma^2),\\
	\text{ and } &|\Omega_{\alpha}^2(\xi_1,\xi_2,\eta_1,\eta_2)| = N^{\alpha}\gamma + O(N^{\frac{\alpha}{2}} \gamma^2).
	\end{split}
\end{equation*}
Since $\Omega_{\alpha} = \Omega_{\alpha}^1 -\Omega_{\alpha}^2$, for $\alpha > 2$, we obtain,
\begin{equation*}
	|\Omega_{\alpha}(\xi_1,\xi_2,\eta_1,\eta_2)| \sim N^{\alpha-1}\gamma^2.
\end{equation*}
We choose $\gamma = N^{-\frac{\alpha-1}{2}-\theta}, \theta>0$, which makes the resonance function small. The $H^{\bar{s}}(\R^2)$ norm of $u_2(t,\cdot,\cdot)$ is given by
\begin{equation*}
	\| u_2(t, \cdot, \cdot)\|_{H^{\bar{s}}(\R^2)} \sim N\gamma^{\frac{3}{2}}= N^{\frac{7}{4}-\frac{3\alpha}{4}-\frac{3\theta}{2}}.
\end{equation*}
For $\Gamma_t$ to be $C^2$, we require
\begin{equation*}
1 \sim \|\phi_1\|_{H^{\bar{s}}(\R^2)}  \|\phi_2\|_{H^{\bar{s}}(\R^2)} \gtrsim N^{\frac{7}{4}-\frac{3\alpha}{4}-\frac{3\theta}{2}},
 \end{equation*}
i.e., $\alpha \geqslant \frac{7}{3}$. This completes the proof.
\end{proof}

\section{Resonance, transversality, and the nonlinear Loomis--Whitney inequality}
\label{section:Resonance}
In this section, we analyze the resonance function and use it to obtain trilinear estimates via the nonlinear Loomis--Whitney inequality. Moreover, we employ transversality in the resonant case to obtain genuinely bilinear estimates.
We recall that the dispersion relation for the fKP-I equation is given by
\begin{equation*}
	\omega_{\alpha}(\xi,\eta) = |\xi|^{\alpha}\xi +\frac{\eta^2}{\xi},
\end{equation*}
and the resonance function is given by
\begin{equation*}
\Omega_{\alpha}(\xi_1,\xi_2,\eta_1,\eta_2) = |\xi_1+\xi_2|^{\alpha}(\xi_1+\xi_2) - |\xi_1|^{\alpha}\xi_1 - |\xi_2|^{\alpha}\xi_2 -\frac{(\eta_1\xi_2-\eta_2\xi_1)^2}{\xi_1\xi_2(\xi_1+\xi_2)}.
\end{equation*}
We say that we are in the resonant case, if
\begin{equation*}
|\Omega_{\alpha}| \ll ||\xi_1+\xi_2|^{\alpha}(\xi_1+\xi_2) - |\xi_1|^{\alpha}\xi_1 - |\xi_2|^{\alpha}\xi_2|.
\end{equation*}
Suppose that we have $N_{\max} \sim |\xi_1+\xi_2|\sim |\xi_1| \gtrsim |\xi_2|\sim N_{\min}$, then from the computation done in Lemma \ref{lemma:ResoFuncSize}, we get that the right-hand side in the above equation has size $N_{\max}^{\alpha}N_{\min}$. We find in the resonant case
\begin{equation*}
N_{\max}^{\alpha}N_{\min} \sim \Big| \frac{(\eta_1\xi_2-\eta_2\xi_1)^2}{\xi_1\xi_2(\xi_1+\xi_2)} \Big|. 
\end{equation*}
This can be further simplified to 
\begin{equation}\label{eq:ResofKP}
	|\eta_1\xi_2-\eta_2\xi_1| \sim N_{\max}^{\frac{\alpha}{2}+1}N_{\min}.
\end{equation}
We consider the gradient of the dispersion relation next:
\begin{equation*}
	\nabla \omega_{\alpha}(\xi,\eta) = \Big ( (\alpha + 1)|\xi|^{\alpha}-\frac{\eta^2}{\xi^2}, \frac{2\eta}{\xi}\Big).
\end{equation*}
Using \eqref{eq:ResofKP}, we have
\begin{equation}\label{eq:GradDisp}
	|\nabla \omega_{\alpha}(\xi_1,\eta_1)-\nabla\omega_{\alpha}(\xi_2,\eta_2)| \gtrsim \Big |\frac{\eta_1}{\xi_1} - \frac{\eta_2}{\xi_2}\Big| \sim N_{\max}^{\frac{\alpha}{2}}.
\end{equation}
The above relation shall be employed to obtain precise multilinear estimates via the nonlinear Loomis--Whitney inequality and bilinear Strichartz estimates.

\subsection{Nonlinear Loomis--Whitney inequality}
In this section, we state the setting and prove the trilinear estimate in the resonant case via the nonlinear Loomis--Whitney inequality from \cite{KS}. We recall the assumptions on the parametrizations.
\subsubsection{Assumption: }\label{assump:LWAssump}
For $i=1,2,3$, there exist $0<\beta\leqslant 1$, $b>0$, $A \geqslant 1$, $F_i \in C^{1,\beta}(\mathcal{U}_i)$, where $\mathcal{U}_i$ denote open and convex sets in $\R^2$ and $G_i \in O(3)$ such that
\begin{enumerate}
	\item the oriented surfaces $S_i$ are given by
	\begin{equation*}
		S_i = G_igr(F_i), \quad gr(F_i)=\{(z,x,y) \in \R^3 | ~z=F_i(x,y), (x,y)\in \mathcal{U}_i\}.
	\end{equation*}
\item the unit normal vector field $\mathbf{n}_i$ on $S_i$ satisfies the H\"older condition 
\begin{equation*}
	\sup_{\sigma,\tilde{\sigma} \in S_i} \frac{|\mathfrak{n}_i(\sigma) - \mathfrak{n}_i(\tilde{\sigma})|}{|\sigma-\tilde{\sigma}|^{\beta}} + \frac{|\mathfrak{n}_i(\sigma)(\sigma-\tilde{\sigma})|}{|\sigma-\tilde{\sigma}|^{1+\beta}} \leqslant b;
\end{equation*}
\item the matrix $N(\sigma_1,\sigma_2,\sigma_3) = (\mathfrak{n}_1(\sigma_1), \mathfrak{n}_2(\sigma_2),\mathfrak{n}_3(\sigma_3))$ satisfies the transversality condition
\begin{equation*}
	A^{-1} \leqslant \det~ N(\sigma_1,\sigma_2,\sigma_3) \leqslant 1,
\end{equation*}
for all $(\sigma_1,\sigma_2,\sigma_3) \in S_1 \times S_2 \times S_3$.
\end{enumerate}
For $\epsilon >0$, by $S_i(\epsilon)$ we denote
\begin{equation*}
    S_i(\epsilon):=G_i\{(z,x,y)\in \R \times \mathcal{U}_i : |z-F_i(x,y)| <\epsilon\}.
    \end{equation*}

\begin{theorem}\cite[Theorem 4.3]{KS}
	\label{thm:KSLoomisWhitney}
	Let $A$ be dyadic and $f_i \in L^2(S_i(\epsilon))$, $i=1,2$. Suppose that $(S_i)_{i=1}^3$ satisfies Assumption \ref{assump:LWAssump}. Then, for $\epsilon >0$, we find the following estimate to hold:
	\begin{equation}
	\label{eq:KSLoomisWhitney}
		\|f_1\ast f_2\|_{L^2(S_3(\epsilon))} \lesssim \epsilon^{\frac{3}{2}} A^{\frac{1}{2}} \|f_1\|_{L^2(S_1(\epsilon))} \|f_2\|_{L^2(S_2(\epsilon))},
	\end{equation}
where the implicit constant is independent of $\beta$ and $b$.
\end{theorem}

In the following, we apply Theorem \ref{thm:KSLoomisWhitney} in the resonant case to obtain a trilinear estimate:
\begin{lemma}\label{lemma:LW}
	Let $\alpha >0$, and $N_1,N_2,N_3\in 2^{\Z}$ be such that $N_2 \lesssim N_1 \sim N_3$, $N_1 \gtrsim 1$ and $L_1,L_2,L_3\in 2^{\N_0}$ such that $L_1,L_2,L_3 \ll N_1^{\alpha} N_2$. Let $f, g , h: \R \times \R^2 \rightarrow \R_{+}$ be $L^2$ functions supported in $\tilde{D}_{N_1,\leqslant L_1}, \tilde{D}_{N_2,\leqslant L_2}$ and $\tilde{D}_{N_3,\leqslant L_3}$, respectively. Then
	\begin{equation}\label{eq:LW}
\Big| \int (f \ast g)\cdot h\Big| \lesssim N_1^{-\frac{3\alpha}{4} +\frac{1}{2}} N_2^{-\frac{1}{2}} (L_1L_2L_3)^{\frac{1}{2}}\|f\|_{L^2} \|g\|_{L^2} \|h\|_{L^2}.
	\end{equation}
\end{lemma}
\begin{proof}
Taking into account the localization of the functions, for $N_2\ll N_1 \sim N_3$, we write the left-hand side of \eqref{eq:LW} as
\begin{equation*}
	\Big |\int (f_{1,N_1,L_1} \ast f_{2,N_2,L_2})\cdot f_{3,N_3,L_3}\Big|.
\end{equation*}
We shall estimate the above in the resonant case, where $L_1,L_2,L_3 \leqslant C N_1^{\alpha}N_2$. Furthermore, we can decompose $f_i$, $i=1,2,3$ into  $L_i$ number of pieces, respectively. This means $f= \sum_c f^c_i$, where $f^c_i$ is supported on
\begin{equation*}
c \leqslant \tau - \omega_{\alpha}(\xi,\eta) \leqslant c + 1
\end{equation*}
for some $c \in \Z$.
To lighten the notation, we still denote the decomposed pieces by $f_{i,N_i,L_i}$, $i=1,2,3$. After a harmless translation, we can suppose that these are supported in the unit neighborhood of the characteristic surface. Then it suffices to prove \eqref{eq:LW} with $L_i$, $i=1,2,3$ replaced by $1$ because the sum over the additional decomposition is handled by the Cauchy-Schwarz inequality.
We consider the characteristic surface $S_i$, $i=1,2$ given by
\begin{equation*}
	S_i = \Big\{(\tau_i,\xi_i,\eta_i)\in \R\times\R \times \R: \tau_i=\xi_i|\xi_i|^{\alpha}+\frac{\eta_i^2}{\xi_i}\Big\}
\end{equation*}
with surface normals (not necessarily of unit length)
\begin{equation}
\label{eq:SurfaceNormals}
	\mathbf{n}_i=\Big (1,-( (\alpha+1)|\xi_i|^{\alpha}-\frac{\eta_i^2}{\xi_i^2}), -\frac{2\eta_i}{\xi_i}  \Big).
\end{equation}

The condition $L_{\max} \ll N_1^\alpha N_2$ by convolution constraint
\begin{equation*}
\tau_1 - \omega_\alpha(\xi_1,\eta_1) + \tau_2 - \omega_\alpha(\xi_2,\eta_2) - (\tau_3 - \omega_\alpha(\xi_1+\xi_2,\eta_1+\eta_2)) = \Omega_{\alpha}(\xi_1,\xi_2,\eta_1,\eta_2)
\end{equation*}
implies that $|\Omega_\alpha| \ll N_1^\alpha N_2$. Hence, we are in the resonant case $|\Omega_\alpha^1| \sim |\Omega_\alpha^2|$ and obtain \eqref{eq:ResofKP}.

We use the anisotropic rescaling $\tau \to \tau / N_1^{\alpha+1}$, $\xi \to \xi' = \xi / N_1$, $\eta \to \eta' = \eta / N_1^{\alpha/2+1}$, which leaves the dispersion relation invariant and normalizes the large $\xi$-frequency. \eqref{eq:ResofKP} becomes
\begin{equation}
\label{eq:RenormalizedGroupVelocities}
\big| \frac{\eta_1'}{\xi_1'} - \frac{\eta_2'}{\xi_2'} \big| \sim 1.
\end{equation}

Recall that we suppose by additional decomposition that the functions have modulations of size $\lesssim 1$.
After rescaling, the functions have modulation thickness $\varepsilon = N_1^{-(\alpha+1)}$. We obtain
\begin{equation}
\label{eq:Rescaling}
\begin{split}
&\quad \int (f_{1,N_1,L_1} * f_{2,N_2,L_2}) f_{3,N_3,L_3} d \xi d \eta d \tau \\
 &= \int \big( f_{1,N_1,L_1} * f_{2,N_2,L_2} \big) f_{3,N_3,L_3} d \xi d \eta d \tau \\
&= N_1^2 N_1^{2(\frac{\alpha}{2}+1)} N_1^{2(\alpha+1)} \int (f'_{1,N_1,L_1} * f'_{2,N_2,L_2}) f'_{3,N_3,L_3} d \xi' d\eta' d\tau'.
\end{split}
\end{equation}

To apply the nonlinear Loomis-Whitney inequality, we need to find a lower bound for the determinant of the normals and check the regularity conditions \emph{Assumption} \ref{assump:LWAssump} (2). To this end, we shall carry out several almost orthogonal decompositions. It should be mentioned that these decompositions already play a crucial role in the proof of \cite[Lemma~5.1~(a)]{IKT}. 

\medskip

For the decompositions \eqref{eq:RenormalizedGroupVelocities} is the key. This allows us to decompose $\eta_i' / \xi_i'$ into intervals of length $1$, and we find that $\eta_2'/\xi_2'$ is correspondingly decomposed into intervals of length $1$. We can suppose that
\begin{equation*}
\big| \frac{\eta_i'}{\xi_i'} \big| \lesssim 1
\end{equation*}
by Galilean invariance $\eta_i' \rightarrow \eta_i' + c \xi_i'$. By this symmetry we can suppose that the intervals in which $\eta_i' / \xi_i'$ are supported, are essentially centered at the origin. 

Since $|\eta_2'| \lesssim N_2 / N_1$, we can now further decompose $\eta_i'$ into intervals of size $N_2/N_1$ by almost orthogonality. Moreover, $\xi_i'$ can be decomposed into intervals of length $N_2/N_1$ because $|\xi_2'| \sim N_2 / N_1$.

\medskip

Now we can establish the lower bound for the unit normals
\begin{equation}
\label{eq:WedgeProduct}
\big| \mathfrak{n}(\xi_1',\eta_1') \wedge \mathfrak{n}(\xi_2',\eta_2') \wedge \mathfrak{n}(\xi_3',\eta_3')| \gtrsim N_2/N_1.
\end{equation}
Above $(\xi_i',\eta_i')$ are confined to balls of radius $c(N_2/N_1)$ with $c \ll 1$. To show \eqref{eq:WedgeProduct}, we firstly suppose the convolution constraint $(\xi_3',\eta_3') = (\xi_1',\eta_1') + (\xi_2',\eta_2')$ and then argue by perturbation. With the unnormalized normals given in \eqref{eq:SurfaceNormals}
we observe that since $|\eta_i'/\xi_i'| \lesssim 1$, the surface normals from \eqref{eq:SurfaceNormals} are already essentially normalized. We compute with details given in the Appendix:
\begin{equation*}
\begin{split}
&\quad |\mathbf{n}(\xi_1',\eta_1') \wedge \mathbf{n}(\xi_2',\eta_2') \wedge \mathbf{n}(\xi_1'+\xi_2',\eta_1'+\eta_2')| \\
&\sim \frac{|\eta_1' \xi_2' - \eta_2' \xi_1'|}{|\xi_1' \xi_2' (\xi_1'+\xi_2')|} \big| (\alpha+1) (|\xi_1'|^\alpha \xi_1' + \xi_2'^\alpha \xi_2' - |\xi_1'+\xi_2'|^\alpha (\xi_1'+\xi_2')) - \frac{(\eta_1' \xi_2' - \eta_2' \xi_1')^2}{\xi_1' \xi_2' (\xi_1'+\xi_2')} \big) \\
&\sim \frac{N_2}{N_1}.
\end{split}
\end{equation*}
For $|\xi_3'| \sim 1$ and $|\eta_3'| \lesssim 1$ we can observe that
$| \partial_{\xi'} \mathbf{n}(\xi'_3,\eta'_3)| + |\partial_{\eta'} \mathbf{n}(\xi'_3,\eta_3')| \lesssim 1.$
For this reason, and that we can decompose the range of $(\xi_i',\eta_i')$ almost orthogonally into balls of radius $c (N_2/N_1)$ for some $c \ll 1$, we can extend the lower bound from the convolution constraint to the general case in \eqref{eq:WedgeProduct}.

\medskip

To invoke Theorem \ref{thm:KSLoomisWhitney}, we still need to ensure the H\"older regularity conditions for the unit normals of the surface. With the constant in \eqref{eq:KSLoomisWhitney} not depending on $b$ and $\beta$, we can carry out a very crude analysis. In the following let $\beta =1$. Since after our reductions the surface normals \eqref{eq:SurfaceNormals} are essentially normalized, it suffices to verify \emph{Assumption} \ref{assump:LWAssump} (2) with normals given by \eqref{eq:SurfaceNormals}. Let $\sigma = (\omega_\alpha(\xi,\eta),\xi,\eta)$ and $\tilde{\sigma} = (\omega_\alpha(\tilde{\xi},\tilde{\eta}),\tilde{\xi},\tilde{\eta})$. It follows
\begin{equation*}
\frac{|\mathbf{n}_i(\sigma) - \mathbf{n}_i(\tilde{\sigma})|}{|\sigma - \tilde{\sigma}|} \leq \frac{|\nabla \omega_\alpha(\xi,\eta) - \nabla \omega_\alpha(\tilde{\xi},\tilde{\eta})|}{|(\xi-\tilde{\xi},\eta - \tilde{\eta})|} \leq | \partial^2 \omega_\alpha(\xi^*,\eta^*)|
\end{equation*}
by the mean-value theorem. Since $\omega_\alpha \in C^2(\R^2 \backslash (\{0 \} \times \R ))$, we have 
\begin{equation*}
|\partial^2 \omega_\alpha(\xi^*,\eta^*)| \lesssim_{N_1,N_2} 1
\end{equation*}
 for $|(\xi^*,\eta^*)| \lesssim 1$ and $|\xi^*| \gtrsim N_2 / N_1$. 

We turn to the second size condition, which is reduced likewise:
\begin{equation*}
\begin{split}
\frac{|\mathbf{n}_i(\sigma) (\sigma - \tilde{\sigma})|}{|\sigma - \tilde{\sigma}|^{1+\beta}} &\leq \frac{| \omega_\alpha(\xi,\eta) - \omega_\alpha(\tilde{\xi},\tilde{\eta}) - \partial_\xi \omega_\alpha(\xi,\eta)(\xi - \tilde{\xi}) - \partial_\eta \omega_\alpha(\xi,\eta)(\eta-\tilde{\eta})| }{|(\xi - \tilde{\xi},\eta - \tilde{\eta})|^2} \\
&\lesssim |\partial^2 \omega_\alpha(\xi^*,\eta^*)| \lesssim_{N_1,N_2} 1
\end{split}
\end{equation*}
by the mean value theorem. This verifies the H\"older condition.

\medskip

We continue with the estimate of \eqref{eq:Rescaling}. An application of the nonlinear Loomis-Whitney inequality provided by Theorem \ref{thm:KSLoomisWhitney} with $\varepsilon = N_1^{-(\alpha+1)}$ and $A \sim N_1 / N_2$ yields
\begin{equation*}
\begin{split}
\big| \int (f'_{1,N_1,L_1} * f'_{2,N_2,L_2}) f'_{3,N_3,L_3} d \xi' d\eta' d\tau' \big| &\lesssim (N_2/N_1)^{-\frac{1}{2}} N_1^{-\frac{3(\alpha+1)}{2}} \prod_{i=1}^3 \| f_{i,N_i,L_i}' \|_{L^2} \\
&\lesssim (N_2/N_1)^{-\frac{1}{2}} N_1^{-\frac{3(\alpha+1)}{2}} N_1^{-\frac{9}{2}-\frac{9 \alpha}{4}} \prod_{i=1}^3 \| f_{i,N_i,L_i} \|_{L^2}.
\end{split}
\end{equation*}
Taking into account the scaling factor from \eqref{eq:Rescaling}, we finish the proof.

\end{proof}



\subsection{Bilinear Strichartz estimates} We employ transversality in the resonant case to derive bilinear estimates. We first note a trivial result.
\begin{lemma}\label{lemma:elem}
	Let $I, J$ be intervals and $f: J \rightarrow \R$ be a smooth function. Then,
	\begin{equation*}
		|\{x: f(x) \in I\}| \leqslant \frac{|I|}{\inf_{y}|f'(y)|}.
	\end{equation*}
	\begin{proof}
		The estimate is a consequence of the mean value theorem. Let $x_1,x_2 \in J$ be such that $f(x_1),f(x_2) \in I$. Then, for $\xi \in (x_1,x_2)$,
		\begin{equation*}
			|x_1-x_2| =\frac{|f(x_1)-f(x_2)|}{|f'(\xi)|} \leqslant \frac{|I|}{\inf_y |f'(y)|}.
		\end{equation*}
	\end{proof}
\end{lemma}

\begin{proposition}
	Let $\alpha > 0$, suppose that $u, v \in L^2(\R \times \R^2)$ have their Fourier supports in $\tilde{D}_{N_1,L_1}$ and $\tilde{D}_{N_2,L_2}$, respectively, and that for $(\tau_1,\xi_1,\eta_1) \in \text{supp} (\hat{u})$ and $(\tau_2,\xi_2,\eta_2)$ $\in \text{supp} (\hat{v})$, the resonance condition holds. Then,
	\begin{equation}\label{eq:BilinearToProve}
\| uv\|_{L^2_{t,x,y}} \lesssim (L_1L_2)^{\frac{1}{2}} \frac{\min(N_1, N_2)^{\frac{1}{2}}} {\max(N_1,N_2)^{\frac{\alpha}{4}}} \|u\|_{L^2} \|v\|_{L^2}.
	\end{equation}
\begin{proof}
Let $\bar{L}=\max(L_1,L_2)$, $\underline{L}=\min(L_1,L_2)$. Using Plancherel's identity and Cauchy-Schwarz inequality, we have
	\begin{equation}\label{eq:exp}
		\begin{split}
\|uv\|_{L^2}&=\Big \| \int_{\R \times\R^2} \hat{u}(\tau_1,\xi_1,\eta_1)\hat{v}(\tau-\tau_1,\xi-\xi_1,\eta-\eta_1) d\tau_1d\xi_1d\eta_1 \Big\|_{L^2_{\tau,\xi,\eta}}\\
&\lesssim \underline{L}^{\frac{1}{2}} |E(\xi,\eta)|^{\frac{1}{2}} \|u\|_{L^2}\|v\|_{L^2},
\end{split}
	\end{equation}
where the set $E$ is given by
\begin{equation*}
\begin{split}
E(\xi,\eta):=	\{(\xi_1,\eta_1)\in \tilde{A}_{N_1} : &\; |\tau-\omega_{\alpha}(\xi_1,\eta_1) -\omega_{\alpha}(\xi-\xi_1,\eta-\eta_1)| \lesssim \bar{L}, \\
&\quad  (\xi-\xi_1,\eta-\eta_1)\in \tilde{A}_{N_2}\}.
\end{split}
\end{equation*}
The measure of this set can be estimated by Fubini's theorem.
From \eqref{eq:GradDisp}, Lemma \ref{lemma:elem} and almost orthogonality, we have
\begin{equation*}
	|E(\xi,\eta)| = \Big | \int d\xi_1 \int d\eta_1 \mathbf{1}_{E(\xi,\eta)}(\xi_1,\eta_1) \Big| \lesssim \min(N_1,N_2)\frac{\bar{L}}{\max(N_1,N_2)^{\frac{\alpha}{2}}}.
\end{equation*}
Substituting this in \eqref{eq:exp}, we obtain
\begin{equation*}
	\|uv\|_{L^2} \lesssim (L_1L_2)^{\frac{1}{2}} \frac{\min(N_1, N_2)^{\frac{1}{2}}} {\max(N_1,N_2)^{\frac{\alpha}{4}}} \|u\|_{L^2} \|v\|_{L^2}.
\end{equation*}
\end{proof}
\end{proposition}
\begin{remark}
The estimate \eqref{eq:BilinearToProve} remains true if we replace the functions on the left-hand side of \eqref{eq:BilinearToProve} by their complex conjugates.
\end{remark}

The next lemma allows us to handle the non-resonant case when the smallest frequency has size $\lesssim 1$.
\begin{lemma}
	Let $\alpha >0$, $N_1,N_2,N_3 \in 2^{\Z}$ be such that $N_1\ll N_2 \sim N_3$, and $L_1,L_2,L_3 \in 2^{\N_0}$. If $f_i: \R^3 \rightarrow \R_{+}, i=1,2,3$ have their Fourier supports in $\tilde{D}_{N_i,L_i}$, and  $\max(L_1,L_2,L_3)\gtrsim N_1N_2^{\alpha}$, we have
	\begin{equation}\label{eq:GuoNonReso}
		\int_{\R^3} (f_1\ast f_2)\cdot f_3 \lesssim  \frac{(L_1 L_2 L_3)^{1/2}}{\max(L_1,L_2,L_3)^{1/4}} N_2^{-\frac{\alpha}{2}} N_1^{\frac{1}{4}}\|f_1\|_{L^2} \|f_2\|_{L^2} \|f_3\|_{L^2}.
	\end{equation}
\end{lemma}
\begin{proof}
	The proof is a generalization of the proof of \cite[Lemma 3.1]{Guoetal} to the case $\alpha>2$. We provide the details for the sake of completeness. Define
	\begin{equation*}
		f_i^{\#}(\tau,\xi,\eta) := f_i(\tau+\omega_{\alpha}(\xi,\eta),\xi,\eta), i=1,2,3.
	\end{equation*}
Then, for $i=1,2,3$, $\|f_i^{\#}\|_{L^2} = \|f_i\|_{L^2}$ and $f_i^{\#}$ are functions supported in $\{(\tau_i,\xi_i,\eta_i): |\tau_i|\sim L_i, (\xi,\eta) \in A_{N_i} \}$. The left-hand side of \eqref{eq:GuoNonReso} can be bounded by
\begin{equation*}
\begin{split}
	&\int_{\R^6} f_1^{\#}(\tau_1,\xi_1,\eta_1)f_2^{\#}(\tau_2,\xi_2,\eta_2) \\ 
	&\qquad \times f_3^{\#}(\tau_1+\tau_2-\Omega_\alpha(\xi_1,\xi_2,\eta_1,\eta_2),\xi_1+\xi_2,\eta_1+\eta_2) \prod_{i=1}^2 d\tau_i d\xi_i d\eta_i.
	\end{split}
\end{equation*}
By using Cauchy-Schwarz inequality, it is sufficient to prove
\begin{equation}\label{eq:reduced}
	\begin{split}
	\int_{\R^4} g_1(\xi_1,\eta_1)&g_2(\xi_2,\eta_2)g(-\Omega_\alpha(\xi_1,\xi_2,\eta_1,\eta_2),\xi_1+\xi_2,\eta_1+\eta_2) d\xi_1 d\xi_2 d\eta_1 d\eta_2\\
	& \lesssim L_{\max}^{\frac{1}{4}} N_2^{-\frac{\alpha}{2}} N_1^{\frac{1}{4}}\|g_1\|_{L^2}\|g_2\|_{L^2}\|g\|_{L^2},
	\end{split}
\end{equation}
where $g_i:\R^2 \rightarrow \R_{+}$ are $L^2$ functions supported in $\tilde{A}_{N_i}$, $i=1,2$ and $g:\R^3\rightarrow \R_{+}$ is an $L^2$ function supported in $[-L_{\max}, L_{\max}]\times \tilde{A}_{N_3}$. After a change of variables,
\begin{equation*}
	\xi_2 \to \xi_2-\xi_1, \quad \eta_2 \to  \eta_2-\eta_1,
\end{equation*}
and using the Cauchy-Schwarz inequality, we find that the left-hand side of \eqref{eq:reduced} is dominated by
\begin{equation*}
	\begin{split}
	& \int_{\R^4} g_1(\xi_1,\eta_1)g_2(\xi_2-\xi_1, \eta_2-\eta_1)g(-\Omega_\alpha(\xi_1,\xi_2-\xi_1,\eta_1,\eta_2-\eta_1),\xi_2,\eta_2) \prod_{i=1}^2 d\xi_i d\eta_i \\
	&\lesssim \int_{\R^2} \Big( \int_{\R^2} |g_1(\xi_1,\eta_1)g_2(\xi_2-\xi_1,\eta_2-\eta_1)|^2 d\xi_1d\eta_1 \Big)^{\frac{1}{2}}  \\
	&\quad \times \Big( \int_{\R^2} |g(-\Omega_\alpha(\xi_1,\xi_2-\xi_1,\eta_1,\eta_2-\eta_1),\xi_2,\eta_2)|^2 d\xi_1 d\eta_1\Big)^{\frac{1}{2}} d\xi_2 d\eta_2.
	\end{split}
	\end{equation*}
Define
\begin{equation*}
	\begin{split}
	\beta_1(\eta_1) &= -\Omega_\alpha(\xi_1,\xi_2-\xi_1,\eta_1,\eta_2-\eta_1) \\
	 &= |\xi_2-\xi_1|^{\alpha}(\xi_2-\xi_1) -|\xi_2|^{\alpha}\xi_2 + |\xi_1|^{\alpha}\xi_1 + \frac{(\eta_1\xi_2-\eta_2\xi_1)^2}{\xi_1\xi_2(\xi_2-\xi_1)},\\
	 \beta_2(\xi_1) &= -(|\xi_2-\xi_1|^{\alpha}(\xi_2-\xi_1) -|\xi_2|^{\alpha}\xi_2 + |\xi_1|^{\alpha}\xi_1).
	 \end{split}
\end{equation*}
We have $|\beta_1| \lesssim L_{\max}$, $|\beta_2|\lesssim L_{\max}$, $|\beta_2|\lesssim N_2^{\alpha}N_1$ and using 
\begin{equation*}
	\beta_1+\beta_2 =\frac{(\eta_1\xi_2-\eta_2\xi_1)^2}{\xi_1\xi_2(\xi_2-\xi_1)},
\end{equation*}
we have
\begin{equation*}
	d\xi_1 d\eta_1 = \frac{\xi_1^{\frac{1}{2}}(\xi_2-\xi_1)^{\frac{1}{2}}}{2(\alpha+1)(\beta_1+\beta_2)^{\frac{1}{2}}\xi_2^{\frac{1}{2}} [(\xi_2-\xi_1)^{\alpha} - \xi_1^{\alpha}]} d\beta_1d\beta_2.
\end{equation*}
Using $|\xi_2| \sim |\xi_2-\xi_1|$ and Fubini, we get 
\begin{equation*}
	\Big( \int_{\R^2} |g(\Omega_\alpha(\xi_1,\xi_2-\xi_1,\eta_1,\eta_2-\eta_1),\xi_2,\eta_2)|^2 d\xi_1 d\eta_1\Big)^{\frac{1}{2}} \lesssim \frac{N_1^{\frac{1}{4}}L_{\max}^{\frac{1}{4}}}{N_2^{\frac{\alpha}{2}}} \|g(\cdot, \xi_2,\eta_2)\|_{L^2}.
\end{equation*}
This completes the proof.
\end{proof}

	\section{Quasilinear well-posedness}	
\label{section:Quasilinear}	
This section is devoted to the proof of the theorem below, which yields Theorem \ref{thm:QuasilinearLWP}.
\begin{theorem}
\label{thm:QuasilinearPrecise}
Let $\alpha \in (2,\frac{5}{2}]$, $u_0 \in H^{\infty,0}(\R^2)$, and $s>5-2\alpha$. Then, there exists continuous $T = T(\|u_0\|_{H^{s,0}(\R^2)})>0$ such that there is a unique solution 
\begin{equation}
\label{eq:DataSolutionMapping}
    u=S_T^{\infty}(u_0) \in C([-T,T];H^{\infty,0}(\R^2))
\end{equation}
of \eqref{eq:fKPI}. In addition for $s'\geqslant s$
\begin{equation*}
    \sup_{|t|\leqslant T}\|S^{\infty}_T(u_0)(t)\|_{H^{s',0}(\R^2)} \lesssim C(T,s',\|u_0\|_{H^{s',0}(\R^2)}).
\end{equation*}
Moreover, the mapping given by \eqref{eq:DataSolutionMapping}
extends uniquely to a continuous mapping 
\begin{equation*}
    S_T^{s'}:H^{s',0}(\R^2) \to C([-T,T];H^{s',0}(\R^2)).
\end{equation*}
\end{theorem}
Existence of local-in-time solutions for initial data in $H^{2,0}$ to the KP-I equation was proved by Molinet--Saut--Tzvetkov \cite{MolinetSautTzvetkov2007}. The proof is a non-trivial variant of the energy method, which relies on commutator estimates. Also, persistence of regularity is discussed in \cite{MolinetSautTzvetkov2007}. These arguments transpire to the fKP-I case and show the existence of a mapping $S_T^\infty$.

		\subsection{Short-time bilinear estimates}
	In this subsection, we prove short-time bilinear estimates which we need to control the nonlinearity. 
	\begin{proposition}\label{prop:ShortTimeBilinear}
		Let $2<\alpha\leqslant \frac{5}{2}$,  $T\in(0,1]$. There is $\varepsilon = \varepsilon(\alpha)$ such that for the time localization $T(N) = N^{-(5-2\alpha)-\varepsilon}$, such that for $s'\geqslant 0$, and $ u,v \in F^{s',0}(T)$, the following estimate holds:
			\begin{equation}
				\| \partial_x(uv)\|_{\mathcal{N}^{s',0}(T)} \lesssim  \|u\|_{F^{0,0}(T)} \|v\|_{F^{s',0}(T)} + \|v\|_{F^{0,0}(T)} \|u\|_{F^{s',0}(T)}.
			\end{equation}
	\end{proposition}

\begin{remark}
	As a particular case of the above proposition, we obtain
	\begin{equation*}
		\| \partial_x(uv)\|_{\mathcal{N}^{0,0}(T)} \lesssim \| u\|_{F^{0,0}(T)} \|v\|_{F^{0,0}(T)}.
	\end{equation*}
\end{remark}

Proposition  \ref{prop:ShortTimeBilinear} will be proved by means of dyadic estimates which we prove in the following. We first consider the $High \times Low \rightarrow High$ interaction. In this case, we can choose the time localization $T(N) = N^{-(5-2\alpha)-\varepsilon}$ for any $\varepsilon>0$ to prove a favorable estimate.
\begin{lemma}\label{lemma:fKPHL}
	Let $\varepsilon > 0$ and the time localization be given by $T(N) = N^{-(5-2\alpha)-\varepsilon}$.
	Let $N_1,N_2,N \in 2^{\N_0}$ be such that $N_2 \ll N_1 \sim N$ and $u_{N_1}\in F_{N_1}$, $v_{N_2}\in F_{N_2}$. Then, the following estimate holds for some $c(\varepsilon) > 0$:
	\begin{equation}\label{eq:fKPHL}
		\| P_N (\partial_x(u_{N_1}v_{N_2}))\|_{\mathcal{N}_N} \lesssim N_1^{-c(\varepsilon)}\|u_{N_1}\|_{F_{N_1}} \|v_{N_2}\|_{F_{N_2}}.
	\end{equation}
\begin{proof}
	Using the definition of the $\mathcal{N}_N$ norm, we can bound the left-hand side of \eqref{eq:fKPHL} by
	\begin{equation*}
	\begin{split}
		&\; \sup_{t_N \in \R} \| (\tau-\omega_{\alpha}(\xi,\eta)+iN^{(5-2\alpha)+\varepsilon})^{-1} N \mathbf{1}_{A_N}(\xi) \mathcal{F}[u_{N_1}\cdot \eta_0(N^{(5-2\alpha)+\varepsilon}(t-t_N))] \\
		&\quad \ast \mathcal{F}[v_{N_2} \cdot \eta_0(N^{(5-2\alpha)+\varepsilon}(t-t_N))]\|_{X_N}
		\end{split}
	\end{equation*}
Let 
\begin{equation*}
	f_{N_1}:= \mathcal{F}[u_{N_1} \cdot\eta_0(N^{(5-2\alpha)+\varepsilon}(t-t_N)) ] \quad \text{ and }\quad  g_{N_2}:= \mathcal{F}[v_{N_2} \cdot \eta_0(N^{(5-2\alpha)+\varepsilon}(t-t_N))].
\end{equation*}
Using the properties \eqref{prop} and \eqref{eq:TimeMult}, it suffices to prove that if $L_1,L_2 \geqslant N^{(5-2\alpha)+\varepsilon}$ and 
\begin{equation*}
	f_{N_1,L_1}, g_{N_2,L_2}:\R \times \R^2 \rightarrow \R_{+}
\end{equation*}
are functions supported in $D_{N_1,L_1}$ and $D_{N_2,L_2}$ and for $L_i = N^{(5-2\alpha)+\varepsilon}$ in  $D_{N_i,\leqslant L_i}$, respectively, then
\begin{equation}\label{eq:ToProveHL}
	N \sum_{L\geqslant N^{(5-2\alpha)+\varepsilon}} L^{-\frac{1}{2}}\| \mathbf{1}_{D_{N,\leqslant L}} (f_{N_1,L_1}\ast g_{N_2,L_2})\|_{L^2} \lesssim N_1^{-c(\varepsilon)} L_1^{\frac{1}{2}} \|f_{N_1,L_1}\|_{L^2} ~L_2^{\frac{1}{2}}\|g_{N_2,L_2}\|_{L^2}.
\end{equation}
We also note that by duality, it suffices to prove:
\begin{equation}\label{eq:DualHL}
	\Big | \int (f_{N_1,L_1} \ast g_{N_2,L_2})\cdot h_{N,L}\Big| \lesssim N_1^{-1-c(\varepsilon)} L_1^{\frac{1}{2}}\|f_{N_1,L_1}\|_{L^2}~ L_2^{\frac{1}{2}}\|g_{N_2,L_2}\|_{L^2}~L^{\frac{1}{2}-}\|h_{N,L}\|_{L^2},
\end{equation}
where $h_{N,L}$ is supported in $D_{N,L}$.\\
Let $L_{\max} = \max(L_1,L_2,L)$. In case $N_2 = 1$, we make an additional dyadic decomposition in the low frequencies. Now we abuse notation, and let $N_2 \in 2^{\Z}$ denote the dyadic frequency. We consider two cases:

\medskip

\noindent $\bullet$ \underline{$L_{\max} \leqslant N_2 N_1^{\alpha}$}:  For the case $N_2\gtrsim N_1^{(3-\frac{3\alpha}{2})+\varepsilon_2}$, using the estimate \eqref{eq:LW}, the left-hand side of \eqref{eq:DualHL} can be bounded by
\begin{equation*}
	\begin{split}
&~\quad N_1^{-\frac{3}{4}\alpha +\frac{1}{2}} N_2^{-\frac{1}{2}} L_1^{\frac{1}{2}} \|f_{N_1,L_1}\|_{L^2} ~L_2^{\frac{1}{2}} \|g_{N_2,L_2}\|_{L^2} L^{\frac{1}{2}}\|h_{N,L}\|_{L^2} \\
&\lesssim N_1^{-1-\varepsilon_2/2}L_1^{\frac{1}{2}} \|f_{N_1,L_1}\|_{L^2} ~L_2^{\frac{1}{2}} \|g_{N_2,L_2}\|_{L^2} L^{\frac{1}{2}} \|h_{N,L}\|_{L^2}.
\end{split}
\end{equation*}
To decrease the power of $L$ by $\varepsilon_3$, we use that $L \leqslant N_1^{\alpha+1}$, which yields
\begin{equation*}
\lesssim N_1^{-1-\varepsilon_2/2+(\alpha+1) \varepsilon_3 } L_1^{\frac{1}{2}} \| f_{N_1,L_1} \|_{L^2} ~L_2^{\frac{1}{2}} \| g_{N_2,L_2} \|_{L^2}  L^{\frac{1}{2}-\varepsilon_3} \| h_{N,L} \|_{L^2}.
\end{equation*}
This is acceptable choosing $\varepsilon_3 = \varepsilon_3(\varepsilon_2)$.

In the case $N_2\lesssim N_1^{(3-\frac{3\alpha}{2})+\varepsilon_2}$, using the bilinear Strichartz estimate \eqref{eq:BilinearToProve}, we have
\begin{equation*}
\text{LHS of }\eqref{eq:ToProveHL} \lesssim N N^{(\alpha-\frac{5}{2})-\frac{\varepsilon}{2}} \frac{N_2^{\frac{1}{2}}}{N_1^{\frac{\alpha}{4}}} L_1^{\frac{1}{2}} \|f_{N_1,L_1}\|_{L^2} L_2^{\frac{1}{2}} \|g_{N_2,L_2}\|_{L^2}.
\end{equation*}
We obtain
\begin{equation*}
\text{LHS of }\eqref{eq:ToProveHL} \lesssim N^{\frac{\varepsilon_2}{2}-\frac{\varepsilon}{2}} L_1^{\frac{1}{2}} \|f_{N_1,L_1}\|_{L^2} L_2^{\frac{1}{2}} \|g_{N_2,L_2}\|_{L^2},
\end{equation*}
which is acceptable for $N_2 \gtrsim 1$ choosing $\varepsilon_2<\varepsilon$.
If $N_2 \lesssim 1$, we interpolate with the estimate (note that the power of $N_2$ is positive, whereas the power of $N$ is negative)
\begin{equation*}
\text{LHS of }\eqref{eq:ToProveHL} \lesssim N_2^{\frac{1}{2}} N^{\frac{3\alpha}{4}-\frac{3}{2}-\frac{\varepsilon}{2}} L_1^{\frac{1}{2}} \|f_{N_1,L_1}\|_{L^2} L_2^{\frac{1}{2}} \|g_{N_2,L_2}\|_{L^2}.
\end{equation*}
to find
\begin{equation*}
\text{LHS of }\eqref{eq:ToProveHL} \lesssim N_2^{c_1(\varepsilon)} N_1^{-c_2(\varepsilon)} L_1^{\frac{1}{2}} \|f_{N_1,L_1}\|_{L^2} L_2^{\frac{1}{2}} \|g_{N_2,L_2}\|_{L^2}
\end{equation*}
for $c_1,c_2 > 0$, which is acceptable for $N_2 \lesssim 1$ because
the additional factor $N_2^{c_1(\varepsilon)}$ can be used to carry out the summation in $N_2$.

\medskip

\noindent $\bullet$ \underline{$L_{\max} \geqslant N_1^{\alpha}N_2$}: In the case $N_2\gtrsim 1$, we assume that $L \geqslant N_1^\alpha N_2$ (other cases give improved estimates). We use \eqref{eq:StriEmbed} as follows:
\begin{equation*}
	\begin{split}
	\text{LHS of }\eqref{eq:ToProveHL}	&\lesssim NN_1^{-\frac{\alpha}{2}} N_2^{-\frac{1}{2}} \|\mathcal{F}^{-1}(f_{N_1,L_1})\|_{L^4} \|\mathcal{F}^{-1}(g_{N_2,L_2})\|_{L^4} \\
		&\lesssim N N_1^{\frac{2-\alpha}{8}} N_2^{\frac{2-\alpha}{8}} N_1^{-\frac{\alpha}{2}} N_2^{-\frac{1}{2}} L_1^{\frac{1}{2}} \|f_{N_1,L_1}\|_{L^2} L_2^{\frac{1}{2}} \|g_{N_2,L_2}\|_{L^2} \\
		&=N_1^{\frac{5}{4}-\frac{5\alpha}{8}} N_2^{-\frac{1}{4}-\frac{\alpha}{8}}
	 L_1^{\frac{1}{2}} \|f_{N_1,L_1}\|_{L^2} L_2^{\frac{1}{2}} \|g_{N_2,L_2}\|_{L^2}.
	\end{split}
\end{equation*}
In case $N_2\lesssim 1$, we assume $L_{\max} = L$ (the other cases are similar). We use the estimate \eqref{eq:GuoNonReso}:
\begin{equation*}
	\begin{split}
	&\; \quad \Big| \int (f_{N_1,L_1}\ast g_{N_2,L_2})\cdot h_{N,L}\Big| \\
	 &\lesssim N_2^{\frac{1}{4}}N_1^{-\frac{\alpha}{2}} (L_1L_2)^{\frac{1}{2}} L^{\frac{1}{4}}\|f_{N_1,L_2}\|_{L^2}\|g_{N_2,L_2}\|_{L^2}\|h_{N,L}\|_{L^2}\\
	&\lesssim N_2^{\frac{1}{4}} N_1^{-\frac{\alpha}{2}}N_2^{-\frac{1}{4}+}N_1^{-\frac{\alpha}{4}+}L_1^{\frac{1}{2}}\|f_{N_1,L_1}\|_{L^2}~L_2^{\frac{1}{2}}\|g_{N_2,L_2}\|_{L^2}~L^{\frac{1}{2}-}\|h_{N,L}\|_{L^2}\\
	&\lesssim N_1^{-\frac{3\alpha}{4}+} N_2^{0+}L_1^{\frac{1}{2}}\|f_{N_1,L_1}\|_{L^2}~L_2^{\frac{1}{2}}\|g_{N_2,L_2}\|_{L^2}~L^{\frac{1}{2}-}\|h_{N,L}\|_{L^2}, 
	\end{split}
\end{equation*}
which is sufficient  to obtain \eqref{eq:fKPHL} after summing up.
	\end{proof}
	\end{lemma}

Next, we consider the $High \times High \rightarrow Low$ interaction. In this case we have to increase time localization to match the localization of the input frequencies. This will give a constraint on $\varepsilon$ because the larger $\varepsilon$ becomes, the more we lose when adding time localization.
\begin{lemma}\label{lemma:fKPHH}
For any $\alpha \in (2,\frac{5}{2}]$ there is $\varepsilon(\alpha) >0$ such that for the time localization $T=T(N)=N^{-(5-2\alpha)-\varepsilon}$ we have the following:
Let $N_1,N_2, N \in 2^{\N_0}$ be such that $N_1 \geqslant 2^{10}$, $N \ll N_1 \sim N_2$, and $u_{N_1}\in F_{N_1}, v_{N_2}\in F_{N_2}$. Then, the following estimate holds for some $c(\varepsilon) > 0$:
\begin{equation}\label{eq:ToProveHH}
	\|P_{N}(\partial_x(u_{N_1}v_{N_2}))\|_{\mathcal{N}_N} \lesssim N_1^{-c(\varepsilon)} \|u_{N_1}\|_{F_{N_1}} \|v_{N_2}\|_{F_{N_2}}.
\end{equation}
\end{lemma}
\begin{proof}
	Let $\gamma: \R \rightarrow [0,1]$ be a smooth function supported in $[-1,1]$ such that 
	\begin{equation*}
		\sum_{n \in \Z} \gamma^2(t-n)\equiv 1, \quad t\in \R.
	\end{equation*}
We need to further localize the nonlinearity to intervals of size $N_1^{(2\alpha-5)-}$. Moreover, we carry out an additional decomposition in case $N=1$ into very low frequencies. By abuse of notation, let now $N \in 2^{\Z}$ and $N_+ = \max(N, 1)$. Using the definition of the $\mathcal{N}_N$ norm, the left-hand side of \eqref{eq:ToProveHH} is dominated by
\begin{equation}
	\begin{split}
 &\quad \sup_{t_N \in \R}\big\|(\tau-\omega_{\alpha}(\xi,\eta)+iN_+^{(5-2\alpha)+\varepsilon})^{-1}N \mathbf{1}_{\tilde{A}_N}(\xi,\eta) \\
 &\sum_{|m|\lesssim (\frac{N_1}{N_+})^{(5-2\alpha)+\varepsilon}} \mathcal{F}[u_{N_1}\cdot \eta_0(N_{+}^{(5-2\alpha)+\varepsilon}(t-t_N))\gamma(N_1^{(5-2\alpha)+\varepsilon} (t-t_N)-m)]\\
	 &\ast \mathcal{F}[v_{N_2}\cdot \eta_0(N_{+}^{(5-2\alpha)+\varepsilon}(t-t_N))\gamma(N_1^{(5-2\alpha)+\varepsilon} (t-t_N)-m)]\big\|_{X_N}.
	\end{split} 
\end{equation}
Hence, it suffices to prove that if $f_{N_1,L_1}, g_{N_2,L_2}: \R \times \R^2\rightarrow \R_{+}$ are functions supported in $D_{N_1,L_1}$ and $D_{N_2,L_2}$, respectively, then
\begin{equation}\label{eq:HHSimplified}
\begin{split}
	&\quad N \Big(\frac{N_1}{N_{+}}\Big)^{(5-2\alpha)+\varepsilon} \sum_{L\geqslant N_+^{(5-2\alpha)+\varepsilon}} L^{-\frac{1}{2}} \|\mathbf{1}_{D_{N,L}}(f_{N_1,L_1} \ast g_{N_2,L_2})\|_{L^2} \\
	 &\lesssim N_1^{-c_1(\varepsilon)} [\min(N, 1)]^{c_2(\varepsilon)} L_1^{\frac{1}{2}}\|f_{N_1,L_1}\|_{L^2}~L_2^{\frac{1}{2}}\|g_{N_2,L_2}\|_{L^2}.
	 \end{split}
\end{equation}
By duality, it suffices to show:
\begin{equation}\label{eq:DualEqHH}
\begin{split}
 	N \Big( \frac{N_1}{N_+} \Big)^{(5-2\alpha)+\varepsilon} \Big | \int (f_{N_1,L_1}\ast g_{N_2,L_2})\cdot h_{N,L}\Big| &\lesssim N_1^{-c_1(\varepsilon)} [\min(N, 1)]^{c_2(\varepsilon)}~(L_1L_2L^{1-})^{\frac{1}{2}} \\
	&\quad \|f_{N_1,L_1}\|_{L^2}\|g_{N_2,L_2}\|_{L^2}\|h_{N,L}\|_{L^2},
	\end{split}
\end{equation}
where $h_{N,L}$ is supported in $\tilde{D}_{N,L}$.
For $L_{\max}=\max(L,L_1,L_2)$, we again consider two cases:

\noindent $\bullet$ \underline{$L_{\max} \leqslant N_1^{\alpha}N$}:  We consider two subcases:

$\star~ N \lesssim N_1^{3-\frac{3\alpha}{2}}$: Note that since $\alpha>2$, we have $N\ll 1$. We assume $L_{\max} = L_2$ (if $L_{\max}=L$, we obtain the same estimate without using the dual term). Using the bilinear Strichartz estimate \eqref{eq:BilinearToProve}, we find $\varepsilon_2$, $\varepsilon_3 > 0$ for any $\varepsilon>0$ such that
\begin{equation*}
\begin{split}
 &~\quad \Big| \int (f_{N_1,L_1}\ast g_{N_2,L_2})\cdot h_{N,L}\Big| \\
 &\lesssim \frac{N^{\frac{1}{2}}}{N_1^{\frac{\alpha}{4}}}(LL_1)^{\frac{1}{2}}\|f_{N_1,L_1}\|_{L^2} \|h_{N,L}\|_{L^2} \|g_{N_2,L_2}\|_{L^2}\\
 &\lesssim N^{\frac{1}{2}} N_1^{\frac{3\alpha}{4}-\frac{5}{2}-\varepsilon_2} (L_1L_2L^{1-\varepsilon_3})^{\frac{1}{2}} \|f_{N_1,L_1}\|_{L^2}\|g_{N_2,L_2}\|_{L^2}\|h_{N,L}\|_{L^2}.
 \end{split}
\end{equation*}
This gives
\begin{equation*}
\begin{split}
&\quad N N_1^{(5-2\alpha)+\varepsilon} \Big | \int (f_{N_1,L_1}\ast g_{N_2,L_2})\cdot h_{N,L}\Big| \\
 &\lesssim 
N^{\frac{3}{2}} N_1^{\frac{5}{2}-\frac{5 \alpha}{4} - \varepsilon_2 + \varepsilon} (L_1L_2L^{1-\varepsilon_3})^{\frac{1}{2}} \|f_{N_1,L_1}\|_{L^2}\|g_{N_2,L_2}\|_{L^2}\|h_{N,L}\|_{L^2},
\end{split}
\end{equation*}
which is sufficient since $\alpha > 2$, which leads to a negative power of $N_1$ and summation over $N$ is possible because $N \ll 1$.

$\star~ N \gtrsim N_1^{3-\frac{3\alpha}{2}}$: For $N \lesssim 1$, we use \eqref{eq:LW} to bound \eqref{eq:DualEqHH} to obtain
\begin{equation*}
\begin{split}
&~\quad N N_1^{(5-2\alpha+\varepsilon)} N_1^{-\frac{3\alpha}{4}+\frac{1}{2}} N^{-\frac{1}{2}} (L_1 L_2 L)^{\frac{1}{2}} \|f_{N_1,L_1}\|_{L^2} \|g_{N_2,L_2}\|_{L^2} \|h_{N,L}\|_{L_2}\\
&\lesssim N^{\frac{1}{2}} N_1^{\frac{11}{2}-\frac{11 \alpha}{4} + \varepsilon} (L_1 L_2 L)^{\frac{1}{2}} \|f_{N_1,L_1}\|_{L^2} \|g_{N_2,L_2}\|_{L^2} \|h_{N,L}\|_{L^2}\\ 
&\lesssim N^{\frac{1}{2}} N_1^{\frac{11}{2}-\frac{11 \alpha}{4} + \varepsilon + (\alpha + 1)\varepsilon_3} (L_1 L_2 L^{1-\varepsilon_3})^{\frac{1}{2}} \|f_{N_1,L_1}\|_{L^2} \|g_{N_2,L_2}\|_{L^2} \|h_{N,L}\|_{L^2}.
\end{split}
\end{equation*}
This is acceptable if we choose $0<\varepsilon< \frac{11 \alpha}{4} - \frac{11}{2}$ and $\varepsilon_3 = \varepsilon_3(\varepsilon)$.

If $N\gtrsim 1$, we use \eqref{eq:LW} to bound the left-hand side of \eqref{eq:DualEqHH} by
\begin{equation*}
\begin{split}
&~\quad N \big( \frac{N_1}{N} \big)^{(5-2\alpha)+\varepsilon} N_1^{-\frac{3 \alpha}{4} + \frac{1}{2}} N^{-\frac{1}{2}} (L_1L_2L)^{\frac{1}{2}}\|f_{N_1,L_1}\|_{L^2} \|g_{N_2,L_2}\|_{L^2} \|h_{N,L}\|_{L_2}\\
&\lesssim N^{(2\alpha - 4 - \varepsilon)} N_1^{\frac{11}{2}-\frac{11 \alpha}{4} + \varepsilon} (L_1 L_2 L)^{\frac{1}{2}} \|f_{N_1,L_1}\|_{L^2} \|g_{N_2,L_2}\|_{L^2} \|h_{N,L}\|_{L_2} \\
&\lesssim N_1^{\frac{3}{2}-\frac{3\alpha}{4}} (L_1 L_2 L^{1-\varepsilon_3})^{\frac{1}{2}} \|f_{N_1,L_1}\|_{L^2} \|g_{N_2,L_2}\|_{L^2} \|h_{N,L}\|_{L_2} \\
&\lesssim N_1^{\frac{3}{2}-\frac{3\alpha}{4} + (\alpha+1) \varepsilon_3} (L_1 L_2 L^{1-\varepsilon_3})^{\frac{1}{2}} \|f_{N_1,L_1}\|_{L^2} \|g_{N_2,L_2}\|_{L^2} \|h_{N,L}\|_{L_2}.
	\end{split}
\end{equation*}
The above is sufficient for $\alpha>2$ by choosing $\varepsilon_3$ small enough.



\noindent $\bullet$ \underline{$L_{\max} \geqslant N_1^{\alpha}N$}: First, we consider the case $N\gtrsim 1$.\\
If $L \geqslant N_1^\alpha N$, we can apply two linear Strichartz estimates to find for the left-hand side of \eqref{eq:HHSimplified}:
\begin{equation*}
	\begin{split}
		&\lesssim N_1^{(5-2\alpha)+\varepsilon} N^{(-4+2\alpha)-\varepsilon}(N_1^\alpha N)^{-\frac{1}{2}} \|\mathcal{F}^{-1}(f_{N_1,L_1})\|_{L^4} \|\mathcal{F}^{-1}( g_{N_2,L_2})\|_{L^4}\\
		&\lesssim N_1^{(5-2\alpha)+\varepsilon}N^{(-4+2\alpha)-\varepsilon}N^{-\frac{1}{2}} N_1^{-\frac{\alpha}{2}} N_1^{\frac{2-\alpha}{4}} L_1^{\frac{1}{2}} \|f_{N_1,L_1}\|_{L^2}~L_2^{\frac{1}{2}}\|g_{N_2,L_2}\|_{L^2}\\
		&\lesssim N^{(-\frac{9}{2}+2\alpha)-\varepsilon}N_1^{(\frac{11}{2}-\frac{11\alpha}{4})+\varepsilon} (L_1L_2)^{\frac{1}{2}} \|f_{N_1,L_1}\|_{L^2}~\|g_{N_2,L_2}\|_{L^2}.
			\end{split}
\end{equation*}
For $-\frac{9}{2}+2 \alpha \leqslant 0$, we have to choose $0<\varepsilon<\frac{11 \alpha}{4}-\frac{11}{2}$ to find
\begin{equation*}
\lesssim N^{-\varepsilon} N_1^{-c(\varepsilon)} (L_1 L_2)^{\frac{1}{2}} \| f_{N_1,L_1} \|_{2} \|g_{N_2,L_2} \|_2,
\end{equation*}
which is sufficient. For $\alpha > \frac{9}{4}$, we find
\begin{equation*}
\begin{split}
&\lesssim N^{-\varepsilon} N_1^{(1-\frac{3 \alpha}{4})+\varepsilon}(L_1 L_2)^{\frac{1}{2}} \| f_{N_1,L_1} \|_{2} \|g_{N_2,L_2} \|_2 \\
&\lesssim N^{-\varepsilon} N_1^{-\frac{11}{16}+\varepsilon} (L_1 L_2)^{\frac{1}{2}} \| f_{N_1,L_1} \|_{2} \|g_{N_2,L_2} \|_2,
\end{split}
\end{equation*}
which is acceptable for $\varepsilon < \frac{11}{16}$.

For $L \leqslant N_1^\alpha N$, we suppose that $L_{\max}=L_2$ (note the symmetry between $L_{\max} = L_1$ and this case) and bound the left-hand side of \eqref{eq:DualEqHH} by
\begin{equation*}
	\begin{split}
		&~ \quad N \big( \frac{N_1}{N} \big)^{(5-2\alpha)+\varepsilon} \big| \int (f_{N_1,L_1} \ast g_{N_2,L_2}) h_{N,L} \big| \\
		&\lesssim N^{2 \alpha -4 - \varepsilon} N_1^{5-2\alpha +\varepsilon} \|\mathcal{F}^{-1}(f_{N_1,L_1})\|_{L^4} \|\mathcal{F}^{-1}(h_{N,L})\|_{L^4}\|g_{N_2,L_2}\|_{L^2}\\
		&\lesssim N^{2\alpha - 4 -\varepsilon} N^{\frac{2-\alpha}{8}} N_1^{5-2\alpha+\varepsilon} N_1^{\frac{2-\alpha}{8}} N_1^{-\frac{\alpha}{2}} N^{-\frac{1}{2}} (L_1 L_2 L)^{\frac{1}{2}} \| f_{N_1,L_1} \|_2 \| g_{N_2,L_2} \|_2 \| h_{N,L} \|_2 \\
		&\lesssim N^{\frac{15 \alpha}{8}-\frac{17}{4}-\varepsilon} N_1^{\frac{21}{4}-\frac{21 \alpha}{8} + \varepsilon} (L_1 L_2 L)^{\frac{1}{2}} \| f_{N_1,L_1} \|_2 \| g_{N_2,L_2} \|_2 \| h_{N,L} \|_2 
		\end{split}
\end{equation*}
For $\frac{15 \alpha}{8}-\frac{17}{4} \leqslant 0$, this gives
\begin{equation*}
\begin{split}
&\lesssim N_1^{\frac{21}{4}-\frac{21 \alpha}{8} + \varepsilon} (L_1 L_2 L)^{\frac{1}{2}} \| f_{N_1,L_1} \|_2 \| g_{N_2,L_2} \|_2 \| h_{N,L} \|_2 \\
&\lesssim N_1^{\frac{21}{4}-\frac{21 \alpha}{8} + \varepsilon + (\alpha +1)\varepsilon_3} (L_1 L_2 L^{1-\varepsilon_3})^{\frac{1}{2}} \| f_{N_1,L_1} \|_2 \| g_{N_2,L_2} \|_2 \| h_{N,L} \|_2,
\end{split}
\end{equation*}
which is acceptable if we choose $0<\varepsilon<\frac{21 \alpha}{8} - \frac{21}{4}$ and $\varepsilon_3 = \varepsilon_3(\varepsilon)$.

If $\frac{15 \alpha}{8}-\frac{17}{4}>0$, which means $\alpha > \frac{34}{15}$, we find
\begin{equation*}
\begin{split}
&\lesssim N^{-\varepsilon} N_1^{1-\frac{6 \alpha}{8} + \varepsilon} (L_1 L_2 L)^{\frac{1}{2}} \| f_{N_1,L_1} \|_2 \| g_{N_2,L_2} \|_2 \| h_{N,L} \|_2 \\
&\lesssim N^{-\varepsilon} N_1^{-\frac{7}{10}+\varepsilon} (L_1 L_2 L)^{\frac{1}{2}} \| f_{N_1,L_1} \|_2 \| g_{N_2,L_2} \|_2 \| h_{N,L} \|_2 \\
&\lesssim N^{-\varepsilon} N_1^{-\frac{7}{10}+\varepsilon + (\alpha+1)\varepsilon_3} (L_1 L_2 L^{1-\varepsilon_3})^{\frac{1}{2}} \| f_{N_1,L_1} \|_2 \| g_{N_2,L_2} \|_2 \| h_{N,L} \|_2.
\end{split}
\end{equation*}
This is acceptable for $\varepsilon < 7/10$ and choosing $\varepsilon_3 = \varepsilon_3(\varepsilon)$ small enough.
\vspace*{0.3cm}

In case $N\lesssim 1$, we use the estimate \eqref{eq:GuoNonReso}:
\begin{equation*}
	\begin{split}
\Big | \int (f_{N_1,L_1}\ast g_{N_2,L_2}\cdot h_{N,L})\Big| &\lesssim N_1^{-\frac{\alpha}{2}} N^{\frac{1}{4}} L_{\max}^{-\frac{1}{4}}  L_1^{\frac{1}{2}}\|f_{N_1,L_1}\|_{L^2}~ L_2^{\frac{1}{2}}\|g_{N_2,L_2}\|_{L^2}~L^{\frac{1}{2}}\|h_{N,L}\|_{L^2}\\ &\lesssim N_1^{-\frac{3\alpha}{4}+} N^{0+} L_1^{\frac{1}{2}}\|f_{N_1,L_1}\|_{L^2}~ L_2^{\frac{1}{2}}\|g_{N_2,L_2}\|_{L^2}~L^{\frac{1}{2}-}\|h_{N,L}\|_{L^2},
		\end{split}
	\end{equation*}
which is sufficient to prove the required estimate.
\end{proof}

The case of three comparable frequencies is treated in the following lemma:
\begin{lemma}\label{lemma:fKPHHH}
	Let $\varepsilon > 0$ and the time localization given by $T=T(N)=N^{-(5-2\alpha)-\varepsilon}$. Let $N_1,N_2,N \in 2^{\N_0}$ be such that $N_1\sim N_2 \sim N \gg 1$. Let $u_{N_1}\in F_{N_1}$, $v_{N_2}\in F_{N_2}$. Then, we have, for any $\delta > 0$
	\begin{equation}
	\|P_{N}(\partial_x(u_{N_1}v_{N_2}))\|_{\mathcal{N}_N} \lesssim N_1^{(1-\frac{3\alpha}{4})+\delta} \|u_{N_1}\|_{F_{N_1}} \|v_{N_2}\|_{F_{N_2}}.
	\end{equation}
\begin{proof}
	Using the same reductions as in the previous lemmata, we require to show
	\begin{equation}\label{eq:HHH}
		N\sum_{L\geqslant N^{(5-2\alpha)+}} L^{-\frac{1}{2}}\|\mathbf{1}_{D_{N,L}}(f_{N_1,L_1}\ast g_{N_2,L_2})\|_{L^2}\lesssim N_1^{(1-\frac{3\alpha}{4})+} L_1^{\frac{1}{2}}\|f_{N_1,L_1}\|_{L^2}~L_2^{\frac{1}{2}}\|g_{N_2,L_2}\|_{L^2}.
	\end{equation}
For $L_{\max}$ as before, we consider:\\
$\bullet$ \underline{$L_{\max} \leqslant N_1^{\alpha+1}$}: We invoke duality and use \eqref{eq:LW}:
\begin{equation*}
\begin{split}
\Big | \int (f_{N_1,L_1}\ast g_{N_2,L_2})\cdot h_{N,L}\Big| &\lesssim N_1^{-\frac{3\alpha}{4}
} (L_1L_2L )^{\frac{1}{2}} \|f_{N_1,L_1}\|_{L^2} \|g_{N_2,L_2}\|_{L^2}\|h_{N,L}\|_{L^2}. \\ 
	\end{split}
\end{equation*}
$\bullet$ \underline{$L_{\max} \geqslant N_1^{\alpha+1}$}: For $L \geqslant N_1^{\alpha+1}$, using the $L^4$ Strichartz estimate and the size of $L$, we have
\begin{equation*}
\begin{split}
	\text{LHS of }\eqref{eq:HHH} &\lesssim N N^{-\frac{\alpha+1}{2}} N_1^{\frac{2-\alpha}{8}} N_2^{\frac{2-\alpha}{8}} L_1^{\frac{1}{2}}\|f_{N_1,L_1}\|_{L^2}~L_2^{\frac{1}{2}}\|g_{N_2,L_2}\|_{L^2} \\
	 &\lesssim N_1^{1-\frac{3\alpha}{4}} L_1^{\frac{1}{2}}\|f_{N_1,L_1}\|_{L^2}~L_2^{\frac{1}{2}}\|g_{N_2,L_2}\|_{L^2}.
	 \end{split}
\end{equation*}
For $N_1^{(5-2\alpha)+\varepsilon} \leqslant L \leqslant  N_1^{\alpha+1}$, we find the above estimate up to $N_1^{\delta}$ by two $L^4$ Strichartz estimates involving the dual function and a logarithmic summation loss.
\end{proof}
\end{lemma}

Finally, we consider the very low frequency case:
\begin{lemma}\label{lemma:fKPVLow}
	Let $\varepsilon > 0$ and the time localization given by $T=T(N)=N^{-(5-2\alpha)-\varepsilon}$.
	Let $N_1,N_2,N \in 2^{\N_0}$ be such that $N_1,N_2,N \lesssim 1$. Let $u_{N_1}\in F_{N_1}$ and $v_{N_2}\in F_{N_2}$. Then, we have
	\begin{equation}
		\|P_{N}(\partial_x(u_{N_1}v_{N_2}))\|_{\mathcal{N}_N} \lesssim \|u_{N_1}\|_{F_{N_1}} \|v_{N_2}\|_{F_{N_2}}.
	\end{equation}
\begin{proof}
	This estimate is a direct application of \eqref{eq:StriEmbed}. Using the definitions of the function spaces, it is sufficient to prove that for $L_1,L_2 \geqslant 1$ and $f_{N_1,L_1},g_{N_2,L_2}: \R\times \R^2 \rightarrow \R_{+}$, supported in $\tilde{D}_{N_1,L_1}, \tilde{D}_{N_2,L_2}$, respectively, we have
	\begin{equation}\label{eq:ToProveVeryLow}
		N \sum_{L\geqslant 1} L^{-\frac{1}{2}} \|\mathbf{1}_{\tilde{D}_{N,L}} (f_{N_1,L_1}\ast g_{N_2,L_2})\|_{L^2} \lesssim L_1^{\frac{1}{2}}\|f_{N_1,L_1}\|_{L^2} ~L_2^{\frac{1}{2}} \|g_{N_2,L_2}\|_{L^2}.
	\end{equation}
Using \eqref{eq:StriEmbed}, we have
	\begin{equation*}
\text{LHS of }\eqref{eq:ToProveVeryLow} \lesssim N L_1^{\frac{1}{2}}\|f_{N_1,L_1}\|_{L^2}~L_2^{\frac{1}{2}}\|g_{N_2,L_2}\|_{L^2},
	\end{equation*}
which is sufficient.
\end{proof}
	\end{lemma}

\begin{proof}[Proof of Proposition \ref{prop:ShortTimeBilinear}]
	Given $\alpha \in (2,\frac{5}{2}]$, we choose $\varepsilon = \varepsilon(\alpha)$ such that the estimate from Lemma \ref{lemma:fKPHH} is valid. Note that the $High \times High \rightarrow Low$ interaction is the only interaction, which imposes a constraint on time localization.
	We decompose the nonlinearity $\partial_x (uv)$ as follows:
	\begin{equation*}
	\begin{split}
		\partial_x(uv) &= \Big(\sum_{N_1\ll N_2\sim N} + \sum_{N_2\ll N_1 \sim N} + \sum_{N\ll  N_1\sim N_2} + \sum_{N_1\sim N_2\sim N \gg 1} +\sum_{N,N_1,N_2\lesssim 1}\Big) \\ 
		&\quad \times P_N\partial_x(P_{N_1}u \cdot P_{N_2}v).
		\end{split}
	\end{equation*}
Of the first two summands above, it is sufficient to consider the first by making the assumption that the derivative hits the high frequency. Each of the terms can be then separately handled by Lemma \ref{lemma:fKPHL}, Lemma \ref{lemma:fKPHH}, Lemma \ref{lemma:fKPHHH}, and Lemma \ref{lemma:fKPVLow}, respectively. We multiply each of the estimates in the lemmata by $N^{2s'}$ and sum up dyadically over the spatial frequencies to obtain the required estimates.
	
	\end{proof}

\subsection{Energy estimates}
We prove the energy estimates for the solution and the difference of the solutions in this section. The former is crucial to conclude an a priori estimate for the solution while the latter is required to prove the continuity of the data-to-solution map.

To begin, we assume that $T\in (0,1]$, $N_1,N_2,N_3 \in 2^{\Z}$ with $\max(N_i) \geqslant 1$, $u_i, \in F_{N_i}(T)$, $i=1,2,3$. Without any loss of generality, we assume that $N_1\leqslant N_2 \leqslant N_3$.  Let $\gamma: \R\rightarrow [0,1]$ denote a smooth function supported in $[-1,1]$ with the property that 
\begin{equation*}
	\sum_{n\in \Z} \gamma^3(t-n) = 1, \quad t\in \R.
\end{equation*}
We fix extensions $\tilde{u}_i$ of $u_i$ such that $\|\tilde{u}_i\|_{F_{N_i}} \leqslant 2\|u_i\|_{F_{N_i}(T)}$. Then, we use the function $\gamma$ to divide the time interval to sub-intervals of size $ N_3^{2\alpha-5-\varepsilon}$:
\begin{equation} \label{eq:NotationsEnergyEstimate1}
	\begin{split}
	&~\quad \Big|	\int_{[0,T]\times \R^2} u_1 u_2 u_3 ~dxdy dt \Big|\\
	& \lesssim \sum_{|n|\leqslant CN_3^{(5-2\alpha)+\varepsilon} }\Big| \int_{\R\times \R^2} (\gamma(N_3^{(5-2\alpha)+\varepsilon}t-n)\mathbf{1}_{[0,T]}(t)\tilde{u}_1) (\gamma(N_3^{(5-2\alpha)+\varepsilon}t-n)\mathbf{1}_{[0,T]}(t)\tilde{u}_2)\\
	&\hspace{4cm}\times (\gamma(N_3^{(5-2\alpha)+\varepsilon}t-n)\mathbf{1}_{[0,T]}(t)\tilde{u}_3)~dx dydt\Big|\\
	&=\sum_{|n|\leqslant CN_3^{(5-2\alpha)+\varepsilon} }\Big|\int_{\R\times \R^2} \big( \mathcal{F}((\gamma(N_3^{(5-2\alpha)+\varepsilon}t-n)\mathbf{1}_{[0,T]}(t)\tilde{u}_1))\\
	&\qquad \ast \mathcal{F}((\gamma(\ldots )\mathbf{1}_{[0,T]}(t)\tilde{u}_2) \big)(\xi,\eta,\tau) \mathcal{F}((\gamma(\ldots)\mathbf{1}_{[0,T]}(t)\tilde{u}_3))(\xi,\eta,\tau)~d\xi d\eta d\tau\Big|\\
	&=\sum_{|n|\leqslant CN_3^{(5-2\alpha)+\varepsilon}} \Big |\int_{\R\times\R^2} (f_1 \ast f_2 )\cdot f_3~ d\xi d\eta d\tau \Big|,
	\end{split}
\end{equation}
with
\begin{equation}
\label{eq:NotationsEnergyEstimate}
	f_i:= \mathcal{F}((\gamma(N_3^{(5-2\alpha)+\varepsilon}t-n)\mathbf{1}_{[0,T]}(t)\tilde{u_i})), \quad i=1,2,3.
\end{equation}
In the above summation over $n \in \Z$, we consider the two sets:
\begin{equation*}
	\begin{split}
	A &= \{n \in \Z, |n|\leqslant CN_3^{(5-2\alpha)+\varepsilon} : \gamma(N_3^{(5-2\alpha)+\varepsilon}t-n)\mathbf{1}_{[0,T]}(t) =\gamma(N_3^{(5-2\alpha)+\varepsilon}t-n)  \},\\
	 \quad A^c &= \{n \in \Z , |n|\leqslant CN_3^{(5-2\alpha)+\varepsilon}: 0 \in \text{supp} ( \gamma(N_3^{(5-2\alpha)+\varepsilon}\cdot-n)) \\
	 &\quad \vee T \in \text{supp} ( \gamma(N_3^{(5-2\alpha)+\varepsilon}\cdot -n)) \}.
\end{split}
\end{equation*}
Since $T \in (0,1]$ and $\gamma$ is supported in $[-1,1]$, we have that $|A| \lesssim N_3^{(5-2\alpha)+\varepsilon}$ while $|A^c| \leqslant 4$. On the physical side, the temporal support of $f_i$, $i=1,2,3$ is of size $\sim N_3^{(2\alpha-5)-\varepsilon}$. We can further decompose
\begin{equation}\label{eq:ModulationDecomposition}
	f_i = \sum_{L_i\geqslant N_3^{(5-2\alpha)+\varepsilon} }f_{i,L_i}.
\end{equation}
with\footnote{In \eqref{eq:ModulationDecomposition} we abuse notation: $L_i$ is a dyadic number, when we write $L_i = N_3^{\kappa}$ this refers to the largest dyadic number smaller than $N_3^\kappa$, as defined in \eqref{eq:ModulationDecompositionII}.}
\begin{equation}
\label{eq:ModulationDecompositionII}
\text{supp}(f_{i,L}) \subseteq
\begin{cases}
&\tilde{D}_{N_i, \leqslant L}, \; L= \max \{ \tilde{L} \in 2^{\N_0} : \tilde{L} \leqslant N_3^{5-2 \alpha + \varepsilon} \}, \\
&\tilde{D}_{N_i,L}, \; \text{else}.
\end{cases}
\end{equation}

In the following computations, we shall assume that we have already made the above reduction.  For $n \in A^c$, we use the following estimate to substitute for \eqref{prop} (cf. \cite[p.~291]{IKT}),
\begin{equation*}
 \sup_{L\geqslant 1}   L^{\frac{1}{2}} \|\eta_L(\tau-\omega_{\alpha}(\xi,\eta)) \cdot  f_N^{I}\|_{L^2} \lesssim \|f_N\|_{X_N},
\end{equation*}
where $f_N^{I} = \mathcal{F}(\mathbf{1}_I(t)f_{N}\cdot \mathcal{F}^{-1}(f_N))$ for an interval $I \subset \R$ (in our case $I$ is an interval of length $\min(1,N_{\max}^{(2\alpha-5)-\varepsilon})$. Since in the estimates below, we can spare a small power of $L_{\max}$ and gain a factor $N_3^{0+}$ compared to $N_3^{(5-2\alpha)+\varepsilon} = |A|$, we can also handle the contribution of $A^c$. We shall focus on $n \in A$ in the following.




\subsubsection{Energy estimate for the solution}

In this section we shall prove energy estimates
	\begin{equation}
	\label{eq:EnergyEstimateSolution}
		\|u\|_{E^{s',0}(T)}^2 \lesssim \|u_0\|_{H^{s',0}}^2 + \|u\|_{F^{s,0}(T)} \|u\|_{F^{s',0}(T)}^2.
	\end{equation}
	for solutions to \eqref{eq:fKPI} for some $s' \geqslant s \geqslant 0$ with $s = s(\alpha)$. If $\alpha$ is large enough, we can reach $s'=0$. Also, the time localization will depend on $\alpha$.

\begin{proposition}\label{prop:EnergyEstSol}
		Let $2<\alpha\leqslant \frac{5}{2}$ and  $T\in(0,1]$. Let $0<\varepsilon<\frac{21 \alpha}{8} - \frac{21}{4}$.
		\begin{itemize}
		\item $2<\alpha\leqslant \frac{24}{11}$: Then, for the time localization $T(N) = N^{-(5-2\alpha)-\varepsilon}$, $s' \geqslant s > 6 - \frac{11 \alpha}{4} + \varepsilon$, and $u \in F^{s',0}(T)$, the estimate \eqref{eq:EnergyEstimateSolution} holds for smooth solutions $u$ to \eqref{eq:fKPI}.
		\item $\frac{24}{11} < \alpha \leqslant \frac{5}{2}$: Additionally, we suppose $\varepsilon < \frac{11 \alpha}{4}-6$. Then, for the time localization $T(N)=N^{-(5-2\alpha)-\varepsilon}$, $s' \geqslant s \geqslant 0$, the estimate \eqref{eq:EnergyEstimateSolution} holds for smooth solutions $u$ to \eqref{eq:fKPI}.
		\end{itemize}

\begin{proof}
	We consider equation \eqref{eq:fKPI} on $(-T,T)\times \R^2$ for Littlewood-Paley pieces $P_Nu$. Multiplying this equation with $P_{N}u$ and integrating, we obtain
	\begin{equation}\label{eq:LPEnergyEstimate}
	\begin{split}
		\sup_{t_N\in [-T,T]}\|P_N u(t_N)\|_{L^2}^2 &\leqslant \|P_N u_0\|_{L^2}^2 \\
		&\qquad + \sup_{t_N\in [-T,T]} \Big| \int_{[0,t_N]\times \R^2}P_N u P_N(u\partial_x u)~dt dx dy\Big|.
	\end{split}
	\end{equation}
After proving suitable bounds for the last term, \eqref{eq:EnergyEstimateSolution} follows from multiplying \eqref{eq:LPEnergyEstimate} with $N^{2s'}$ and summation in $N$.

 We consider the integrand:
\begin{equation*}
	P_N u P_N(P_{N_1}u\cdot P_{N_2}\partial_xu).
\end{equation*}
Using the notation from \eqref{eq:NotationsEnergyEstimate1} and \eqref{eq:NotationsEnergyEstimate}, we define
\begin{equation}\label{eq:DefinitionOffi}
  \begin{split}
 f_1 &= \mathcal{F}(\gamma(N_{\max}^{(5-2\alpha)+\varepsilon}t-n)\mathbf{1}_{[0,T]}(t) P_{N_1} u),\\
  \tilde{f}_2 &= \mathcal{F}(\gamma(N_{\max}^{(5-2\alpha)+\varepsilon}t-n)\mathbf{1}_{[0,T]}(t) P_{N_2}\partial_x u), \\
 f_3 &= \mathcal{F}(\gamma(N_{\max}^{(5-2\alpha)+\varepsilon}t-n)\mathbf{1}_{[0,T]}(t) P_N u),
 \end{split}  
\end{equation}
and consider following cases: 
\begin{enumerate}[(i)]
	\item $N\ll N_1 \sim N_2$,
	\item $N_2 \ll N_1 \sim N$,
	\item $N_1 \ll N_2 \sim N$,
	\item $N\sim N_1 \sim N_2$.
\end{enumerate}
 As is the case with the bilinear estimates, at first, we have $N \in 2^{\mathbb{N}_0}$ to take into account the definition of the function spaces. For $N=1$, we carry out an additional dyadic decomposition into very low frequencies $N \in 2^{\Z}$ to take advantage of the derivative, which is smoothing for $N \ll 1$. Note that in the estimates proved below, we always have summability for $N \lesssim 1$.\\
 
 In case (i), the resultant frequency $N$ is low. We divide the time interval of integration into sub-intervals of size $\max(N_1,N_2,N)^{(2\alpha-5)-\varepsilon} \sim N_1^{(2\alpha-5)-\varepsilon}$ via the technique elucidated above. With $N \lesssim N_1$ and $L_{\max}=\max(L_1,L_2,L)$, we consider the following:\\
$\bullet$ \underline{$L_{\max}  \leqslant N_1^{\alpha}N$}: Using notation from \eqref{eq:NotationsEnergyEstimate} and \eqref{eq:DefinitionOffi}, we have
\begin{equation}\label{eq:CaseOne}
	\begin{split}
		&~\quad \sum_{n \in A} \Big | \int_{\R\times \R^2} (f_1 \ast \tilde{f}_2)\cdot f_3~ d\xi d\tau d\eta \Big|\\
		&\lesssim N N_1^{5-2\alpha+\varepsilon} \sup_{n \in A} \sum_{N_1^{5-2\alpha +\varepsilon} \leqslant L_i \leqslant N_1^\alpha N} \Big | \int_{\R\times \R^2} (f_{1,L_1} \ast f_{2,L_2})\cdot f_{3,L_3} d\xi d\eta d\tau \Big|,
	\end{split}
\end{equation}
where
\begin{equation*}
	f_2 = \mathcal{F}(\gamma(N_1^{(5-2\alpha)+\varepsilon}t-n)\mathbf{1}_{[0,T]}(t) P_{N_2} u)
\end{equation*}
and we use the notation from \eqref{eq:ModulationDecomposition} for the decomposition in modulation of $f_i$. Using the nonlinear Loomis--Whitney inequality \eqref{eq:LW}, we obtain that \eqref{eq:CaseOne} can be bounded by
\begin{equation}\label{eq:DecisiveEstimateSolution}
	\begin{split}
		&~\quad N N_1^{5-2\alpha+ \varepsilon}N_1^{-\frac{3\alpha}{4}+\frac{1}{2}}N^{-\frac{1}{2}} \prod_{i=1}^3 \sum_{L_i \geqslant N_1^{(5-2\alpha)+}} L_i^{\frac{1}{2}} \| f_{i,L_i} \|_{L^2} \\
		&\lesssim N^{\frac{1}{2}} N_1^{\frac{11}{2}-\frac{11\alpha}{4}+ \varepsilon} \|u_{N_1}\|_{F_{N_1}(T)}\|u_{N_2}\|_{F_{N_2}(T)}\|u_N\|_{F_N(T)}.
	\end{split}
\end{equation}
 Note that for $\varepsilon < \frac{21 \alpha}{8} - \frac{21}{4}$ the exponent of $N_1$ is negative. For $N \lesssim 1$, we have easy summation in $N$ for any $s' \geqslant s \geqslant 0$. For $\frac{24}{11} < \alpha \leq \frac{5}{2}$ and $\varepsilon$ according to the assumptions, we have
 \begin{equation*}
 \lesssim N^{6-\frac{11 \alpha}{4} + \varepsilon}\|u_{N_1}\|_{F_{N_1}(T)}\|u_{N_2}\|_{F_{N_2}(T)}\|u_N\|_{F_N(T)}.
\end{equation*}  
Since $6-\frac{11 \alpha}{4} + \varepsilon < 0$, we have easy summation for $s' \geqslant s \geqslant 0$ and $N \gtrsim 1$. For $2 < \alpha < \frac{24}{11}$, we estimate with easy summation:
 \begin{equation*}
 \lesssim N_1^{6-\frac{11 \alpha}{4} + \varepsilon}\|u_{N_1}\|_{F_{N_1}(T)}\|u_{N_2}\|_{F_{N_2}(T)}\|u_N\|_{F_N(T)}.
\end{equation*}

\begin{remark}
We note the bilinear Strichartz estimate \eqref{eq:BilinearToProve} gives the same result if $N \lesssim N_1^{3-\frac{3\alpha}{2}}$.
\end{remark}
$\bullet$ \underline{$L_{\max} \geqslant N_1^{\alpha}N$}: We assume that $L_{\max}= L_2$ (since the estimate below becomes better if $L_{\max} = L$). Using the same notation as in the previous case and the linear Strichartz estimate \eqref{eq:StriEmbed}, we obtain
\begin{equation*}
\begin{split}
		&~\quad \Big | \int_{[0,t_N] \times \R^2} P_N u P_N(P_{N_1}u \cdot P_{N_2}\partial_x u) \Big|\\
	&\lesssim  NN_1^{5-2\alpha +\varepsilon} \sum_{\substack{{L_1,L \geqslant N_1^{(5-2\alpha)+\varepsilon}}\\ {L_2 \geqslant N_1^\alpha N}}}  \Big| \int_{\R \times \R^2}  (f_{1,L_1} * f_{2,L_2}) \cdot f_{3,L_3}  ~d\tau d\xi d\eta \Big|\\
	&\lesssim N N_1^{5-2\alpha+\varepsilon} \sum_{\substack{{L_1,L \geqslant N_1^{(5-2\alpha)+\varepsilon}}\\ {L_2 \geqslant N_1^\alpha N}}} \| \mathcal{F}^{-1} (f_{1,L_1})\|_{L^4} \|\mathcal{F}^{-1}(f_{3,L_3})\|_{L^4} \|f_{2,L_2}\|_{L^2}\\
	&\lesssim N N_1^{5-2\alpha+\varepsilon} N^{\frac{2-\alpha}{8}} N_1^{\frac{2-\alpha}{8}} N_1^{-\frac{\alpha}{2}} N^{-\frac{1}{2}} \prod_{i=1}^3 \sum_{L_i \geqslant N_1^{(5-2\alpha)+\varepsilon}} L_i^{\frac{1}{2}} \| f_{i,L_i} \|_{L^2}\\
	&\lesssim N^{\frac{3}{4}-\frac{\alpha}{8}} N_1^{\frac{21}{4}-\frac{21\alpha}{8}+\varepsilon}\|u_N\|_{F_N(T)}\|u_{N_1}\|_{F_{N_1}(T)}\|u_{N_2}\|_{F_{N_2}(T)}.
\end{split}	
\end{equation*}
 Note that for $\varepsilon < \frac{21 \alpha}{8} - \frac{21}{4}$ the exponent of $N_1$ is negative. For $N \lesssim 1$, we have easy summation in $N$ for any $s' \geqslant s \geqslant 0$. In the following let $N \gtrsim 1$. For $\frac{24}{11} < \alpha < 4$ and $\varepsilon$ according to the assumptions, we have
 \begin{equation*}
 \lesssim N^{6-\frac{11 \alpha}{4} + \varepsilon}\|u_{N_1}\|_{F_{N_1}(T)}\|u_{N_2}\|_{F_{N_2}(T)}\|u_N\|_{F_N(T)}.
\end{equation*}  
Since $6-\frac{11 \alpha}{4} + \varepsilon < 0$, we have easy summation for $s' \geqslant s \geqslant 0$ and $N \gtrsim 1$. For $2 < \alpha < \frac{24}{11}$, we estimate
 \begin{equation*}
 \lesssim N_1^{6-\frac{11 \alpha}{4} + \varepsilon}\|u_{N_1}\|_{F_{N_1}(T)}\|u_{N_2}\|_{F_{N_2}(T)}\|u_N\|_{F_N(T)}
\end{equation*}
with easy summation.

Case (ii) can be handled in a similar way as case (i) as the derivative hits the low frequency. For case (iii), we use a commutator argument, see \cite[Lemma ~6.1]{IKT} and \cite[Remark~5.9]{KimSchippa2021} to transfer the derivative to the low frequency. We can then use the same argument as in case (i) to obtain the required estimate. Case (iv) can be handled similarly.
	\end{proof}
\end{proposition}



\subsubsection{Energy estimate for the difference equation}
Let $u_1, u_2$ solve \eqref{eq:fKPI} with initial data $\phi_1$ and $\phi_2$, respectively. The difference of the solutions, $v=u_1-u_2$ satisfies the following:
\begin{equation}\label{eq:DiffEq}
	\left\{ \begin{array}{cl}
		\partial_t v - D^{\alpha}_x\partial_{x} v - \partial_{x}^{-1} \partial_y^2 v &= \partial_x(v(u_1+u_2))/2, \quad (t,x,y) \in \R\times \R \times \R, \\ 
		v(0) &= \phi_1-\phi_2=:\phi.
	\end{array} \right.
\end{equation}
\begin{proposition}\label{prop:fKPEnergyEstDiff}
Let $2<\alpha \leqslant \frac{5}{2}$ , $T \in (0,1]$ and $u_1,u_2$ be solutions to \eqref{eq:fKPI} with initial data $\phi_1$ and $\phi_2$, respectively. Then, for $0<\varepsilon<\frac{21 \alpha}{8} - \frac{21}{4}$, $\varepsilon '> \varepsilon$, the time localization $ T(N) = N^{-(5-2\alpha)-\varepsilon}$, $s \geqslant 5-2\alpha+\varepsilon'$ and $v=u_1-u_2$ a solution to \eqref{eq:DiffEq}, the following estimates hold:
\begin{align}
	\label{eq:DifferenceSolutionEstimateI}
		\|v\|_{E^{0,0}(T)}^2 &\lesssim \|\phi\|_{L^2}^2 + \|v\|_{F^{0,0}(T)}^2 (\|u_1\|_{F^{s,0}(T)} + \|u_2\|_{F^{s,0}(T)}),\\
	\label{eq:DifferenceSolutionEstimateII}
	\|v\|_{E^{s,0}(T)}^2 &\lesssim \|\phi\|_{H^{s,0}}^2 + \|v\|_{F^{s,0}(T)}^3 \\
	&\qquad + (\|v\|_{F^{s,0}(T)}^2 \|u_2\|_{F^{s,0}(T)} + \|v\|_{F^{0,0}(T)}\|v\|_{F^{s,0}(T)}\|u_2\|_{F^{2s,0}(T)}). \nonumber 
\end{align} 
\begin{proof}
We use the fundamental theorem of calculus to obtain from \eqref{eq:DiffEq}:
	\begin{equation}\label{eq:LPDifferenceEnergyEstimate}
		\begin{split}
			\sup_{t_N\in [-T,T]}\|P_{N}v(t_N)\|_{L^2}^2 &\lesssim \|P_N \phi\|_{L^2}^2 \\ &\qquad + \sup_{t_N\in [-T,T]}\Big| \int_{[0,t_N]\times \R^2} P_Nv P_N(\partial_x(v(u_1+u_2)))~ dxdy dt\Big|.
		\end{split}
	\end{equation}
		We are required to handle the last term in the above display. For the proof of \eqref{eq:DifferenceSolutionEstimateI}, we treat the term $P_N v P_N(\partial_x(vu_1))$ since the second term, namely $P_N v P_N(\partial_x(vu_2))$ can be estimated similarly. We have
		\begin{equation}\label{eq:Term}
			\begin{split}
				P_N (vu_1) &= P_N (P_{\ll N} v \cdot u_1) + P_N(P_{\gtrsim N} v \cdot u_1)\\
				&\sim P_N(P_{\ll N}v \cdot P_N u_1) + P_N(P_{\gtrsim N}v \cdot P_{\gtrsim N}u_1).
			\end{split}
		\end{equation}
		Corresponding to the integrand in the last term of \eqref{eq:LPDifferenceEnergyEstimate}, we need to consider
		\begin{equation}\label{eq:Terms1}
			P_N v \cdot \partial_x P_N(P_{\ll N}v \cdot P_N u_1) = P_N v \cdot \partial_x  P_{\ll N}v \cdot P_N u_1 + P_N v \cdot P_{\ll N} v \cdot \partial_x P_N u_1
		\end{equation}
		and 
		\begin{equation}\label{eq:TermEasy}
			P_N v \cdot \partial_x P_N(P_{\gtrsim N}v \cdot P_{\gtrsim N} u_1 ) = \sum_{N_2 \sim N_1 \gtrsim N} P_N v \cdot \partial_x P_N (P_{N_1}v\cdot P_{N_2}u_1).
	\end{equation}

The first term on the right-hand side of \eqref{eq:Terms1} and \eqref{eq:TermEasy} can be estimated like in Proposition \ref{prop:EnergyEstSol} because the derivative hits the low frequency term. However, for the second term, the derivative hits the high frequency term and the resulting term is not amenable to an integration by parts argument to transfer the derivative to the low frequency term. We treat \eqref{eq:TermEasy} as follows: Fix extensions of $P_N v$, $P_{N_1}v$ and $P_{N_2}u_1$ and still denote them by $P_N v$, $P_{N_1}v$ and $P_{N_2}u_1$ to lighten the notation. Then, using Parseval's identity and using the reductions explained before, we have
\begin{align*}
	&~\qquad \Big |\int_{[0,t_N]\times \R^2} P_N v \cdot \partial_x P_N (P_{N_1}v \cdot P_{N_2} u_1)~dxdy dt \Big| \\
	&\lesssim  N\Big | \int_{\R \times \R^2} \widehat{P_{N}v} \cdot (\widehat{P_{N_1}v} \ast \widehat{P_{N_1}u_1}) ~d\xi d\eta d\tau \Big|\\
	& \lesssim N\sum_{|n|\leqslant CN_1^{(5-2\alpha)+\varepsilon}} \Big | \int_{\R\times \R^2} (f_1 \ast f_2)\cdot f_3 ~d\xi d\eta d\tau \Big|,
\end{align*}
where, now,
\begin{align*}
	f_1 &= \mathcal{F}(\gamma(N_1^{(5-2\alpha)+\varepsilon}t-n)\mathbf{1}_{[0,T]}(t) P_{N_1} v), \\
	f_2 &= \mathcal{F}(\gamma(N_1^{(5-2\alpha)+\varepsilon}t-n)\mathbf{1}_{[0,T]}(t) P_{N_2} u_1),\\
	f_3 &= \mathcal{F}(\gamma(N_1^{(5-2\alpha)+\varepsilon}t-n)\mathbf{1}_{[0,T]}(t)  P_N v).
\end{align*}
After summing up in $n$, we need to control the following term:
\begin{equation}\label{eq:TermToControl}
	N N_1^{(5-2\alpha)+\varepsilon} \Big | \int_{\R\times \R^2} (f_1 \ast f_2)\cdot f_3 ~d\xi d\eta d\tau  \Big|, \text{ where } N \lesssim N_1 \sim N_2.
\end{equation}
For the decomposition in modulation for functions $f_i$, $i=1,2,3$, we use the notation \eqref{eq:ModulationDecomposition} and consider the following cases:
\begin{itemize}
	\item $\underline{L_{\max} \leqslant NN_1^{\alpha}}: $ We further consider two subcases depending on the size of the high and low $x$ frequencies:\\
	$\star~N^{\frac{1}{2}}\leqslant N_1^{\frac{3}{2}-\frac{3\alpha}{4}}$: After decomposing the functions in modulation, an application of the bilinear Strichartz estimate \eqref{eq:BilinearToProve} to a high-low interaction gives
	\begin{equation}\label{eq:DecisiveEstimate}
		\begin{split}
			\eqref{eq:TermToControl}	& \lesssim NN_1^{(5-2\alpha)+\varepsilon} \sum_{L_i \leqslant NN_1^{\alpha}} 
			\Big| \int_{\R \times \R^2} (f_{1,L_1}\ast f_{2,L_2})\cdot f_{3,L_3} d\xi d\eta d\tau \Big|\\
			&\lesssim N N_1^{(5-2\alpha)+\varepsilon}  \frac{N^{\frac{1}{2}}}{N_1^{\frac{\alpha}{4}}} N_1^{\big(\frac{2\alpha -5}{2} \big)-\frac{\varepsilon}{2}} \prod_{i=1}^3 \sum_{L_i\geqslant N_1^{ (5-2\alpha)+\varepsilon}} L_i^{\frac{1}{2}} \| f_{i,L_i} \|_{L^2}.
		\end{split}
	\end{equation}
If $N \gtrsim 1$, we find
\begin{equation*}
\lesssim N_1^{7-\frac{7 \alpha}{2} + \frac{\varepsilon}{2}} \prod_{i=1}^3 \sum_{L_i\geqslant N_1^{ (5-2\alpha)+\varepsilon}} L_i^{\frac{1}{2}} \| f_{i,L_i} \|_{L^2}.
\end{equation*}	
This suffices for $\varepsilon < 7 \alpha - 14$, which is ensured by hypothesis. For $N \lesssim 1$, we interpolate \eqref{eq:DecisiveEstimate} with the estimate in the above display, to find
\begin{equation*}
\lesssim N^{c_1} N_1^{-c_2} \prod_{i=1}^3 \sum_{L_i\geqslant N_1^{ (5-2\alpha)+\varepsilon}} L_i^{\frac{1}{2}} \| f_{i,L_i} \|_{L^2},
\end{equation*}
which allows for summation in $N \lesssim 1$.\\
	$\star~ N_1^{\frac{3}{2}-\frac{3\alpha}{4}} \leqslant N^{\frac{1}{2}}$: In this case, an application of \eqref{eq:LW} gives 
	\begin{equation*}
			\eqref{eq:TermToControl}
			\lesssim N N_1^{(5-2\alpha)+\varepsilon}  N_1^{-\frac{3\alpha}{4}+\frac{1}{2}} N^{-\frac{1}{2}} \prod_{i=1}^3 \sum_{L_i \geqslant N_1^{(5-2\alpha)+}} L_i^{\frac{1}{2}} \| f_{i,L_i} \|_{L^2}.
	\end{equation*}
	If $N \gtrsim 1$, we find
	\begin{equation*}
	\lesssim N_1^{6-\frac{11 \alpha}{4} + \varepsilon} \prod_{i=1}^3 \sum_{L_i \geqslant N_1^{(5-2\alpha)+\varepsilon}} L_i^{\frac{1}{2}} \| f_{i,L_i} \|_{L^2},
	\end{equation*}				
	which yields \eqref{eq:DifferenceSolutionEstimateI} for $s>6-\frac{11 \alpha}{4} + \varepsilon$. This suffices.
	If $N \lesssim 1$, we have straight-forward summation if $\varepsilon < \frac{11 \alpha}{4} - \frac{11}{2}$.
	\item $\underline{L_{\max}\geqslant NN_1^{\alpha}}$: For $L_{\max}=L_3$, using the $L^4$ Strichartz estimate, we obtain
	\begin{equation*}
		\begin{split}
			\eqref{eq:TermToControl} & \lesssim NN_1^{(5-2\alpha)+\varepsilon}\sum_{\substack{{L_1\geqslant NN_1^{\alpha}}\\{L_2,L_3\geqslant N_1^{(5-2\alpha)+\varepsilon} }}} \Big| \int_{\R\times \R^2} (f_{1,L_1}\ast f_{2,L_2})\cdot f_{3,L_3} d\tau d\xi d\eta \Big|\\
			&\lesssim NN_1^{(5-2\alpha)+\varepsilon} (NN_1^{\alpha})^{-\frac{1}{2}}N_1^{\frac{1}{4}-\frac{\alpha}{8}} N_2^{\frac{1}{4}-\frac{\alpha}{8}} \prod_{i=1}^3 \sum_{L_i \geqslant N_1^{(5-2\alpha)+\varepsilon}} L_i^{\frac{1}{2}}\|f_{i,L_i}\|_{L^2}\\
			&\lesssim N^{\frac{1}{2}} N_1^{\frac{11}{2}-\frac{11\alpha}{4}+\varepsilon}\|P_Nv\|_{F_N(T)} \|P_{N_2}u_1\|_{F_{N_2}(T)}\|P_{N_1}v\|_{F_{N_1}(T)}.
		\end{split}
	\end{equation*}
	This suffices for $N \lesssim 1$ because the exponent of $N_1$ is negative. If $N \gtrsim 1$, we can estimate	
	\begin{equation*}
		\lesssim N_1^{6-\frac{11\alpha}{4}+\varepsilon} \|P_Nv\|_{F_N(T)} \|P_{N_2}u_1\|_{F_{N_2}(T)}\|P_{N_1}v\|_{F_{N_1}(T)}.
	\end{equation*}
	If $L_{\max} = L_2$, we apply the $L^4$ Strichartz estimate to $f_{1,L_1}$ and $f_{3,L_3}$ and utilize the modulation gain from $f_{2,L_2}$ as follows:
	\begin{equation*}
		\begin{split}
			\eqref{eq:TermToControl} &\lesssim N N_1^{(5-2\alpha)+\varepsilon}  N^{\frac{1}{4}-\frac{\alpha}{8}} N_1^{\frac{1}{4}-\frac{\alpha}{8}} N_1^{-\frac{\alpha}{2}} N^{-\frac{1}{2}} \prod_{i=1}^3 \sum_{L_i \geqslant N_1^{(5-2\alpha)+\varepsilon}} L_i^{\frac{1}{2}} \| f_{i,L_i} \|_{L^2}\\
			&\lesssim N^{\frac{3}{4}-\frac{\alpha}{8}} N_1^{\frac{21}{4}-\frac{21\alpha}{8}+\varepsilon} \|P_Nv\|_{F_N(T)} \|P_{N_2}u_1\|_{F_{N_2}(T)}\|P_{N_1}v\|_{F_{N_1}(T)}.
		\end{split}
	\end{equation*}
	If $N \lesssim 1$, we have straight-forward summation for $2<\alpha<4$ because $\frac{3}{4} - \frac{\alpha}{8}>0$, and the exponent of $N_1$ is negative. If $N \gtrsim 1$, we obtain
	\begin{equation*}
	\lesssim N_1^{6-\frac{11 \alpha}{4} + \varepsilon} \|P_Nv\|_{F_N(T)} \|P_{N_2}u_1\|_{F_{N_2}(T)}\|P_{N_1}v\|_{F_{N_1}(T)},
	\end{equation*}
	which suffices. 	The case $L_{\max} = L_1$ can be treated similarly. 
\end{itemize}
The proof of \eqref{eq:DifferenceSolutionEstimateI} is concluded by summing up in the $x$ frequencies. Note that for these terms we cover the same regularity as in Proposition \ref{prop:EnergyEstSol}. The term which leads to worse estimates is the following:
\begin{equation*}
	P_N v \cdot P_{N_1}v \cdot \partial_x P_{N_2} u_1, \quad N_1 \ll N_2 \sim N,
\end{equation*}
which corresponds to the second term in \eqref{eq:Terms1}. 
Using the notation and reductions explained in the beginning of this section, we define
\begin{equation*}
	\begin{split}
		f_1 &= \mathcal{F}(\gamma(N^{(5-2\alpha)+\varepsilon}t-n)\mathbf{1}_{[0,T]}(t) P_{N_1} v), \\
		\tilde{f}_2 &= \mathcal{F}(\gamma(N^{(5-2\alpha)+\varepsilon}t-n)\mathbf{1}_{[0,T]}(t) \partial_x P_{N_2} u_1),\\
		f_3 &= \mathcal{F}(\gamma(N^{(5-2\alpha)+\varepsilon}t-n)\mathbf{1}_{[0,T]}(t)  P_N v),
	\end{split}
\end{equation*}
After considering the derivative in $\tilde{f}_2$, carrying out summation in $n$, we require to handle the following term:
\begin{equation*}
	NN^{(5-2\alpha)+\varepsilon}\Big | \int_{\R \times \R^2} (f_1 \ast f_2)\cdot f_3 d\tau d\xi d\eta~ \Big|,
\end{equation*} 
where 
\begin{equation*}
	f_2 = \mathcal{F}(\gamma(N^{(5-2\alpha)+\varepsilon}t-n)\mathbf{1}_{[0,T]}(t) P_{N_2} u_1).
\end{equation*}
Furthermore, decomposing $f_i$ in modulation $L_i$, $i=1,2,3$, we reduce to estimating a term of the form
\begin{equation}\label{eq:TermToEstimate}
	N N^{(5-2\alpha)+ \varepsilon} \sum_{N_1^{(5-2\alpha)+\varepsilon} \leqslant L_i \leqslant N_1^\alpha N} \Big|\int_{\R \times \R^2} (f_{1,L_1} * f_{2,L_2}) \cdot f_{3,L_3}  ~d\tau d\xi d\eta \Big|.
\end{equation}
Depending on the size of $L_{\max}$, we consider the following cases:
\begin{itemize}
	\item \underline{$L_{\max} \leqslant N^{\alpha}N_1$}: We further consider two subcases:\\
	$\star~N_1^{\frac{1}{2}}\leqslant N_2^{\frac{3}{2}-\frac{3\alpha}{4}}$: An application of the bilinear Strichartz estimate \eqref{eq:BilinearToProve} gives
	\begin{equation*}
		\eqref{eq:TermToEstimate} \lesssim N^{(5-2\alpha)+\varepsilon} N \frac{N_1^{\frac{1}{2}}}{N_2^{\frac{\alpha}{4}}} N_2^{\big(\frac{2\alpha -5}{2} \big)-\frac{\varepsilon}{2}} \prod_{i=1}^3 \sum_{L_i\geqslant N_2^{ (5-2\alpha)+\varepsilon}} L_i^{\frac{1}{2}} \| f_{i,L_i} \|_{L^2}.
	\end{equation*}
	Summation in $N_1$ gives
	\begin{equation*}
		\lesssim N^{(5-2\alpha)+\frac{\varepsilon}{2}}  \|P_{N_1}v\|_{F_{N_1}(T)}\|P_{N_2}u_1\|_{F_{N_2}(T)} \|P_{N}v\|_{F_{N_1}(T)}.
	\end{equation*}
	$\star~ N_2^{\frac{3}{2}-\frac{3\alpha}{4}} \leqslant N_1^{\frac{1}{2}}$: We use \eqref{eq:LW} to obtain
	\begin{equation*}
		\begin{split}
			\eqref{eq:TermToEstimate}
			&\lesssim N^{(5-2\alpha)+\varepsilon} N N^{-\frac{3\alpha}{4}+\frac{1}{2}} N_1^{-\frac{1}{2}} \prod_{i=1}^3 \sum_{L_i \geqslant N_2^{(5-2\alpha)+\varepsilon}} L_i^{\frac{1}{2}} \| f_{i,L_i} \|_{L^2} \\
			&\lesssim N^{(5-2\alpha)+\varepsilon} \|P_Nv\|_{F_N(T)} \|P_{N_2}u_1\|_{F_{N_2}(T)}\|P_{N_1}v\|_{F_{N_1}(T)}.
		\end{split}
	\end{equation*}
Now we handle the non-resonant case.	
\item \underline{$L_{\max} \geqslant N^{\alpha}N_1$}: We apply the estimate \eqref{eq:GuoNonReso} by assuming that $L_{\max} = L_3$. Note that in this case the small frequency $N_1$ can have size $\lesssim 1$.
\begin{equation*}
	\begin{split}
		\eqref{eq:TermToEstimate} &\lesssim N N^{(5-2\alpha)+\varepsilon} N_1^{\frac{1}{4}} N^{-\frac{\alpha}{2}} L_{\max}^{-\frac{1}{4}} \prod_{i=1}^3 \sum_{\substack{{L_i \geqslant N^{(5-2\alpha)+\varepsilon}},\\{L_{\max}\geqslant N_1N^{\alpha}}}} L_i^{\frac{1}{2}} \| f_{i,L_i} \|_{L^2} \\
		&\lesssim N^{(6-\frac{11\alpha}{4})+\varepsilon} \|P_Nv\|_{F_N(T)} \|P_{N_2}u_1\|_{F_{N_2}(T)}\|P_{N_1}v\|_{F_{N_1}(T)}. 
	\end{split}
\end{equation*}
This suffices if $N_1 \gtrsim 1$. If $N_1 \lesssim 1$, we can interpolate with the prior estimate to find 
\begin{equation*}
\lesssim N_1^\delta N^{(6-\frac{11\alpha}{4})+2\varepsilon} \|P_Nv\|_{F_N(T)} \|P_{N_2}u_1\|_{F_{N_2}(T)}\|P_{N_1}v\|_{F_{N_1}(T)}.
\end{equation*}
with straight-forward summation
The other assumptions, namely $L_{\max}=L_1$ or $L_{\max}=L_2$ lead to the same conclusion.
\end{itemize}
The proof of \eqref{eq:DifferenceSolutionEstimateI} follows by substituting the obtained estimates in \eqref{eq:LPDifferenceEnergyEstimate} and carrying out a summation in the $x$ frequencies. For \eqref{eq:DifferenceSolutionEstimateII}, we multiply the same by $N^{2s}$ and sum up. Noting that $u_1 = v+u_2$ leads to \eqref{eq:DifferenceSolutionEstimateII}.
\end{proof}
\end{proposition}

\subsection{Proof of Theorem \ref{thm:QuasilinearPrecise}} We conclude the proof of Theorem \ref{thm:QuasilinearLWP} in this section. In the first step, we show a priori estimates.

\emph{A priori estimates:} Let $\alpha \in (2,\frac{5}{2}]$, $\varepsilon' > 0$, $\varepsilon = \varepsilon(\alpha,\varepsilon')$, and
\begin{equation*}
s \geqslant 
\begin{cases}
 6-\frac{11\alpha}{4} + \varepsilon', &\quad \alpha \in (2,\frac{24}{11}], \\
 0, &\quad \alpha \in (\frac{24}{11},\frac{5}{2}],
 \end{cases}
\end{equation*}
and $u_0 \in H^{\infty,0}$. We can suppose by rescaling and subcriticality that $\| u_0 \|_{H^{s,0}} \leqslant \varepsilon_0 \ll 1$ with $\varepsilon_0$ determined later. By the local well-posedness in $H^{2,0}$, we have existence of solutions in $H^{2,0}$ for $T_{\max} = T_{\max}(\| u_0 \|_{H^{2,0}})$. 

By Lemma \ref{lemma:LinShortTime}, Proposition \ref{prop:ShortTimeBilinear}, and Proposition \ref{prop:EnergyEstSol}, we have the following set of estimates for $T \leqslant \min(T_{\max},1)$ and time localization $T(N) = N^{-(5-2\alpha)-\varepsilon}$ provided that $\varepsilon$ is chosen small enough:
\begin{equation*}
	\left\{ \begin{array}{cl}
		\| u \|_{F^{s,0}(T)} &\lesssim \| u \|_{E^{s,0}(T)} + \| \partial_x (u^2) \|_{\mathcal{N}^{s,0}(T)}, \\
		\| \partial_x(u^2) \|_{\mathcal{N}^{s,0}(T)} &\lesssim \| u \|^2_{F^{s,0}(T)}, \\
		\| u \|^2_{E^{s,0}(T)} &\lesssim \| u_0 \|^2_{H^{s,0}} + \| u \|^3_{F^{s,0}(T)}.
	\end{array} \right.
\end{equation*}
This yields
\begin{equation}
    \label{eq:SolutionsFs}
    \| u \|^2_{F^{s,0}(T)} \lesssim \| u_0 \|^2_{H^{s,0}} + \| u \|^4_{F^{s,0}(T)} + \| u \|^3_{F^{s,0}(T)}.
\end{equation}

Secondly, we have (cf. \cite[Lemma~4.2,~p.~279]{IKT} )
\begin{equation*}
	\lim_{T\downarrow 0} \|u\|_{E^{s,0}(T)} \lesssim \|u_0\|_{H^{s,0}}, \quad \lim_{T\downarrow 0} \|\partial_x(u^2)\|_{\mathcal{N}^{s,0}(T)} =0.
\end{equation*}
Hence, by choosing $\varepsilon_0$ small enough, we find by \eqref{eq:SolutionsFs} and a continuity argument that
\begin{equation}\label{eq:fKPAPriori}
	\|u\|_{F^{s,0}(T)} \lesssim \|u_0\|_{H^{s,0}}
\end{equation}
for $T= \min(1,T_{\max})$. Another application of Lemma \ref{lemma:LinShortTime} and Propositions \ref{prop:ShortTimeBilinear} and \ref{prop:EnergyEstSol} yields
\begin{equation*}
	\left\{ \begin{array}{cl}
		\| u \|_{F^{2,0}(T)} &\lesssim \| u \|_{E^{2,0}(T)} + \| \partial_x (u^2) \|_{\mathcal{N}^{2,0}(T)}, \\
		\| \partial_x(u^2) \|_{\mathcal{N}^{2,0}(T)} &\lesssim \| u \|_{F^{2,0}(T)} \| u \|_{F^{s,0}(T)}, \\
		\| u \|^2_{E^{2,0}(T)} &\lesssim \| u_0 \|^2_{H^{2,0}} + \| u \|^2_{F^{2,0}(T)} \| u \|_{F^{s,0}(T)}.
	\end{array} \right.
\end{equation*}
This set of estimates yields
\begin{equation*}
    \| u \|_{F^{2,0}(T)}^2 \lesssim \| u_0 \|_{H^{2,0}}^2 + \| u \|_{F^{s,0}(T)} \| u \|_{F^{2,0}(T)}^2 +  \| u \|^2_{F^{2,0}(T)} \| u \|_{F^{s,0}(T)}^2,
\end{equation*}
and therefore, for $\| u \|_{F^{s,0}(T)} \lesssim \varepsilon_0$ we have $\| u \|_{F^{2,0}(T)} \lesssim \| u_0 \|_{H^{2,0}}$. Consequently, we have existence up to $T=1$ choosing $\varepsilon_0$ sufficiently small only depending on $\| u_0 \|_{H^{s,0}} \leq \varepsilon_0$.

Since $s=0$ for $\alpha>\frac{24}{11}$, by the above a priori estimates and the conservation of mass  \eqref{eq:MassConservation}, we can show \emph{global existence} of solutions.

\begin{theorem}[Global~existence~for~smooth~solutions]
Let $\alpha \in ( \frac{24}{11}, \frac{5}{2}]$, and $u_0 \in H^{\infty,0}(\R^2)$. For any $s \geqslant 0$ we have a solution $u \in C(\R;H^{s,0})$ to \eqref{eq:fKPI}, and there exist $C_1(\| u_0 \|_{L^2}),C_2(\| u_0 \|_{L^2}) >0$ such that
\begin{equation}
\label{eq:APriori}
\| u(t) \|_{H^{s,0}} \leq C_1 e^{C_2 |t|} \| u_0 \|_{H^{s,0}}.
\end{equation}

\end{theorem}
\begin{proof}
It is enough to consider $t>0$ by time-reversal. Firstly, we can rescale the initial data $u_0$ to $u_{0 \lambda}$, which satisfies $\| u_{0 \lambda} \|_{L^2} = \varepsilon \ll 1$. By the local well-posedness result in $H^{2,0}$ due to Molinet--Saut--Tzvetkov \cite{MolinetSautTzvetkov2007} we have for the corresponding solution $u_\lambda \in C([0,T],H^{s,0})$ with $T = T(\| u_{0 \lambda} \|_{H^{2,0}})$. Let $T' \leq 1 \wedge T$. We have the following set of estimates:
\begin{equation*}
\left\{ \begin{array}{cl}
\| u_\lambda \|_{F^{0,0}(T')} &\lesssim \| u_\lambda \|_{E^{0,0}(T')} + \| \partial_x (u_\lambda)^2 \|_{\mathcal{N}^{0,0}(T')}, \\
\| \partial_x (u_\lambda)^2 \|_{\mathcal{N}^{0,0}(T')} &\lesssim \| u_\lambda \|^2_{F^{0,0}(T')}, \\
\| u_\lambda \|_{E^{0,0}(T')}^2 &\lesssim \| u_{0 \lambda} \|^2_{L^2} + \| u_\lambda \|^3_{F^{0,0}(T')}.
\end{array} \right.
\end{equation*}
By $\| u_{0 \lambda} \|_{L^2} \ll 1$, a continuity argument like above gives for $T' \leq 1 \wedge T$:
\begin{equation}
\label{eq:APrioriF0}
\| u_\lambda \|_{F^{0,0}(T')} \lesssim \| u_{0 \lambda} \|_{L^2}.
\end{equation}
Secondly, we have
\begin{equation*}
\left\{ \begin{array}{cl}
\| u_\lambda \|_{F^{s,0}(T')} &\lesssim \| u_\lambda \|_{E^{s,0}(T')} + \| \partial_x (u_\lambda^2) \|_{\mathcal{N}^{s,0}(T')}, \\
\| \partial_x (u_\lambda^2) \|_{\mathcal{N}^{s,0}(T')} &\lesssim \| u_\lambda \|_{F^{s,0}(T')} \| u_\lambda \|_{F^{0,0}(T')}, \\
\| u_\lambda \|^2_{E^{s,0}(T')} &\lesssim \| u_{0 \lambda} \|^2_{H^{s,0}} + \| u_\lambda \|^2_{F^{s,0}(T')} \| u \|_{F^{0,0}(T')}.
\end{array} \right.
\end{equation*}
By \eqref{eq:APrioriF0} we obtain
\begin{equation}
\label{eq:APrioriFs}
\| u_\lambda \|_{F^{s,0}(T')} \lesssim \| u_{0 \lambda} \|_{H^{s,0}}.
\end{equation}
Consequently, $\| u_\lambda(t) \|_{H^{2,0}}$ remains bounded for $t \leq 1 \wedge T$ and by the local well-posedness result $u_\lambda$ exists until $t=1$: We have $u_\lambda \in C([0,1],H^{s,0})$ with $\| u_\lambda(1) \|_{H^{s,0}} \lesssim \| u_{0 \lambda} \|_{H^{s,0}}$. However, $\| u_{0 \lambda} \|_{L^2} = \| u_{0 \lambda} \|_{L^2} = \varepsilon \ll 1$. For this reason, the argument can be iterated and we find $u_\lambda \in C(\R; H^{s,0})$ with
\begin{equation*}
\| u_\lambda(t) \|_{H^{s,0}} \leq C_1 e^{C_2 t} \| u_{0 \lambda} \|_{H^{s,0}},
\end{equation*}
which follows from iterating \eqref{eq:APrioriFs}. Hence,
\begin{equation*}
\| u(t) \|_{H^{s,0}} \leq C_1(\lambda) e^{C_2(\lambda) t} \| u_0 \|_{H^{s,0}}
\end{equation*}
with $\lambda = \lambda(\| u_0 \|_{L^2})$. The proof is complete.
\end{proof}
\medskip
Now we prove the continuity of the data-to-solution map. In the first step, we show Lipschitz continuous dependence of the solutions in $L^2$ for small initial data of higher regularity.\\
\emph{Lipschitz continuous dependence in $L^2$:}
Let $s>5-2\alpha$ and  $u_1, u_2$ denote two local-in-time solutions with initial data $\| u_i(0) \|_{H^{s,0}} \leqslant \varepsilon_0$.
By the above argument, we have for $s' \geqslant s$ 
\begin{equation}\label{eq:fKPHighReg}
	\| u \|_{F^{s',0}(1)} \lesssim \|u_0\|_{H^{s',0}}.
\end{equation}
Let $v= u_1 - u_2$ denote the solution to the difference equation
\begin{equation*}
\partial_t v - D_x^\alpha \partial_x v - \partial_x^{-1} \partial_y^2 v = \partial_x (v (u_1+u_2))/2.
\end{equation*}
From Lemma \ref{lemma:LinShortTime}, Proposition \ref{prop:ShortTimeBilinear}, Proposition \ref{prop:fKPEnergyEstDiff}, we have
\begin{equation}
	\left\{ 	\begin{array}{cl}
		\|v\|_{F^{0,0}(1)} &\lesssim \|v\|_{E^{0,0}(1)} + \| \partial_x(v(u_1+u_2))\|_{\mathcal{N}^{0,0}(1)},	\\
		\|\partial_x(v(u_1+u_2))\|_{\mathcal{N}^{0,0}(1)} &\lesssim \| v\|_{F^{0,0}(1)} (\|u_1\|_{F^{s,0}(1)} + \|u_2\|_{F^{s,0}(1)}), \\
		\|v\|_{E^{0,0}(1)}^2 &\lesssim \|v(0)\|_{L^2}^2 + \|v\|_{F^{0,0}(1)}^2 (\|u_1\|_{F^{s,0}(1)} + \|u_2\|_{F^{s,0}(1)} ).
	\end{array}\right.
\end{equation}
This enables us to conclude
\begin{equation}\label{eq:fKPZeroReg}
	\|v\|_{F^{0,0}(1)} \lesssim \|v(0)\|_{L^2},
\end{equation}
since $\|u_i \|_{F^{s,0}(1)} \lesssim \varepsilon_0$ are chosen sufficiently small.

\medskip

\emph{Continuity of the data-to-solution mapping:}
Also, from Lemma \ref{lemma:LinShortTime}, Proposition \ref{prop:ShortTimeBilinear}, and Proposition \ref{prop:fKPEnergyEstDiff}, we have
\begin{equation}
	\left\{ 	\begin{array}{cl}
		\|v\|_{F^{s,0}(T)} &\lesssim \|v\|_{E^{s,0}(T)} + \| \partial_x(v(u_1+u_2))\|_{\mathcal{N}^{s,0}(T)}	\\
		\|\partial_x(v(u_1+u_2))\|_{\mathcal{N}^{s,0}(T)} &\lesssim \| v\|_{F^{s,0}(T)} (\|u_1\|_{F^{s,0}(T)} + \|u_2\|_{F^{s,0}(T)})\\
		\|v\|_{E^{s,0}(T)}^2 &\lesssim \|v(0)\|_{H^{s,0}}^2 + \|v\|_{F^{s,0}(T)}^3 \\ &\quad + \|v\|_{F^{0,0}(T)} \|v\|_{F^{s,0}(T)}\|u_2\|_{F^{2s,0}(T)}.
	\end{array}\right.
\end{equation}
From the above set of estimates, we can conclude a priori estimates for $\|v\|_{F^{s,0}(T)}$:
\begin{equation}\label{eq:APrioriV}
	\begin{split}
		\|v\|_{F^{s,0}(T)}^2 &\lesssim \|v(0)\|_{H^{s,0}}^2 + \|v\|_{F^{s,0}(T)}^3 + \|v\|_{F^{0,0}(T)} \|v\|_{F^{s,0}(T)}\|u_2\|_{F^{2s,0}(T)}.
	\end{split}
\end{equation}
We use the smallness of $\| u_i \|_{H^{s,0}}$ to absorb the term from the nonlinear estimate into the left-hand side.

For $s> 5-2\alpha$, let $\phi \in H^{s,0}$ be fixed and $\{\phi_n\}_{n=1}^{\infty} \in H^{\infty,0}$ be such that
\begin{equation}
	\lim_{n\rightarrow \infty}\phi_n = \phi.
\end{equation}
By rescaling and subcriticality, we can again assume that $\|\phi\|_{H^{s,0}} \leqslant \varepsilon_0 \ll 1$ and $\|\phi_n\|_{H^{s,0}} \leqslant 2\varepsilon_0 \ll 1 \text{ for all } n \in \N$.
Let $u_1$ be the solution corresponding to initial data $\phi_n$, and $u_2$ be the solution corresponding to initial data $P_{\leqslant N}\phi_n$. We construct the data-to-solution mapping as an extension of the data-to-solution mapping for smooth initial data. Let
\begin{equation*}
S_T^{\infty}(\phi_n)\in C([-1,1];H^{\infty,0})
\end{equation*}
denote the solution corresponding to smooth initial data. We can take the existence time as $1$ by the a priori estimates and persistence property argued above.

To prove the continuity of the data-to-solution map, we need to show that the sequence $S_T^{\infty}(\phi_n) \in C([-1,1];H^{\infty,0})$ is a Cauchy sequence in the space $C([-1,1]; H^{s,0})$, $s > 5-2\alpha$.\\
Hence, it suffices to show that for any $\delta>0$, there exists $M_{\delta} \in \N$ such that
\begin{equation*}
	\| S_T^{\infty}(\phi_n) - S_T^{\infty}(\phi_m)\|_{C([-1,1];H^{s,0})} \leqslant \delta ~~\text{for all } m,n \geqslant M_{\delta}.
\end{equation*}
For $K \in 2^{\N_0}$, let $\phi_n^K:=P_{\leqslant K}\phi_n$.
We have
\begin{equation}
	\begin{split}
		\| S_T^{\infty}(\phi_n) - S_T^{\infty}(\phi_m)\|_{C([-1,1];H^{s,0})}
		&\leqslant \| S_T^{\infty}(\phi_n) - S_T^{\infty}(\phi_n^K)\|_{C([-1,1]; H^{s,0})} \\
		&\quad + \| S_T^{\infty}(\phi_m) - S_T^{\infty}(\phi_m^K)\|_{C([-1,1];H^{s,0})} \\
		&\quad + \| S_T^{\infty}(\phi_n^K) - S_T^{\infty}(\phi_m^K)\|_{C([-1,1]; H^{s,0})}.
	\end{split}
\end{equation}
The third term can be handled by using the continuity of the data-to-solution map for smooth data in $H^{2,0}$:
	\begin{equation}
		\| S_T^{\infty}(\phi_n^K) - S_T^{\infty}(\phi_m^K)\|_{C([-1,1];H^{s,0})} \leqslant \| S_T^{\infty}(\phi_n^K) - S_T^{\infty}(\phi_m^K) \|_{H^{2,0}} \to 0
	\end{equation}
	for $m,n \to \infty$ because $\| \phi_m^K - \phi_n^K \|_{H^{2,0}} \to 0$.
Let $v=S_T^{\infty}(\phi_n) - S_T^{\infty}(\phi_n^K)$. We observe that $v$ is the solution corresponding to initial data $P_{>K}\phi_n$. From \eqref{eq:fKPZeroReg}, we have
\begin{equation*}
	\|v\|_{F^{0,0}(T)} \lesssim \|\phi_n -\phi_n^K\|_{L^2} \lesssim K^{-s}\|P_{>K}\phi_n\|_{H^{s,0}}.
\end{equation*}
From \eqref{eq:fKPHighReg}, we have for $u_2$:
\begin{equation}
	\| u_2\|_{F^{2s,0}(T)} \lesssim \|\phi_n^K\|_{H^{2s,0}} \lesssim K^s\|\phi_n\|_{H^{s,0}}.
\end{equation}
Combining the above with \eqref{eq:APrioriV}, we conclude an a priori estimate for $v$ which now depends on the profile of the initial data, namely on $P_{>K}\phi_n$. We have
\begin{equation*}
	\begin{split}
		\| S_T^{\infty}(\phi_n) - S_T^{\infty}(\phi_m)\|_{C([-1,1];H^{s,0})} \lesssim \|P_{>K}\phi_n\|_{H^{s,0}} + \|P_{>K}\phi_m\|_{H^{s,0}} + C(m,n,K).
	\end{split}
\end{equation*}
By the convergence of $\phi_n$ and choosing $K$ large enough so that
\begin{equation*}
	\|P_{>K}\phi_n\|_{H^{s,0}} + \|P_{>K}\phi_m\|_{H^{s,0}} <\varepsilon,
\end{equation*}
we conclude that $\{S_T^{\infty}(\phi_n)\}_{n\in \N}$ is a Cauchy sequence in $C([-1,1];H^{s,0})$. This shows that $S_T^\infty$ extends to a continuous map $S_T: H^{s,0} \to C([-1,1];H^{s,0})$. $\hfill \Box$


\section{Semilinear well-posedness}
\label{section:Semilinear}
For $\alpha>\frac{5}{2}$, we observe via estimates \eqref{eq:LW} and \eqref{eq:BilinearToProve} that we can remedy the derivative loss completely without having to use frequency-dependent time localization. We show local well-posedness through a fixed point argument. This we carry out within the standard Fourier restriction spaces as our auxiliary spaces.
Let  $s,b \in \R$ and $\omega_{\alpha}(\xi,\eta)=|\xi|^{\alpha}\xi+\frac{\eta^2}{\xi}$. The space $X^{s,b}$ corresponding to the fractional KP-I equation \eqref{eq:fKPI} is defined as the closure of Schwartz functions with respect to the norm
\begin{equation*}
	\| u\|_{X^{s,b}(\R \times \R^2)} := \| \langle \xi \rangle^s \langle \tau-\omega_{\alpha}(\xi,\eta)\rangle^b \hat{u} (\tau,\xi)\|_{L_{\tau,\xi,\eta}^2(\R \times \R^2)} = \|U_{\alpha}(-t)u\|_{H^{b}_tH^s_{x,y}(\R \times \R^2)},
\end{equation*}
where $U_{\alpha}(t)$ denotes the solution operator corresponding to the linear equation. We localize in time as usual by setting
\begin{equation*}
X_T^{s,b} = \{ f:[0,T] \times \R^2 \to \C \, | \; \exists \, \tilde{f}  \in X^{s,b}: \tilde{f} \big|_{[0,T]} = f \}
\end{equation*}
endowed with norm
\begin{equation*}
\| f \|_{X_T^{s,b}} = \inf_{\tilde{f} \big|_{[0,T]} = f} \| \tilde{f} \|_{X^{s,b}}.
\end{equation*}

 With the function spaces introduced, we give a precise version of Theorem \ref{thm:SemilinearLWP}.
\begin{theorem}
\label{thm:SemilinearLWPPrecise}
	Let $\alpha>\frac{5}{2}$ and $s> \frac{5}{4}-\frac{\alpha}{2}$. Then, there is $b>1/2$ such that for $T=T(\| u_0 \|_{H^{s,0}})$, \eqref{eq:fKPI} is analytically locally well-posed in $H^{s,0}$ with the solution lying in $X_T^{s,b} \hookrightarrow C([0,T];H^{s,0})$.
	\end{theorem}
The section is devoted to the proof of Theorem \ref{thm:SemilinearLWPPrecise}. We begin with a reminder on the basic properties of $X^{s,b}$ spaces, which show that for the proof of the theorem, it suffices to show the bilinear estimate
\begin{equation*}
\| \partial_x (uv) \|_{X^{s,b-1}} \lesssim \| u \|_{X^{s,b}} \| v \|_{X^{s,b}} 
\end{equation*}
for some $b>1/2$. The bilinear estimate is proved in Subsection \ref{subsection:BilinearEstimate}.

\subsection{Properties of $X^{s,b}$ spaces} 
Proofs of the following basic properties can be found in \cite[Section~2.5]{Tao2006}.
First, recall that free solutions are in $X^{s,b}$ locally in time. Recall the linear propagator of \eqref{eq:fKPI} from \eqref{eq:LinearProp}.
\begin{lemma}[{cf.~\cite[Lemma~2.8]{Tao2006}}]
\label{lemma:XsbTimeLocalisedFreeSol}
Let $s \in \R$, $u_0 \in H^{s,0}(\R^2)$ and $\eta \in \mathcal{S}(\R)$. Then, the following estimate holds:
\begin{equation*}
\| \eta(t) U_\alpha(t) u_0 \|_{X^{s,b}} \lesssim_{b,\eta} \| u_0 \|_{H^{s,0}(\R^2)}.
\end{equation*}
\end{lemma}
This yields the following transfer principle for $b>1/2$, stating that properties of free solutions are inherited by $X^{s,b}$-functions:
\begin{lemma}[{cf.~\cite[Lemma~2.9]{Tao2006}}]
Let $b>1/2$ and $s \in \R$. Let $Y$ be a Banach space comprised of functions in $\R \times \R^2$ with the property that
\begin{equation*}
\| e^{it \tau_0} U_\alpha(t) f \|_{Y} \lesssim \| f \|_{H^{s,0}}
\end{equation*}
for all $\tau_0 \in \R$ and $f \in H^{s,0}$. Then, we have the embedding
\begin{equation*}
\| u \|_Y \lesssim_b \| u \|_{X^{s,b}}.
\end{equation*}
\end{lemma}
By Duhamel's formula for solutions to 
\begin{equation*}
\left\{ \begin{array}{cl}
    \partial_t u - D_x^\alpha \partial_x u - \partial_x^{-1} \partial_y^2 u &= F(u), \\
    u(0) &= u_0, 
\end{array} \right.
\end{equation*}
we can write
\begin{equation}
\label{eq:DuhamelEnergyFormula}
u(t) = U_\alpha(t) u_0 + \int_0^t U_\alpha(t-s) F(u(s)) ds.
\end{equation}
The following energy inequality for $X^{s,b}$ spaces becomes evident:
\begin{lemma}\label{lemma:XsbEnergyEst}
Let $u$ be like in \eqref{eq:DuhamelEnergyFormula} and $\eta \in C^\infty_c(\R)$, $s \in \R$, and $b>1/2$. Then, the following estimate holds:
\begin{equation*}
\| \eta(t) u \|_{X^{s,b}} \lesssim \| u_0 \|_{H^{s,0}(\R^2)} + \| F(u) \|_{X^{s,b-1}}.
\end{equation*}
\end{lemma}	
For the frequency and modulation localization operators we use same notations like in Section \ref{section:Notation}.

\subsection{Bilinear estimate} To prove Theorem \ref{thm:SemilinearLWP} via the fixed point theorem, we require to control the nonlinearity in the $X^{s,b-1}$ norm which we do in the following. We prove the estimate in a fixed time interval $[0,1]$ so that we do not have to keep track of additional decomposition in modulation or gain of small powers in $T$.  For brevity, we also omit the subscript $1$ for the length of the time interval in the $X^{s,b}$ norms.
\label{subsection:BilinearEstimate}
\begin{proposition}\label{prop:XsbBilinear}
Let $\frac{5}{2}<\alpha<4$. Then, for $s> \frac{5}{4}-\frac{\alpha}{2}$, there is some $b>\frac{1}{2}$ such that the following estimate holds:
\begin{equation*}
	\|\partial_x(uv)\|_{X^{s,b-1}} \lesssim \|u\|_{X^{s,b}}\|v\|_{X^{s,b}}.
\end{equation*}

For $s \geq 0$, there is some $b>\frac{1}{2}$ such that the following holds:
\begin{equation*}
	\|\partial_x(uv)\|_{X^{s,b-1}} \lesssim \|u\|_{X^{0,b}}\|v\|_{X^{s,b}}.
\end{equation*}

\begin{proof}
	By duality and Plancherel's theorem, we can reduce the above to proving
	\begin{equation}\label{eq:DualXsb}
		\Big | \int_{\R^3} \xi \widehat{(uv)} \cdot \overline{\widehat{w}}~ d\tau d\xi d\eta \Big| \lesssim \|u\|_{X^{s,b}} \|v\|_{X^{s,b}} \|w\|_{X^{-s,1-b}}.
	\end{equation}

Let $N_i \in 2^{\Z}$, $L_j \in 2^{\N_0}$. 	
For functions $f_{N_1,L_1}, g_{N_2,L_2}$ and $h_{N,L}$ supported in $\tilde{D}_{N_1,L_1}, \tilde{D}_{N_2,L_2}$ and $\tilde{D}_{N,L}$, respectively, we focus on dyadic estimates
\begin{equation}\label{eq:DualXsbD}
	\Big  |\int(f_{N_1,L_1}\ast g_{N_2,L_2})\cdot h_{N,L} \Big | \lesssim L_1^{\frac{1}{2}} L_2^{\frac{1}{2}} L^{\frac{1}{2}-} C(N_1,N_2,N) \|f_{N_1,L_1}\|_{L^2} \| g_{N_2,L_2}\|_{L^2} \|h_{N,L}\|_{L^2}.
\end{equation}
We turn to a case-by-case analysis depending on the size of the $x$ frequencies.

\medskip

\noindent \textbf{High-High-Low} ($N_2\ll N_1\sim N$): 
We first treat the $High \times Low \to High$ interaction and prove \eqref{eq:DualXsbD} with $C(N_1,N_2) =  N_1^{(\frac{1}{4}-\frac{\alpha}{2})+} N_2^{0+}$. Summing up the above  in $L,L_1,L_2,N,N_1$, and $N_2$ proves \eqref{eq:DualXsb} for $s> \frac{5}{4}-\frac{\alpha}{2}$ after taking into account the additional derivative loss. For $L_{\max}=\max(L_1,L_2,L)$, two cases arise:\\
	$\bullet$	\underline{$L_{\max}\leqslant N_1^{\alpha}N_2$}:
	For $N_2 \lesssim N_1^{\frac{1-\alpha}{2}}$, we use Cauchy-Schwarz inequality and the bilinear Strichartz estimate \eqref{eq:BilinearToProve}. Note that since $\alpha>\frac{5}{2}$, we have $N_2 \ll 1$ necessarily.
	\begin{equation}
	\label{eq:ResonantEstimateSemilinear}
		\begin{split}
			\Big |\int (f_{N_1,L_1}\ast g_{N_2,L_2})\cdot h_{N,L} \Big| &\lesssim (L_1L_2)^{\frac{1}{2}} \frac{N_2^{\frac{1}{2}}}{N_1^{\frac{\alpha}{4}}} \|f_{N_1,L_1}\|_{L^2} \|g_{N_2,L_2}\|_{L^2} \|h_{N,L}\|_{L^2}\\
			&\lesssim (L_1L_2)^{\frac{1}{2}} 
			N_1^{(\frac{1}{4}-\frac{\alpha}{2})-} N_2^{0+} \|f_{N_1,L_1}\|_{L^2}\|g_{N_2,L_2}\|_{L^2}\|h_{N,L}\|_{L^2}.
		\end{split}
	\end{equation}	
	
	For $N_2\gtrsim N_1^{\frac{1-\alpha}{2}}$, we use the nonlinear Loomis--Whitney inequality \eqref{eq:LW}:
	\begin{equation*}
		\begin{split}
			\text{LHS of } \eqref{eq:DualXsbD} &\lesssim (LL_1L_2)^{\frac{1}{2}} N_1^{-\frac{3\alpha}{4}+\frac{1}{2}} N_2^{-\frac{1}{2}} \|f_{N_1,L_1}\|_{L^2}\|g_{N_2,L_2}\|_{L^2}\|h_{N,L}\|_{L^2}\\
			&\lesssim (L_1L_2)^{\frac{1}{2}} L^{\frac{1}{2}-} N_1^{(\frac{1}{4}-\frac{\alpha}{2})+} N_2^{0+} \|f_{N_1,L_1}\|_{L^2}\|g_{N_2,L_2}\|_{L^2}\|h_{N,L}\|_{L^2}.
		\end{split}
	\end{equation*}
	$\bullet$	\underline{$L_{\max}\geqslant N_1^{\alpha}N_2$}: For $N_2\gtrsim 1$, we use the embedding \eqref{eq:StriEmbed}. Without loss of generality, we assume that $L_{\max} = L$ (other assumptions give same or improved estimates). Using Plancherel's identity and Cauchy-Schwarz inequality, we have
	\begin{equation*}
		\begin{split}
			\text{LHS of } \eqref{eq:DualXsbD} &\lesssim \|\mathcal{F}^{-1}(f_{N_1,L_1})\|_{L^4}\|\mathcal{F}^{-1}(g_{N_2,L_2})\|_{L^4}\|h_{N,L}\|_{L^2}\\
			&\lesssim N_1^{\frac{1}{4}-\frac{\alpha}{8}} N_2^{\frac{1}{4}-\frac{\alpha}{8}} N_1^{-\frac{\alpha}{2}+} N_2^{-\frac{1}{2}+} (L_1L_2)^{\frac{1}{2}} L^{\frac{1}{2}-} \|f_{N_1,L_1}\|_{L^2}\|g_{N_2,L_2}\|_{L^2}\|h_{N,L}\|_{L^2}\\
			&=N_1^{(\frac{1}{4}-\frac{5\alpha}{8})+} N_2^{(-\frac{1}{4}-\frac{\alpha}{8})+} (L_1L_2)^{\frac{1}{2}} L^{\frac{1}{2}-}\|f_{N_1,L_1}\|_{L^2}\|g_{N_2,L_2}\|_{L^2}\|h_{N,L}\|_{L^2}.
		\end{split}	
	\end{equation*}
	For $N_2\lesssim 1$, we use the estimate \eqref{eq:GuoNonReso}:
	\begin{equation*}
		\begin{split}
			\text{LHS of } \eqref{eq:DualXsbD} &\lesssim (L_1L_2)^{\frac{1}{2}} L^{\frac{1}{4}} N_1^{-\frac{\alpha}{2}} N_2^{\frac{1}{4}} \|f_{N_1,L_1}\|_{L^2}\|g_{N_2,L_2}\|_{L^2} \|h_{N,L}\|_{L^2}\\
			&\lesssim (L_1L_2)^{\frac{1}{2}} L^{\frac{1}{2}-} N_1^{-\frac{3\alpha}{4}+} N_2^{0+} \|f_{N_1,L_1}\|_{L^2}\|g_{N_2,L_2}\|_{L^2}\|h_{N,L}\|_{L^2},
		\end{split}
	\end{equation*}
which is sufficient since $1-\frac{3\alpha}{4} + <0$ for $\alpha>\frac{5}{2}$.

\medskip

The $High \times High \to Low$ interaction ($N \ll N_1 \sim N_2$) is treated as follows:  We can argue dually to the previous case, but arguing like in \eqref{eq:ResonantEstimateSemilinear} we find
\begin{equation*}
\big| \int (f_{N_1,L_1}\ast g_{N_2,L_2})\cdot h_{N,L} \Big| \lesssim (L L_1)^{\frac{1}{2}} N^{0+} N_1^{(\frac{1}{4}-\frac{\alpha}{2})-} \| f_{N_1,L_1} \|_{L^2} \| g_{N_2,L_2} \|_{L^2} \| h_{N,L} \|_{L^2}.
\end{equation*}
To lower the modulation of $h$, we interpolate with the following estimate, which we find from two $L^4$ Strichartz estimates:
\begin{equation*}
\big| \int (f_{N_1,L_1}\ast g_{N_2,L_2})\cdot h_{N,L} \Big| \lesssim (L_1 L_2)^{\frac{1}{2}} N_1^{\frac{2-\alpha}{4}} \| f_{N_1,L_1} \|_{L^2} \| g_{N_2,L_2} \|_{L^2} \| h_{N,L} \|_{L^2}.
\end{equation*}

\medskip

\noindent \textbf{Very low frequencies} ($N_1\sim N_2\sim N \lesssim 1$): After using Plancherel's identity and Cauchy-Schwarz inequality, we use the linear $L^4$ Strichartz estimate via \eqref{eq:StriEmbed}:
\begin{equation*}
	\begin{split}
\Big |\int (f_{N_1,L_1} \ast g_{N_2,L_2})\cdot h_{N,L}\Big|  &= \Big|\int \mathcal{F}^{-1}(f_{N_1,L_1}) \mathcal{F}^{-1}(g_{N_2,L_2})\mathcal{F}^{-1}(h_{N,L}) \Big|\\
&\lesssim \|\mathcal{F}^{-1}(f_{N_1,L_1})\|_{L^4} \|\mathcal{F}^{-1}(g_{N_2,L_2})\|_{L^4}\|h_{N,L}\|_{L^2}\\
&\lesssim (L_1L_2)^{\frac{1}{2}} L^{\frac{1}{2}-}\|f_{N_1,L_1}\|_{L^2}\|g_{N_2,L_2}\|_{L^2}\|h_{N,L}\|_{L^2},
\end{split}
\end{equation*}  
which is sufficient for \eqref{eq:DualXsbD}.

\medskip

\noindent \textbf{Three comparable frequencies} ($N_1 \sim N_2 \sim N \gtrsim 1$): We shall prove the estimate \eqref{eq:DualXsbD} with $C(N_1,N_2,N) = N^{-\frac{3\alpha}{4}+}$ by considering two cases:\\
$\bullet$ \underline{$L_{\max} \leqslant N_1^{\alpha+1}$}: We use the estimate \eqref{eq:LW}:
\begin{equation*}
	\begin{split}
		\text{LHS of }\eqref{eq:DualXsbD} \lesssim N_1^{-\frac{3\alpha}{4}+} (L_1L_2)^{\frac{1}{2}}L^{\frac{1}{2}-} \|f_{N_1,L_1}\|_{L^2}\|g_{N_2,L_2}\|_{L^2}\|h_{N,L}\|_{L^2}.
	\end{split}
\end{equation*}
$\bullet$ \underline{$L_{\max} \geqslant N_1^{\alpha+1}$}: We assume $L_{\max}=L$ and employ the linear Strichartz estimate via \eqref{eq:StriEmbed}
\begin{equation*}
	\begin{split}
	\text{LHS of } \eqref{eq:DualXsbD} &\lesssim \|\mathcal{F}^{-1}(f_{N_1,L_1})\|_{L^4} \|\mathcal{F}^{-1}(g_{N_2,L_2})\|_{L^4}\|h_{N,L}\|_{L^2} \\
	&\lesssim N_1^{\frac{2-\alpha}{8}} N_2^{\frac{2-\alpha}{8}} (L_1L_2)^{\frac{1}{2}} \|f_{N_1,L_1}\|_{L^2}\|g_{N_2,L_2}\|_{L^2}\|h_{N,L}\|_{L^2}\\
	&\lesssim N_1^{-\frac{3\alpha}{4}+} (L_1L_2)^{\frac{1}{2}}L^{\frac{1}{2}-} \|f_{N_1,L_1}\|_{L^2}\|g_{N_2,L_2}\|_{L^2}\|h_{N,L}\|_{L^2}.
	\end{split} 
\end{equation*}
In all the cases considered above, we can sum up the dyadic estimates in frequency and modulation for $\alpha>\frac{5}{2}$ owing to $C(N_1,N_2,N)$. This proves the estimate \eqref{eq:DualXsb}.
\end{proof}
\end{proposition}

\subsection{Proof of Theorem \ref{thm:SemilinearLWPPrecise} and Corollary \ref{thm:SemilinearGWP}} 
\label{subsection:ProofOfSemilinear}We give a short proof of Theorem \ref{thm:SemilinearLWPPrecise} by using Lemma \ref{lemma:XsbEnergyEst} and Proposition \ref{prop:XsbBilinear}. We first prove the result on a fixed time interval $[0,1]$ for small initial data. Thereafter, we argue by scaling and subcriticality that the solution also exists for large initial data on a time interval $[0,T]$ where $T = T(\|u_0\|_{H^{s,0}})$. With $\eta$ as before, we define $\Gamma$ as follows:
\begin{equation*}
	\Gamma(u)(t) = \eta (t)U_{\alpha}(t)u_0 + \eta(t)\int_0^tU_{\alpha}(t-s)(u\partial_xu)(s)ds.
\end{equation*}
We shall prove that $u$ is a fixed point of the map $\Gamma$ in a closed ball $\bar{B}_{R} \subset X^{s,b}_1$ of radius $R$ for initial data with sufficiently small norm. We first show that $\Gamma$ is well-defined. For $u \in \bar{B}_R$, using Lemma \ref{lemma:XsbTimeLocalisedFreeSol}, Lemma \ref{lemma:XsbEnergyEst} and Proposition \ref{prop:XsbBilinear}, we obtain
\begin{equation*}
	\begin{split}
	\|\Gamma (u)\|_{X^{s,b}_1} &\lesssim \|\eta(t)U_{\alpha}(t)u_0\|_{X^{s,b}_1} + \Big \| \eta(t)\int_0^tU_{\alpha}(t-s)(u\partial_x u)(s)ds \Big \|_{X^{s,b}_1} \\
	&\lesssim \|u_0\|_{H^{s,0}} + \|u\partial_x u\|_{X^{s,(b-1)}_1}\\ 
	&\leqslant C( \|u_0\|_{H^{s,0}}+ \|u\|_{X^{s,b}_1}^2).
\end{split}	
\end{equation*}
 If we choose the radius $R$  of the ball such that $C\|u_0\|_{H^{s,0}} = \frac{R}{2}$, then 
\begin{equation}
\label{eq:T1}
		\|\Gamma (u)\|_{X^{s,b}_1} \leqslant \frac{R}{2}+ CR^2 \leqslant R, \text{ if } C(2C\|u_0\|_{H^{s,0}}) <\frac{1}{2},
\end{equation}
which shows that $\Gamma$ is well-defined.
 To show that $\Gamma$ is a contraction, for $u, v \in \bar{B}_R$, we have, using Lemma \ref{lemma:XsbEnergyEst} and Proposition \ref{prop:XsbBilinear}
\begin{equation*}
	\begin{split}
		\| \Gamma (u_1) - \Gamma (u_2)\|_{X^{s,b}_1} &\lesssim \Big \| \eta (t)\int_0^t U_{\alpha}(t-s)(u_1 \partial_x u_1 - u_2 \partial_x u_2)(s)ds \Big\|_{X^{s,b}_1} \\
		&\lesssim \|\partial_x(u_1+u_2)(u_1-u_2)\|_{X^{s,(b-1)}_1}\\
		&\lesssim \|u_1+u_2\|_{X^{s,b}_1}\|u_1-u_2\|_{X^{s,b}_1}\\
		&\leqslant 2C_1 R \|u_1-u_2\|_{X^{s,b}_1},
		\end{split}
\end{equation*} 
Hence, $\Gamma$ becomes a contraction on $X^{s,b}_1$ if $\|u_0\|_{H^{s,0}}$ is such that
\begin{equation}\label{eq:T2}
2C_1 (2C\|u_0\|_{H^{s,0}}) <1.
\end{equation}
Using Banach's fixed point theorem, we conclude the existence of a unique solution to \eqref{eq:fKPI} in $X^{s,b}_1$ where the norm of the initial data is chosen as the minimum of that given by \eqref{eq:T1} and \eqref{eq:T2}. \\
Now suppose that $\|u_0\|_{H^{s,0}} \leqslant \varepsilon$ for $\varepsilon \ll 1$ and we have obtained a solution corresponding to this small initial data on the time interval $[0,1]$. 
For $\alpha>\frac{5}{2}$ and $s > \frac{5}{4}-\frac{\alpha}{2}$, from \eqref{eq:SubcriticalInitialDataFKP}, we observe that the anisotropic Sobolev regularity $(s,0)$ is subcritical. Thus any large initial data, say $\|u_0\|_{H^{s,0}} \geqslant \varepsilon$, can be scaled to small data via \eqref{eq:ScalingfKPInitialData}. We then invoke the above argument to obtain a unique solution to \eqref{eq:fKPI} on a time interval $[0,T]$ where $T$ depends only on the norm of the large initial data, $s$, and $\alpha$. The proof is complete.
$\hfill \Box$	

\medskip

Finally, we turn to the proof of Corollary \ref{thm:SemilinearGWP}.	
\begin{proof}[Proof~of~Corollary~\ref{thm:SemilinearGWP}]
In the following let $s \geq 0$. We consider two real-valued initial data $u_0^{(i)} \in H^{s,0}(\R^2)$. We rescale to $\| u_{0 \lambda}^{(i)} \|_{L^2} \leq \varepsilon \ll 1$. By the above this ensures local solutions $u_\lambda^{(i)} \in X^{0,b}$ to \eqref{eq:fKPI} and $\| u_\lambda^{(i)} \|_{X^{0,b}} \lesssim \varepsilon \ll 1$. We let $v_\lambda = u_{\lambda}^{(1)} - u_\lambda^{(2)}$. We obtain by the bilinear estimate:
\begin{equation}
\label{eq:DifferenceEstimate}
\| v_\lambda \|_{X^{s,b}} \lesssim \| v_\lambda(0) \|_{H^{s,0}} + \| v_\lambda \|_{X^{s,b}} ( \| u_\lambda^{(1)} \|_{X^{0,b}} + \| u_\lambda^{(2)} \|_{X^{0,b}}).
\end{equation}
This implies
\begin{equation*}
\| v_\lambda \|_{X^{s,b}} \leqslant C \| v_\lambda(0) \|_{H^{s,0}}.
\end{equation*}
However, by conservation of mass, $\| u_\lambda^{(1)}(1) \|_{L^2} + \| u_\lambda^{(2)}(1) \|_{L^2} \leq 2 \varepsilon \ll 1$. So we can solve \eqref{eq:fKPI} for $t \in [1,2]$ with $u_\lambda^{(i)} \in X^{0,b}_2$ for $i=1,2$, and moreover we can iterate \eqref{eq:DifferenceEstimate}. This gives
\begin{equation*}
\| v_\lambda(2) \|_{H^{s,0}} \leqslant C^2 \| v_\lambda(0) \|_{H^{s,0}},
\end{equation*}
and iterating the argument yields
\begin{equation*}
\| v_\lambda(t) \|_{H^{s,0}} \leqslant C_1 e^{C_2 t} \| v_\lambda(0) \|_{H^{s,0}}.
\end{equation*}
Hence, for $\| v_\lambda(0) \|_{H^{s,0}} \to 0$, we have $\sup_{t \in [0,T]} \| v_\lambda(t) \|_{H^{s,0}} \to 0$. By reversing the scaling, we obtain global well-posedness.
\end{proof}

\section*{Appendix: Calculation of the determinant}
Let $S_i$, $i=1,2$ be the hypersurface given by
\begin{equation*}
	S_i = \Big\{(\tau_i,\xi_i,\eta_i)\in \R\times\R \times \R: \tau_i=\xi_i|\xi_i|^{\alpha}+\frac{\eta_i^2}{\xi_i}\Big\}.
\end{equation*}
The normal to the hypersurface $S_i$ is given by
\begin{equation*}
	\mathbf{n}_i=\Big (1,-( (\alpha+1)|\xi_i|^{\alpha}-\frac{\eta_i^2}{\xi_i^2}), -\frac{2\eta_i}{\xi_i}  \Big).
\end{equation*}
Under the convolution constraints, the third hypersurface will have a normal vector given by
\begin{equation*}
	\mathbf{n}_3=\Big (1,-((\alpha+1)|\xi_1+\xi_2|^{\alpha}-\frac{(\eta_1+\eta_2)^2}{(\xi_1+\xi_2)^2}),- \frac{2(\eta_1+\eta_2)}{\xi_1+\xi_2}  \Big).
\end{equation*}
For the following arguments recall the dyadic localizations\footnote{We omit $'$ for the frequency variables to lighten the notation.}
\begin{equation}
\label{eq:DyadicLocalizationAppendix}
|\xi_1| \sim |\xi_1+\xi_2| \sim 1, \quad |\xi_2| \sim N_{\min}/N_{\max}, \quad \big| \frac{\eta_1}{\xi_1} - \frac{\eta_2}{\xi_2} \big| \sim 1.
\end{equation}

We compute the determinant of these normals. Let
\begin{equation*}
    B :=\begin{vmatrix}
			(\alpha+1)|\xi_1|^{\alpha} -\frac{\eta_1^2}{\xi_1^2} & (\alpha+1)|\xi_2|^{\alpha}-\frac{\eta_2^2}{\xi_2^2} & (\alpha+1)(|\xi_1+\xi_2|)^{\alpha} -\frac{(\eta_1+\eta_2)^2}{(\xi_1+\xi_2)^2}\\
			\frac{2\eta_1}{\xi_1} & \frac{2\eta_2}{\xi_2} & \frac{2(\eta_1+\eta_2)}{\xi_1+\xi_2}\\
			1&1&1
		\end{vmatrix}
\end{equation*}
We obtain by multilinearity for $\tilde{B} = \frac{\xi_1 \xi_2 (\xi_1 + \xi_2) B}{2}$:

\begin{footnotesize}
\begin{equation*}
	\begin{split}
		\tilde{B} &=
		\begin{vmatrix}
			(\alpha+1)|\xi_1|^{\alpha} -\frac{\eta_1^2}{\xi_1^2} & (\alpha+1)|\xi_2|^{\alpha}-\frac{\eta_2^2}{\xi_2^2} & (\alpha+1)(|\xi_1+\xi_2|)^{\alpha} -\frac{(\eta_1+\eta_2)^2}{(\xi_1+\xi_2)^2}\\
			\eta_1 \xi_2(\xi_1+\xi_2) & \eta_2\xi_1(\xi_1+\xi_2) & (\eta_1+\eta_2)\xi_1\xi_2\\
			1 & 1& 1
		\end{vmatrix}\\
		&= 
		\begin{vmatrix}
			(\alpha+1)|\xi_1|^{\alpha} -\frac{\eta_1^2}{\xi_1^2} & (\alpha+1)|\xi_2|^{\alpha}-\frac{\eta_2^2}{\xi_2^2} &(\alpha+1)(|\xi_1+\xi_2|^{\alpha} - |\xi_1|^{\alpha}) -\frac{(\eta_1+\eta_2)^2}{(\xi_1+\xi_2)^2} +\frac{\eta_1^2}{\xi_1^2}\\
			\eta_1\xi_2(\xi_1+\xi_2) & \eta_2\xi_1(\xi_1+\xi_2) & \xi_2(\eta_2\xi_1-\eta_1\xi_2)\\
			1 & 1&0
		\end{vmatrix}\\
		&=
		\begin{vmatrix}
			(\alpha+1)(|\xi_1|^{\alpha}-|\xi_2|^{\alpha}) -\frac{\eta_1^2}{\xi_1^2} +\frac{\eta_2^2}{\xi_2^2}& (\alpha+1)|\xi_2|^{\alpha}-\frac{\eta_2^2}{\xi_2^2} & (\alpha+1)(|\xi_1+\xi_2|^{\alpha} - |\xi_1|^{\alpha}) -\frac{(\eta_1+\eta_2)^2}{(\xi_1+\xi_2)^2} +\frac{\eta_1^2}{\xi_1^2}\\
			(\eta_1\xi_2-\eta_2\xi_1)(\xi_1+\xi_2) & \eta_2\xi_1(\xi_1+\xi_2) & \xi_2(\eta_2\xi_1-\eta_1\xi_2)\\
			0 & 1& 0
		\end{vmatrix}\\
		&=-(\eta_1\xi_2-\eta_2\xi_1)
		\begin{vmatrix}
			(\alpha+1)(|\xi_1|^{\alpha} -|\xi_2|^{\alpha}) +\frac{\eta_2^2}{\xi_2^2} -\frac{\eta_1^2}{\xi_1^2} & (\alpha+1)(|\xi_1+\xi_2|^{\alpha}- |\xi_1|^{\alpha}) +\frac{(\eta_1+\eta_2)^2}{(\xi_1+\xi_2)^2}\\
			\xi_1+\xi_2 & -\xi_2
		\end{vmatrix}\\
		&=-(\eta_1\xi_2-\eta_2\xi_1) \Big ((\alpha+1)(|\xi_1|^{\alpha}\xi_1 +|\xi_2|^{\alpha}\xi_2 - |\xi_1+\xi_2|^{\alpha}(\xi_1+\xi_2)) -\frac{(\eta_1\xi_2-\eta_2\xi_1)^2}{\xi_1\xi_2(\xi_1+\xi_2)}\Big).
	\end{split}
\end{equation*}
This gives
\begin{equation*}
B = -\frac{2 (\eta_1\xi_2-\eta_2\xi_1)}{\xi_1 \xi_2 (\xi_1+\xi_2)} \Big ((\alpha+1)(|\xi_1|^{\alpha}\xi_1 +|\xi_2|^{\alpha}\xi_2 - |\xi_1+\xi_2|^{\alpha}(\xi_1+\xi_2)) -\frac{(\eta_1\xi_2-\eta_2\xi_1)^2}{\xi_1\xi_2(\xi_1+\xi_2)}\Big).
\end{equation*}
\end{footnotesize}

From \eqref{eq:DyadicLocalizationAppendix}, we have that the first factor is
\begin{equation*}
    \Big| \frac{\eta_1 \xi_2 - \eta_2 \xi_1}{\xi_1 \xi_2 (\xi_1 + \xi_2)} \Big| \sim 1.
\end{equation*}

For the second factor observe that
\begin{equation*}
\text{sgn}( (\alpha+1)(|\xi_1|^{\alpha}\xi_1 +|\xi_2|^{\alpha}\xi_2 - |\xi_1+\xi_2|^{\alpha}(\xi_1+\xi_2))) = \text{sgn}( -\frac{(\eta_1\xi_2-\eta_2\xi_1)^2}{\xi_1\xi_2(\xi_1+\xi_2)}).
\end{equation*}
Moreover, by the resonance condition it holds that
\begin{equation*}
\big| (\alpha+1)(|\xi_1|^{\alpha}\xi_1 +|\xi_2|^{\alpha}\xi_2 - |\xi_1+\xi_2|^{\alpha}(\xi_1+\xi_2)) \big| \sim \big| \frac{(\eta_1\xi_2-\eta_2\xi_1)^2}{\xi_1\xi_2(\xi_1+\xi_2)} \big|.
\end{equation*}
The first factor is estimated by the mean-value theorem:
\begin{equation*}
\big| (\alpha+1)(|\xi_1|^{\alpha}\xi_1 +|\xi_2|^{\alpha}\xi_2 - |\xi_1+\xi_2|^{\alpha}(\xi_1+\xi_2)) \big| \sim |\xi_2| \sim \frac{N_{\min}}{N_{\max}}.
\end{equation*}

Hence, the size of this determinant becomes
\begin{equation*}
	B \sim \frac{N_{\min}}{N_{\max}}
\end{equation*}
in the resonant case. $\hfill \Box$

\section*{Acknowledgements}
A.S. gratefully acknowledges financial support by IRTG 2235 during the initial stage of the project.
R.S. acknowledges financial support by the German Research Foundation (DFG) -- Project-Id 258734477 -- SFB 1173.  The second author would like to thank the Institut de Mathématique d'Orsay for kind hospitality, where part of this work was carried out, and  Jean-Claude Saut for explaining the derivation of KP-I equations. 


\bibliographystyle{abbrv}

	\end{document}